\documentclass[12pt]{article}
\usepackage{graphicx}
\usepackage{color}

\usepackage{amssymb}
\usepackage{amsmath}
\usepackage{amsfonts}
\usepackage{amsthm}

\usepackage{epsfig}

\newcommand{\R}{\ensuremath{\mathbb{R}}}

\newcommand{\integer}{\ensuremath{\mathbb{Z}}}
\newcommand{\Z}{\ensuremath{\integer}}

\newcommand{\N}{\ensuremath{\mathbb{N}}}

  \newtheorem{thm}{Theorem}[section]
  \newtheorem{cor}[thm]{Corollary}
  \newtheorem{lemma}[thm]{Lemma}
  \newtheorem{prop}[thm]{Proposition}
  \newtheorem{defn}[thm]{Definition}
  \newtheorem{remark}[thm]{Remark}
  \newtheorem{hyp}[thm]{Hypothesis}

\begin{document}
%-----------------------------------

\title{Parametrizing nilpotent orbits in $p$-adic symmetric spaces using Bruhat-Tits theory}

\author{Ricardo Portilla}

\maketitle

\begin{abstract}
    \centering
    \begin{minipage}{0.65\textwidth}
        Let $k$ be a field with a nontrivial discrete valuation which is complete and has perfect residue field. Let $G$ be the group of $k$-rational points of a reductive, linear algebraic group $\textbf{G}$ equipped with an involution $\theta$ defined over $k.$ Let $\mathfrak{p}$ denote the $(-1)$-eigenspace in the decomposition of the Lie algebra of $G$ under the differential $d\theta.$ If $\textbf{H}$ is a subgroup of $\textbf{G}^{\theta}$, the set of $\theta$-fixed points, which contains the connected component of $\textbf{G}^{\theta},$ then $H=\textbf{H}(k)$ acts on $\mathfrak{p}$, which we treat as a symmetric space. Let $r \in \mathbb{R}.$ Under mild restrictions on $\textbf{G}$ and $k,$ the set of nilpotent $H$-orbits in $\mathfrak{p}$ is parametrized by equivalence classes of noticed Moy-Prasad cosets of depth $r$ which lie in $\mathfrak{p}.$    \end{minipage}
\end{abstract}

\section{Introduction}

Let $k$ be a field equipped with a nontrivial discrete valuation, and let $\textbf{G}$ be a reductive, linear algebraic group defined over $k.$ Let $\mathfrak{g}$ be the vector space of $k$-rational points of Lie($\textbf{G}),$ and let $G=\textbf{G}(k).$ Consider the adjoint action of $G$ on $\mathfrak{g}.$ In [\cite{d}], DeBacker gave a uniform parametrization of the set of nilpotent $G$-orbits in $\mathfrak{g}$ using Bruhat-Tits theory. It is the purpose of this paper to establish a parametrization of nilpotent orbits in the context of $p$-adic symmetric spaces using Bruhat-Tits theory.
	
	More precisely, let $\theta: \textbf{G} \rightarrow \textbf{G}$ be a nontrivial involution defined over $k.$ Under $d\theta,$ the differential of $\theta,$ $\mathfrak{g}$ decomposes into $(+1)$ and $(-1)$-eigenspaces, which we denote $\mathfrak{h}$ and $\mathfrak{p},$ respectively. Let $\textbf{H}$ denote a subgroup with $(\textbf{G}^{\theta})^{\circ} \subset \textbf{H} \subset \textbf{G}^{\theta}$ such that $\textbf{H}$ is defined over $k.$ Vust (in [\cite{vust}]) and Prasad-Yu (in [\cite{prasad-yu}, Theorem 2.4]) showed that $\textbf{H}^{\circ}$ is reductive whenever $\textbf{G}$ is reductive. Thus, we may consider the Bruhat-Tits building of $H=\textbf{H}(k).$ (Note that $H$ preserves $\mathfrak{p}$ under the adjoint action.) Under the assumption that the residual characteristic of $k$ is not two, Prasad and Yu showed (in [\cite{prasad-yu}, Theorem 1.9]) that we may identify $\mathcal{B}(H)$ with the set of $\theta$-fixed points in $\mathcal{B}(G).$ This result was also proved in the case where $\textbf{H}$ is a classical group arising from an involution (as well as spherical buildings) in [\cite{kim}, Theorem 6.7.3]. Using this identification, it makes sense to consider elements of $\mathcal{B}(H)$ as elements lying in $\mathcal{B}(G).$

	In [\cite{moy-prasad}], for each $r \in \mathbb{R},$ Moy and Prasad associate a lattice $\mathfrak{g}_{x,r}$ to each point $x \in \mathcal{B}(G).$ If $r=0$ and $x \in \mathcal{B}(H),$ then $\theta$ acts on each Lie algebra $V_{x,0}:=\mathfrak{g}_{x,0}/\mathfrak{g}_{x,0^+},$ which then gives a decomposition $V_{x,0} = V_{x,0}^+ \oplus V_{x,0}^-$ into $(+1)$ and $(-1)$-eigenspaces. We can then define an action of $H$ on the set of degenerate cosets in $V_{x,0}^-,$ meaning those cosets which contain a nilpotent element in $\mathfrak{g}.$ Thus, in the case when $r=0,$ this paper provides a parametrization of nilpotent $H$-orbits in $\mathfrak{p}$ in terms of equivalence classes of pairs $(F,e),$ where $F$ is the set of $\theta$-fixed points of a $\theta$-stable facet of $\mathcal{B}(G),$ and $e$ is a degenerate coset in $V_{x,0}^-.$ 
	
	Ultimately, we will be interested in doing harmonic analysis on $G/H,$ which is referred to as a $p$-adic symmetric space. For $p$-adic symmetric spaces, spherical characters play the role of characters of irreducible, admissible representations of $G.$ In [\cite{rr}, Theorem 7.11], Rader-Rallis gave a local expansion for spherical characters of irreducible class one representations of $G$ (see [\cite{rr}, Section 1]) in a neighborhood about the identity in terms of $H$-invariant distributions supported on $\mathcal{N} \cap \mathfrak{p},$ the set of nilpotent elements in $\mathfrak{p}.$ We can (and do) identify the $k$-rational points of the tangent space of $\textbf{G}/\textbf{H}$ at the identity with $\mathfrak{p},$ and this is where the nilpotent $H$-orbits will live. A motivation for describing a parametrization of nilpotent $H$-orbits in $\mathfrak{p}$ is to establish a homogeneity result about the spherical character of an irreducible class one representation of $G.$ The analogous homogeneity result for characters of irreducible, admissible representations of $G,$ which occurs in harmonic analysis on $G,$ was given in [\cite{d2}, Theorem 3.5.2].
		
	In this paper, we focus on a particular type of facet which encodes the $H$-orbit structure of $\mathcal{N} \cap \mathfrak{p}.$ Suppose $r \in \mathbb{R}.$ As in [\cite{d}, Section 3.1], we say that $x,y \in \mathcal{B}(G)$ belong to the same \emph{generalized} $r$-\emph{facet} $F'^* \subset \mathcal{B}(G)$ if $\mathfrak{g}_{x,r} = \mathfrak{g}_{y,r} \textup{ and } \mathfrak{g}_{x,r^+} = \mathfrak{g}_{y,r^+}.$ We will only consider the $\theta$-fixed points of $\theta$-stable generalized $r$-facets; we call these \emph{generalized} $(r,\theta)$-\emph{facets}. The generalized $(r,\theta)$-facets form a partition of the Bruhat-Tits building of $H.$ 
	
	In Section 4.4, for $x \in F_{\theta}^*,$ a generalized $(r,\theta)$-facet, we attach an $\mathfrak{f}$-vector space $V_{F_{\theta}^*}:=\mathfrak{g}_{x,r} / \mathfrak{g}_{x,r^+}$ to $F_{\theta}^*.$ We call a coset $e$ lying in $V_{F_{\theta}^*}$ a Moy-Prasad coset. Such a coset is said to be degenerate if it intersects $\mathcal{N},$ the set of nilpotent elements in $\mathfrak{g},$ nontrivially. At this point, we restrict our attention to degenerate cosets in the $(-1)$-eigenspace of $V_{F_{\theta}^*},$ denoted $V_{F_{\theta}^*}^-,$ and let $I_r^n$ denote the set of pairs of the form $(F_{\theta}^*,e),$ with $F_{\theta}^*$ a generalized $(r,\theta$)-facet and $e \in V_{F_{\theta}^*}^-$ a degenerate coset. In Section 4.5, we define a natural equivalence relation $\sim$ on $I_r^n.$ To each pair $(F_{\theta}^*,e) \in I_r^n,$ with some restrictions on $\textbf{G}$ and $k$ described in Section 5 (which are present in the group case in [\cite{d}]), we associate a nilpotent $H$-orbit $\mathcal{O}_{\theta}(F_{\theta}^*,e)$ in $\mathfrak{p},$ which (in Section 6) is described as the unique nilpotent $H$-orbit in $\mathfrak{p}$ of minimal dimension intersecting $e$ nontrivially. Let $\mathcal{O}_{\theta}(0)$ denote the set of nilpotent $H$-orbits in $\mathfrak{p}.$ Upon restricting to a natural subset $I_r^d$ (the \emph{noticed} orbits of Definition 6.18) of pairs in $I_r^n,$ we prove the following theorem:
	
\begin{thm} There is a bijective correspondence between $I_r^d/\sim$ and $\mathcal{O}_{\theta}(0)$ given by the map that sends $(F_{\theta}^*,e)$ to $\mathcal{O}_{\theta}(F_{\theta}^*,e).$\end{thm}
	
	Any reductive, linear algebraic group $\textbf{J}$ defined over a local field can be thought of as a symmetric space in the following way: let $\textbf{G} = \textbf{J} \times \textbf{J}$ and define an involution $\theta$ by $(x,y) \mapsto (y,x).$ Then, the diagonal $\textbf{H}:= \{(x,x) \mid x \in \textbf{J} \}$ occurs as the set of $\theta$-fixed points, and we may identify $\textbf{G} / \textbf{H}$ with $\textbf{J}.$ This is often referred to as the group (or diagonal) case. Note that $\mathfrak{p},$ as defined earlier, may be identified with Lie$(\textbf{J}).$ In the symmetric space setting, since we are interested in $H$-orbits, it is convenient to use the building of $H$ to describe a parametrization of nilpotent $H$-orbits in $\mathfrak{p}.$ More specifically, using arguments similar to those in [\cite{d}], we are able to lift the results from the group case to the symmetric space case at each step.
	
	For our purposes, we will primarily be concerned with the case where $\textbf{H}$ is isotropic over $k.$ This contrasts with the case considered when studying symmetric spaces of real Lie groups. More specifically, recall that $H$ preserves $\mathfrak{p}$ under the adjoint action. We say that $X \in \mathfrak{p}$ is nilpotent whenever there exists a one-parameter subgroup $\lambda \in \textbf{X}_*^{k}(\textbf{G})$ such that
	$$ \lim_{t \rightarrow 0} {}^{\lambda(t)}X = 0. $$ 
It will be demonstrated in Remark \ref{kempf} that, under mild restrictions on the characteristic of $k,$ we may assume $\lambda$ lies in $\textbf{X}_*^k(\textbf{H}).$

\subsection{Example}
	
	We demonstrate the parametrization in the case $r=0$ in the following example. Let $k=\mathbb{Q}_p,$ with $p \neq 2.$ Let $\textbf{G} = \textbf{SL}_3,$ and consider the involution $\theta : \textbf{SL}_3 \rightarrow \textbf{SL}_3$ defined by $A \mapsto J(A^t)^{-1}J,$ where $ J= \left(\begin{array}{ccc}
0 & 0 & 1 \\
0 & 1 & 0   \\
1 & 0 & 0 
\end{array}\right). $ Under $d\theta,$ the Lie algebra $\mathfrak{sl}_3(k)$ decomposes as

$$ \mathfrak{sl}_3(k) =  \left\{ \left(\begin{array}{ccc}
a & b & 0 \\
c & 0 & -b   \\
0 & -c & -a 
\end{array}\right) \right\} \oplus \left\{ \left(\begin{array}{ccc}
x & y & s \\
z & -2x & y   \\
u & z & x 
\end{array}\right) \right \}. $$
As further explained in Appendix A, representatives for the six nilpotent $H$-orbits in $\mathfrak{p}$ are $ \left(\begin{array}{ccc}
0 & 0 & 0 \\
0 & 0 & 0   \\
0 & 0 & 0 
\end{array}\right),   \left(\begin{array}{ccc}
0 & 1 & 0 \\
0 & 0 & 1  \\
0 & 0 & 0 
\end{array}\right) ,$ and  $\left\{ \left(\begin{array}{ccc}
0 & 0 & z \\
0 & 0 & 0   \\
0 & 0 & 0 
\end{array}\right) \mid \overline{z} \in \mathbb{Q}_p^{\times}/(\mathbb{Q}_p^{\times})^2 \right\}.$ The diagonal torus in the set of $\theta$-fixed points, $\textbf{H} = \textbf{PGL}_2,$ is a maximal $k$-split torus $\textbf{T}$ which lies in the diagonal maximal $k$-split torus $\textbf{T}'$ of $\textbf{SL}_3.$ If 

$$ t =  \left(\begin{array}{ccc}
a & 0 & 0 \\
0 & b & 0   \\
0 & 0 & (ab)^{-1} 
\end{array}\right) $$
is an element of $\textbf{T}',$ define $\alpha$ and $\beta$ by $\alpha(t) = ab^{-1}$ and $\beta(t) = ab^2.$ Then, $\{\alpha, \beta\}$ is a choice of simple roots of $\textbf{T}'$ in $\textbf{G}$ with respect to $k.$ We let $\check{\alpha}$ and $\check{\beta}$ denote the associated co-roots, respectively.

	Using [\cite{prasad-yu}, Theorem 1.9], we are able to identify the building of $H$ with the set of $\theta$-fixed points in $\mathcal{B}(G).$ Thus, we identify the apartment corresponding to the diagonal torus in $\textbf{PGL}_2$ with an affine subspace of the apartment corresponding to the diagonal torus in $\textbf{SL}_3.$ 
	
	In order to provide a parameterization of the nilpotent $PGL_2$-orbits in $\mathfrak{p},$ we restrict our attention to subsets of $\mathcal{B}(H)$ which arise naturally by considering $\theta$-stable facets of $\mathcal{B}(G).$ We call a subset $F$ in an apartment $\mathcal{A} \subset \mathcal{B}(H)$ a $\theta$-facet if $F$ is the set of $\theta$-fixed points of a $\theta$-stable facet $F'$ in some apartment of $\mathcal{B}(G).$ 
	
	The corresponding apartments are represented in Figure 1, along with the $\theta$-facets arising from the apartment associated to $\textbf{T}.$

\begin{figure}[ht!]
	\centering
	\resizebox{8.8cm}{!}{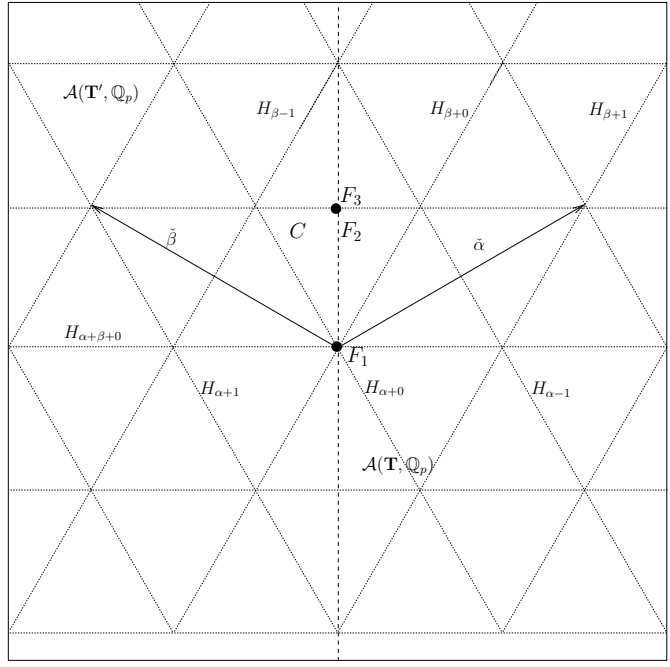}
	\caption{Affine apartments of $G=\textbf{SL}_3(\mathbb{Q}_p)$ and $H=\textbf{PGL}_2(\mathbb{Q}_p)$ (dotted line represents $\mathcal{A}(\textbf{T}, \mathbb{Q}_p)$)}
\end{figure}

The $\theta$-facets in Figure 1 labelled $F_1, F_2,$ and $F_3$ are those which arise as the fixed points of facets in the closure of a fixed alcove $C$ of $\mathcal{A}(\textbf{T}', \mathbb{Q}_p).$ In particular, $F_1$ is the vertex at the base of $\overline{C}$, $F_2$ is the $\theta$-facet arising from $C$ itself, and $F_3$ is the point in the closure of the alcove whose lift is the segment at the top of $\overline{C}.$
	
	If $F$ is a $\theta$-facet containing some point $x \in \mathcal{B}(H),$ we note that $\theta$ induces a map, which we denote $d\theta_F,$ on the Lie algebra $V_F:=\mathfrak{g}_{x}/\mathfrak{g}_{x}^+.$  In this way, we may consider the decomposition of $V_F$ under $d\theta_F,$ and examine nilpotent $H$-orbits in the $(-1)$-eigenspace of $V_F.$ The corresponding Lie algebras associated to each of these $\theta$-facets are listed below:

$$ V_{F_1}=  \left\{ \left(\begin{array}{ccc}
\mathbb{Z}_p / p\mathbb{Z}_p & \mathbb{Z}_p/ p\mathbb{Z}_p & \mathbb{Z}_p/p\mathbb{Z}_p \\
\mathbb{Z}_p/p\mathbb{Z}_p & \mathbb{Z}_p/p\mathbb{Z}_p & \mathbb{Z}_p/p\mathbb{Z}_p \\
\mathbb{Z}_p/p\mathbb{Z}_p & \mathbb{Z}_p/p\mathbb{Z}_p & \mathbb{Z}_p/p\mathbb{Z}_p
\end{array}\right) \right\} =  \left\{ \left(\begin{array}{ccc}
a & b & 0 \\
c & 0 & -b   \\
0 & -c & -a 
\end{array}\right) \right\} \oplus  \left\{ \left(\begin{array}{ccc}
x & y & s \\
z & -2x & y   \\
u & z & x 
\end{array}\right) \right\},$$ 

$$ V_{F_2} =  \left\{ \left(\begin{array}{ccc}
\mathbb{Z}_p / p\mathbb{Z}_p & 0 & 0 \\
0 & \mathbb{Z}_p / p \mathbb{Z}_p & 0   \\
0 & 0 & \mathbb{Z}_p / p \mathbb{Z}_p 
\end{array}\right) \right\} =  \left\{ \left(\begin{array}{ccc}
a & 0 & 0 \\
0 & 0 & 0   \\
0 & 0 & -a 
\end{array}\right) \right\} \oplus \left\{  \left(\begin{array}{ccc}
x & 0 & 0 \\
0 & -2x & 0   \\
0 & 0 & x 
\end{array}\right) \right\}, $$
and 

$$ V_{F_3} = \left\{ \left(\begin{array}{ccc}
\mathbb{Z}_p / p\mathbb{Z}_p & 0 & p^{-1}\mathbb{Z}_p/\mathbb{Z}_p \\
0 & \mathbb{Z}_p/ p\mathbb{Z}_p & 0   \\
p\mathbb{Z}_p / p^2\mathbb{Z}_p & 0 & \mathbb{Z}_p / p \mathbb{Z}_p 
\end{array}\right) \right\}=  \left\{ \left(\begin{array}{ccc}
a & 0 & 0 \\
0 & 0 & 0   \\
0 & 0 & -a 
\end{array}\right) \right\} \oplus \left\{  \left(\begin{array}{ccc}
x & 0 & p^{-1}s  \\
0 & -2x & 0   \\
p u & 0 & x 
\end{array}\right) \right\}, $$
with all lowercase entries being representatives in $\mathbb{Z}_p.$ 
	
At this point, we would like to match up nilpotent $H$-orbits with nilpotent orbits arising from each of the above $\mathfrak{f}$-Lie algebras. In order to obtain a bijection, however, we must restrict ourselves to elements $e \in V_F^-$ whose centralizer (in $V_F^+$) does not contain certain noncentral (meaning elements in $V_F^+$ which do not belong to the center of $V_F$) semisimple elements which are fixed by $\theta.$ This may be thought of as a restriction on the type of Levi subalgebra which is allowed to contain $e$ and thus resembles the distinguished condition found in [\cite{d}, Remark 5.5.2]. We call such nilpotent elements \emph{noticed}.

	The noticed nilpotent $H$-orbits in $V_{F_1}^-$ have representatives of the form $ \left\{ \left(\begin{array}{ccc}
0 & 1 & 0 \\
0 & 0 & 1   \\
0 & 0 & 0 
\end{array}\right) \right\}$ and $ \left\{ \left(\begin{array}{ccc}
0 & 0 & s \\
0 & 0 & 0   \\
0 & 0 & 0 
\end{array}\right) \mid \overline{s} \in \mathbb{F}_p^{\times}/ (\mathbb{F}_p^{\times})^2 \right\}.$ The only noticed nilpotent $H$-orbit in $V_{F_2}^-$ is the trivial orbit. There are two more noticed nilpotent $H$-orbits lying in $V_{F_3}^-$ whose lifts modulo $p^{-1}$ correspond to the two representatives of square classes in $\mathbb{F}_p^{\times}/(\mathbb{F}_p^{\times})^2.$ Upon taking lifts, these six orbits clearly match up with the six nilpotent $H$-orbits in $\mathfrak{p}$ discussed above.

\vspace{.25in} 
\textbf{Acknowledgements:} I would like to thank my advisor, Stephen DeBacker, for all of his help, support, and extensive comments on earlier drafts of this paper. Without his guidance and constant encouragement, this work would not have been possible.

\section{Preliminaries}

\subsection{Algebraic groups, involutions, and associated notation} 

Let $k$ be a field with a nontrivial discrete valuation $\nu,$ and let $K$ be a maximal unramified extension of $k.$ Let $R$ (resp.$R_K)$ denote the ring of integers in $k$ (resp.$K),$ and let $\mathfrak{f}$ (resp.$\mathfrak{F})$ denote the residue field of $k$ (resp.$K$). We assume $k$ is complete and that $\mathfrak{f}$ is perfect. Let $\textbf{G}$ be a reductive, linear algebraic group defined over $k,$ and fix an involution $\theta$ of $\textbf{G}$ which is defined over $k.$ Let $\textbf{G}^{\circ}$ denote the connected component of $\textbf{G}.$

	We fix a uniformizer $\varpi$ of $k,$ with respect to $\nu,$ and let $L$ denote the minimal Galois extension of $K$ over which $\textbf{G}^{\circ}$ splits. Let $\ell = [L:K].$ If $\nu$ also denotes the extension of $\nu$ to $L,$ then we normalize $\nu$ so that $\nu(L^{\times}) = \mathbb{Z}.$
	
	If $k'$ is any field, we let $\overline{k'}$ denote an algebraic closure of $k'.$ Suppose $\textbf{C}$ is a linear algebraic group defined over $k'.$ We will identify $\textbf{C}$ with the $\overline{k'}$-points of $\textbf{C}.$ If $\sigma$ is an involution of $\textbf{C}$ defined over $k',$ we will almost always let $\textbf{C}^{\sigma}$ denote the set $\{ x \in \textbf{C} \mid \sigma(x) = x \};$ in the case that $\textbf{C}$ is a $\sigma$-stable torus, we will let $\textbf{C}^{\sigma}$ be the connected component of this set. Lastly, if $C$ is any group, we let $[C,C]$ denote its derived subgroup, and if $L$ is any Lie algebra, we let $[L,L]$ denote its derived subalgebra.
	
	Let $\textbf{H}$ be a subgroup of $\textbf{G}$ with $(\textbf{G}^{\theta})^{\circ} \subset \textbf{H} \subset \textbf{G}^{\theta}$ such that $\textbf{H}$ is defined over $k.$ By the second paragraph in [\cite{vust}, Section 1.0], $\textbf{H}^{\circ}$ is reductive. We let $G$ denote the group of $k$-rational points of $\textbf{G}$ and similarly let $H$ denote the group of $k$-rational points of $\textbf{H}.$ The involution $\theta$ induces an involution, which we denote $d\theta,$ on the Lie algebra, \mbox{\boldmath$\mathfrak{g}$}:=Lie$(\textbf{G}),$ of $\textbf{G}.$ Under $d\theta,$ \mbox{\boldmath$\mathfrak{g}$} decomposes as \mbox{\boldmath$\mathfrak{g}$} = \mbox{\boldmath$\mathfrak{h}$} $\oplus$ \mbox{\boldmath$\mathfrak{p}$} where \mbox{\boldmath$\mathfrak{h}$} is the Lie algebra of \textbf{H} and \mbox{\boldmath$\mathfrak{p}$} is the $(-1)$-eigenspace of \mbox{\boldmath$\mathfrak{g}$}. We let $\mathfrak{g} =$ \mbox{\boldmath$\mathfrak{g}$}($k$), $\mathfrak{h}$ = \mbox{\boldmath$\mathfrak{h}$}$(k)$ and $\mathfrak{p} =$ \mbox{\boldmath$\mathfrak{p}$}$(k)$ denote the vector spaces of $k$-rational points of \mbox{\boldmath${\mathfrak{g}}$}, \mbox{\boldmath$\mathfrak{h}$}, and \mbox{\boldmath$\mathfrak{p}$}. If $V$ is a $k'$-vector space on which some $k'$-involution $\sigma$ acts, we let $V^+$ and $V^-$ denote the sets $\{ X \in V \mid \sigma(X) = X \}$ and $\{X \in V \mid \sigma(X) = -X \},$ respectively. 
	
	We follow some notational conventions (also found in [\cite{d}, 2.1]) so that when we refer to a \emph{Levi subgroup} of $\textbf{G}$ (resp. $\textbf{H}),$ we mean a Levi subgroup of $\textbf{G}^{\circ}$ (resp. $\textbf{H}^{\circ}$). We apply the same terminology to tori and parabolic subgroups.
	
	Let Ad denote the adjoint action of $G$ on $\mathfrak{g}.$ For $g \in G$ and $X \in \mathfrak{g},$ let ${}^{g}X = \textup{Ad}(g)(X).$ Suppose $\textbf{L}$ is a linear algebraic group defined over $k'$ acting on its Lie algebra \mbox{\boldmath$\mathfrak{l}$} via the adjoint action. If $g \in \textbf{L},$ we will let Int$(g)$ denote conjugation by $g.$ Let $L \supset J$ and $\mathfrak{j}$ be subsets of $\textbf{L}$ and \mbox{\boldmath$\mathfrak{l}$} respectively. Then, let $C_{L}(\mathfrak{j}) = \{ g \in L \mid {}^{g}X = X \textup{ for all } X \in \mathfrak{j} \}.$ Similarly, we let $N_{L}(\mathfrak{j}) = \{ g \in L \mid {}^{g}\mathfrak{j} = \mathfrak{j} \}.$

	Although we restrict ourselves to discussing nilpotent elements lying in $\mathfrak{p},$ our definition of nilpotence will be as in [\cite{d}]. In particular, we call $X \in \mathfrak{p}$ nilpotent provided there exists some $\lambda \in \textbf{X}_{*}^{k}(\textbf{G})$ such that $\lim_{t \rightarrow 0} {}^{\lambda(t)}X = 0.$ We let $\mathcal{N}$ denote the set of nilpotent elements in $\mathfrak{g}$ and define $\mathcal{N}^-:=\mathcal{N} \cap \mathfrak{p}.$ We also let $\mathcal{U}$ denote the set of unipotent elements in $G.$
			
	 If the residue field $\mathfrak{f}$ has positive characteristic, we denote the characteristic of $\mathfrak{f}$ by $p.$ If the residue field has characteristic zero, we let $p=\infty.$ 

\subsection{The Bruhat-Tits building, apartments, and $\theta$-fixed points} 
	 
	We let $\mathcal{B}(G)$ denote the (enlarged) Bruhat-Tits building of $\textbf{G}^{\circ}(k),$ and, similarly, let $\mathcal{B}(H)$ denote the (enlarged) Bruhat-Tits building of $\textbf{H}^{\circ}(k).$ Unless otherwise stated, the symbol $\mathcal{B}(G)$ (resp. $\mathcal{B}(H)$) will always refer to the enlarged Bruhat-Tits building; that is, it takes the center of $\textbf{G}^{\circ}(k)$ (resp. $\textbf{H}^{\circ}(k)$) into account. We note that since $K/k$ is a maximal unramified extension, $\mathfrak{F}$ is an algebraic closure of $\mathfrak{f},$ and $\mathcal{B}(G)$ can be identified with the Gal$(K/k)$-fixed points of $\mathcal{B}(\textbf{G},K),$ the Bruhat-Tits building of $\textbf{G}^{\circ}(K).$ 
		
	If $\textbf{S}$ (resp. $\textbf{S}'$) is a maximal $k$-split torus of $\textbf{H}$ (resp. $\textbf{G}$), we will let $\mathcal{A}(\textbf{S}, k)$ (resp. $\mathcal{A}(\textbf{S}', k)$) denote the associated apartment in $\mathcal{B}(H)$ (resp. $\mathcal{B}(G)$). If $\mathcal{A}$ is an apartment in $\mathcal{B}(H),$ and $\Omega$ is a subset of $\mathcal{A},$ we let $A(\Omega,\mathcal{A})$ denote the smallest affine subspace of $\mathcal{A}$ containing $\Omega.$
	We let dist: $\mathcal{B}(G) \times \mathcal{B}(G) \rightarrow \mathbb{R}_+$ denote a nontrivial $G$-invariant distance function as described in [\cite{tits}, 2.3]. For $x,y \in \mathcal{B}(G),$ let $[x,y]$ denote the geodesic in $\mathcal{B}(G)$ from $x$ to $y$ with respect to dist. Let $(x,y]$ denote $[x,y] \backslash \{x \}.$
	
	For $x \in \mathcal{B}(\textbf{G},K),$ we let $\textbf{G}(K)_x$ and $\textbf{G}(K)_x^+$ denote the parahoric subgroup associated to $x$ and its pro-unipotent radical, respectively. The groups $\textbf{G}(K)_x$ and $\textbf{G}(K)_x^+$ only depend on the facet in $\mathcal{B}(\textbf{G},K)$ containing $x.$ Therefore, if $F \subset \mathcal{B}(\textbf{G},K)$ is the facet containing $x,$ we define $\textbf{G}(K)_F$ and $\textbf{G}(K)_F^+$ to be $\textbf{G}(K)_x$ and $\textbf{G}(K)_x^+,$ respectively. The quotient $\textbf{G}(K)_F / \textbf{G}(K)_F^+$ is the group of $\mathfrak{F}$-points of a connected, reductive group $\textsf{G}_F$ defined over $\mathfrak{f}.$ 	
	
	If $x \in \mathcal{B}(G),$ we denote the parahoric associated to $x$ and its pro-unipotent radical by $G_x$ and $G_{x}^+,$ respectively. We recall that these subgroups are obtained as the sets of Gal$(K/k)$-fixed points of parahorics defined over the maximal unramified extension. That is, we have $G_F = \textbf{G}(K)_F^{\textup{Gal}(K/k)}$ and $G_F^+ = (\textbf{G}(K)_F^+)^{\textup{Gal}(K/k)}.$ The quotient $G_x / G_x^+$ coincides with the group of $\mathfrak{f}$-rational points of the connected, reductive group $\textsf{G}_x$ defined over $\mathfrak{f}.$ Moreover, we have $\textsf{G}_x(\mathfrak{f}) = \textsf{G}_x^{\textup{Gal}(\overline{\mathfrak{f}}/\mathfrak{f})}.$
		
		If $\textbf{S}$ (resp.$\textbf{S}')$ is a maximal $k$-split torus in $\textbf{H}$ (resp.$\textbf{G}),$ we let $\Phi = \Phi(\textbf{S},k)$ (resp.$\Phi(\textbf{S}', k))$ denote the set of roots of $\textbf{S}$ (resp.$\textbf{S}')$ in $\textbf{H}$ (resp.$\textbf{G})$ with respect to $k.$ If $\mathcal{A}$ (resp.$\mathcal{A}'$) is the apartment corresponding to $\textbf{S}$ (resp.$\textbf{S}'),$ let $\Psi = \Psi(\mathcal{A})$ (resp.$\Psi(\mathcal{A}')$) denote the set of affine roots of $\textbf{H}$ (resp.$\textbf{G}$) with respect to $\textbf{S},$ $k$ and $\nu$ (resp. $\textbf{S}', k$ and $\nu).$ If $\psi \in \Psi$ is an affine root, we let $\dot{\psi} \in \Phi$ denote the gradient of $\psi.$ Whenever $\psi$ is an affine root of $\textbf{H}$ (resp. $\textbf{G}$) with respect to $\textbf{S}$ (resp. $\textbf{S}')$ and $\Omega$ is a subset of the apartment associated to $\textbf{S}$ (resp. $\textbf{S}'),$ we let res$_{\Omega}\psi$ denote the restriction of $\psi$ to $\Omega.$
	 
	 Throughout this paper, we will assume that $p \neq 2.$ Under this assumption, Prasad and Yu proved that $\mathcal{B}(H) = \mathcal{B}(G^{\theta})$ can be identified with $\mathcal{B}(G)^{\theta}.$ We will abuse notation and let $\theta$ also denote the induced map on the building of $G.$ Whenever $\Omega'$ is a subset of $\mathcal{B}(G),$ we will define $\Omega'^{\theta}:= \{ x \in \Omega' \mid \theta(x) = x\}.$
	 
\subsection{The Moy-Prasad filtrations}
	 
	 Let $x \in \mathcal{B}(H)$ and $r \in \mathbb{R}.$ We let $\mathfrak{g}_{x,r}$ denote the Moy-Prasad filtration of depth $r$ as defined in [\cite{moy-prasad}, Section 3.2]. As we will show in Section 4, the filtration lattices $\mathfrak{g}_{x,r}$ and $\mathfrak{g}_{x,r^+}$ are $\theta$-stable and thus induce an $\mathfrak{f}$-involution of the $\mathfrak{f}$-vector space $\mathfrak{g}_{x,r}/\mathfrak{g}_{x,r^+}.$ We will denote this involution $d\theta_x.$ 
	 
	 Let $\textbf{S}$ be a maximal $k$-split torus in $\textbf{H},$ and suppose $\textbf{T}$ is a maximal $K$-split $k$-torus in $\textbf{H}$ containing $\textbf{S}.$ By [\cite{prasad-yu}, Theorem 1.9], we can choose a maximal $k$-split torus $\textbf{S}'$ of $\textbf{G}$ containing $\textbf{S}$ and a maximal $K$-split $k$-torus $\textbf{T}'$ of $\textbf{G}$ containing $\textbf{S}'$ and $\textbf{T}.$ Since $\textbf{G}$ is quasi-split over $K,$ we know that $\textbf{Z}':=C_{\textbf{G}^{\circ}}(\textbf{T}')$ is a maximal $k$-torus in $\textbf{G}$ containing $\textbf{T}'.$ We will define $\textbf{Z}$ to be the $C_{\textbf{H}^{\circ}}(\textbf{T}),$ which is a maximal $k$-torus of $\textbf{H}.$
	 
	 Let \mbox{\boldmath$\mathfrak{z}$}$'$ denote the Lie algebra of $\textbf{Z}'.$ Following [\cite{moy-prasad}, 3.2], there is a filtration of \mbox{\boldmath$\mathfrak{z}$}$'(K)$ for each $r \in \mathbb{R}$ which we denote \mbox{\boldmath$\mathfrak{z}$}$'(K)_r.$ Moreover, for each affine functional $\psi \in \Psi(\mathcal{A}(\textbf{T}', K) ),$ there exists a lattice denoted  \mbox{\boldmath$\mathfrak{u}$}$_{\psi}$, which lies in the root space in  \mbox{\boldmath$\mathfrak{g}$}$(K)$ with respect to $\textbf{T}, \textbf{G},$ and $\dot{\psi}.$ The lattice  \mbox{\boldmath$\mathfrak{g}$}$(K)_{x,r}$ is defined as the the $R_K$-submodule of  \mbox{\boldmath$\mathfrak{g}$}$(K)$ spanned by \mbox{\boldmath$\mathfrak{z}$}$'(K)_r$ and the  \mbox{\boldmath$\mathfrak{u}$}$_{\psi}$'s for which $\psi(x) \geq r.$

\section{$\theta$-stable $k$-split tori}

\subsection{A result of Prasad-Yu}

In order to discuss a type of facet which takes both $\theta$ and the facet structure of $\mathcal{B}(G)$ into account, it will be useful to discuss the relationship between apartments in $\mathcal{B}(H)$ and those in $\mathcal{B}(G).$ Prasad and Yu have shown in [\cite{prasad-yu}, Theorem 1.9] that there is an $\textbf{H}^{\circ}(k)$-equivariant map $\iota: \mathcal{B}(H) \rightarrow \mathcal{B}(G)$ such that the image is $\mathcal{B}(G)^{\theta},$ uniquely defined up to translation by $\textbf{X}_{*}(\textbf{C}) \otimes \mathbb{R},$ where $\textbf{C}$ is the maximal $k$-split torus in the center of $\textbf{H}.$ Moreover, for every maximal $k$-split torus $\textbf{S}$ of $\textbf{H},$ there is a maximal $k$-split torus $\textbf{S}'$ of $\textbf{G}$ such that $\mathcal{A}(\textbf{S},k)$ is mapped into $\mathcal{A}(\textbf{S}',k)$ by an affine transformation. In [\cite{prasad-yu}, Lemma 1.9.3], it is shown that such a map is compatible with unramified base change. In particular, there is an $\textbf{H}^{\circ}(K)$-equivariant map $\iota_K : \mathcal{B}(\textbf{H}, K) \rightarrow \mathcal{B}(\textbf{G}, K)$ which is also Gal$(K/k)$-equivariant, such that the restriction of $\iota_K$ to $\mathcal{B}(H)$ shares the same properties as $\iota.$ 
	
	A proof of the next proposition is given in [\cite{prasad-yu}, Remark 1.5.4] when $k'$ is a completion of $K$ and in  [\cite{kim}, Corollary 5.7(i)] when $k'$ is a $p$-adic field or [\cite{kim}, Proposition 3.4.1] when $k'$ is a finite field of odd characteristic. In the statement of the following proposition, let $k'$ be one of the fields mentioned above, and suppose $\textbf{G}'$ is a reductive linear algebraic $k'$-group equipped with a  $k'$-involution $\sigma$ with $\textbf{H}' = (\textbf{G}^{\sigma})^{\circ}.$

\begin{prop} \label{stdapt} Let $\textbf{S}$ be a maximal $k'$-split torus of $\textbf{H}'.$ Then there exists a $\sigma$-stable maximal $k'$-split torus of $\textbf{G}'$ which contains $\textbf{S}.$
\end{prop}

Now, let $k$ be as in Section 2.1. Suppose $\textbf{S}$ is a maximal $k$-split torus of $\textbf{H}.$ The group defined by $\textbf{M}:=C_{\textbf{G}^{\circ}}(\textbf{S})$ is a reductive group defined over $k.$ Moreover, $\textbf{M}$ is $\theta$-stable since $\textbf{S}$ is a $\theta$-stable torus. Thus, applying [\cite{helminck-wang}, Proposition 2.3] to $\textbf{M},$ there exists a $\theta$-stable maximal $k$-torus $\textbf{Z}$ such that the maximal $k$-split torus $\textbf{T}$ in $\textbf{Z}$ is a maximal $k$-split torus in $\textbf{M}.$ In particular, the torus $\textbf{T}$ is a $\theta$-stable maximal $k$-split torus of $\textbf{G}$ which contains $\textbf{S}.$

\begin{remark} Let $\textbf{S}'$ be a $\theta$-stable maximal $k$-split torus of $\textbf{G}.$ The condition for $\textbf{S}'^{\theta}$ to be a maximal $k$-split torus of $\textbf{H}$ is that $\textbf{S}'$ lies in a minimal $\theta$-stable parabolic $k$-subgroup of $\textbf{G}.$ ([\cite{helminck-wang}, Prop. 4.5]) Under some restrictions on $k$ and the derived subgroup of $\textbf{G},$ existence of a $\theta$-stable $k$-parabolic subgroup is shown in [\cite{helminck-wang}, Proposition 4.4]. \end{remark}

\begin{remark} It is not true, in general, that a $\theta$-stable apartment of $\mathcal{B}(G)$ gives rise to an apartment in $\mathcal{B}(H).$ Consider $\textbf{SL}_2$ and the involution $\theta$ defined by $X \mapsto (X^t)^{-1}.$ Supposing $-1 \in (\mathbb{Q}_p^{\times})^2,$ the fixed points under this involution consist of a maximal $k$-split torus. However, the diagonal torus is a $\theta$-stable maximal $k$-split torus in $\textbf{G}$ whose set of $\theta$-fixed points is $\{\pm 1\}.$ \end{remark}

\section{Equivalence of facets}
	
\subsection{$(r,\theta)$-facets and Moy-Prasad lattices}

Fix an apartment $\mathcal{A}'$ of $\mathcal{B}(G).$ For $\psi \in \Psi(\mathcal{A}'),$ define the hyperplane

$$H_{\psi -r } := \{ x \in \mathcal{A}' \mid \psi(x) = r \}.$$
As in [\cite{d}, Section 3.1], we call a nonempty subset $F' \subset \mathcal{A}'$ an $r$-facet of $\mathcal{A}'$ if there is some finite subset $S \subset \Psi(\mathcal{A}')$ for which 

\begin{enumerate} 
\item $F' \subset H_S:=\bigcap_{\psi \in S} H_{\psi-r}$
\item $F'$ is a connected component (in $H_S$) of $H_S \backslash \bigcup_{\psi \in \Psi(\mathcal{A}') \backslash S} (H_S \cap H_{\psi-r}).$
\end{enumerate}
If $F'$ is an $r$-facet in $\mathcal{A}',$ we define its dimension to be the dimension $A(F', \mathcal{A}').$ 

	The following remark, which is a consequence of the definitions above, will be important for later discussion of $\theta$-stable $r$-facets.

\begin{remark} \label{facetremark} If $F_1', F_2'$ are $r$-facets in $\mathcal{A}',$ and $F_2' \cap A(F_1', \mathcal{A}') \neq \emptyset,$ then $F_2'$ is entirely contained in $A(F_1', \mathcal{A}').$ To see why this is true, write 
$$ A(F_1', \mathcal{A}') = H_S:= \bigcap_{\psi \in S} H_{\psi-r}, $$
where $S$ is finite and $F_1'$ is a connected component of $H_S \backslash \bigcup_{\psi \in \Psi(\mathcal{A}') \backslash S} (H_S \cap H_{\psi-r}).$ Write
$$ A(F_2', \mathcal{A}') = H_{S'}:=\bigcap_{\rho \in S'} H_{\rho-r}, $$
where $S'$ is finite and $F_2'$ is a connected component of $H_{S'} \backslash \bigcup_{\rho \in \Psi(\mathcal{A}') \backslash S'} (H_{S'} \cap H_{\rho-r}).$ 

	It will be enough to show that $S \subset S'.$ If $S$ is empty, then the statement is obviously true, so suppose $\psi \in S \backslash S'.$ Let $x_2 \in F_2' \cap A(F_1', \mathcal{A}').$ Then, since $F_2'$ lies in $H_{S'} \backslash \bigcup_{\rho \in \Psi(\mathcal{A}') \backslash S'} (H_{S'} \cap H_{\rho-r}),$ we must have $x_2 \in \mathcal{A}' \backslash H_{\psi -r}.$ But, since $x_2 \in A(F_1', \mathcal{A}'),$ we have $\psi(x_2) = r,$ a contradiction. Thus, we must have $S \subset S',$ that is, $F_2'$ is entirely contained in $A(F_1', \mathcal{A}').$
\end{remark}

\begin{defn} \label{thetaf} Define an $(r,\theta)$-facet to be a nonempty subset $F$ in an apartment $\mathcal{A} \subset \mathcal{B}(H)$ such that there exists an apartment $\mathcal{A'} \subset \mathcal{B}(G)$ and an $r$-facet $F' \subset \mathcal{A'}$ with $F = F'^{\theta}.$ 
\end{defn}

\begin{remark} \label{thetafacetremark} In Definition \ref{thetaf}, we have defined a structure on apartments in $\mathcal{B}(H)$ which is finer than the $r$-facet structure of apartments in $\mathcal{B}(H).$ For example, take $r=0, \textbf{G} = \textbf{SL}_2$ equipped with the involution $\theta(A) = J(A^t)^{-1}J$ as in Example 1.1. From Figure 1, we see that $\theta$-facet $F_2$ is a strictly smaller subset of the $H$-alcove of $\mathcal{B}(H)$ that has boundary $\{F_1, F_1 + (\check{\alpha} + \check{\beta}) \}$  \end{remark}

\begin{defn} Let $F$ be an $(r,\theta)$-facet in an apartment $\mathcal{A} \subset \mathcal{B}(H).$ Define the dimension of an $(r,\theta)$-facet $F$ by \end{defn}

$$ \textup{ dim } F := \textup{dim }A(F, \mathcal{A}) $$

Let $\mathcal{A}'$ be an apartment of $\mathcal{B}(G).$ Suppose $F$ is an $(r,\theta)$-facet which lies inside an $r$-facet $F' \subset \mathcal{A}',$ and let $x,y \in F.$ Then, in particular, since $x$ and $y$ lie inside the $r$-facet $F',$ we have $\mathfrak{g}_{x,r} = \mathfrak{g}_{y,r}$ and $\mathfrak{g}_{x,r^+} = \mathfrak{g}_{y,r^+}.$ A proof of this statement can be found in [\cite{d}, 3.1.4]. This allows us to make the following definition. 

\begin{defn} Let $F$ be an $(r,\theta)$-facet of $\mathcal{A}.$ Fix $x \in F.$ Set \end{defn}

$$ \mathfrak{g}_{F} := \mathfrak{g}_{x,r} $$
and
$$ \mathfrak{g}_{F}^+ := \mathfrak{g}_{x,r^+}.$$

\begin{lemma} \label{facetpa} Let $F$ be an $(r,\theta)$-facet of $\mathcal{A} \subset \mathcal{B}(H).$ Suppose $x \in \mathcal{A}.$ Then $x \in F$ if and only if $\mathfrak{g}_{x,r} = \mathfrak{g}_F$ and $\mathfrak{g}_{x,r^+} = \mathfrak{g}_F^+.$ \end{lemma}

 \begin{proof} We apply the analogous result in [\cite{d}, 3.1.4]. If $x \in F,$ then, by definition, we have $\mathfrak{g}_{x,r} = \mathfrak{g}_F$ and $\mathfrak{g}_{x,r^+} = \mathfrak{g}_F^+.$ Suppose $F = F'^{\theta},$ where $F'$ is an $r$-facet in an apartment $\mathcal{A'} \subset B(G).$ Then, by [\cite{d}, Lemma 3.1.4], since $\mathfrak{g}_{x,r} = \mathfrak{g}_F = \mathfrak{g}_{F'}$ and $\mathfrak{g}_{x,r^+} = \mathfrak{g}_F^+= \mathfrak{g}_{F'}^+,$ we have $x \in F' \cap \mathcal{A} = F.$ 
 
 \end{proof}
  
Let $x \in \mathcal{B}(H).$ Before showing that there is a reasonable decomposition of the Moy-Prasad lattice $\mathfrak{g}_{x,r}$ with respect to $\theta,$ we demonstrate the relationship between the Moy-Prasad lattices $\mathfrak{h}_{x,r}$ and $\mathfrak{g}_{x,r}.$ The statements we make when discussing these lattices make sense because of [\cite{prasad-yu}, Theorem 1.9]. 

	We first make an observation. Consider the parahoric subgroups $\textbf{H}(K)_x$ and $\textbf{G}(K)_x.$ It is clear that $\textbf{H}(K)_x \subset \textup{stab}_{\textbf{G}(K)}(x).$ By the argument given in [\cite{prasad-yu}, Prop 1.7], we must have $\textbf{H}(K)_x \subset \textbf{G}(K)_x \cap \textbf{H}(K).$ On other other hand, the map $\theta$ induces an involution of the smooth, affine $R$-group scheme associated to $\textbf{G}(K)_x,$ which we denote $\mathcal{G}.$ Call the fixed points of this group scheme $\mathcal{G}^{\theta}.$ Since $\textbf{H}(K)_x \subset \textbf{G}(K)_x^{\theta},$ we have an induced inclusion of $\mathcal{H} \subset \mathcal{G}^{\theta},$ where $\mathcal{H}$ is the smooth, connected, affine $R$-group scheme associated to the parahoric $\textbf{H}(K)_x.$ Let $\mathcal{G}'$ be the smooth, affine (not necessarily connected) $R$-group scheme associated to $\textup{stab}_{\textbf{H}(K)}(x).$ Then, we have inclusions
	
	$$ \mathcal{H} \subset \mathcal{G}^{\theta} \subset \mathcal{G}',$$
where the group scheme $\mathcal{H}$ is of finite index in $\mathcal{G}'.$
Taking Lie algebras and $R_K$-rational points, since $x$ is Gal$(K/k)$-fixed, this gives us the equality $\mathfrak{h}_x = \mathfrak{g}_x \cap \mathfrak{h}.$

\begin{lemma} \label{mplattice} Let $x \in \mathcal{B}(H).$ Then, we have 

$$ \mathfrak{g}_{x,r} \cap \mathfrak{h} = \mathfrak{h}_{x,r}. $$
\end{lemma}

\begin{proof} For this proof only, if $M/k$ is a finite extension, let $\textup{ord}(M)$ denote the image $\nu(M^{\times}),$ where $\nu$ denotes the extension of $\nu$ to $M.$ Following the proof of [\cite{jkyu}, Lemma 8.2], we will first assume that $r \in \textup{ord}(k).$ If $\pi_r$ is an element of $k$ with valuation $-r,$ then $\pi_r \mathfrak{g}_{x,r} = \mathfrak{g}_{x,0} = \mathfrak{g}_x,$ so the result follows from the observation preceding this proof. 
	
	If $r \in \textup{ord}(k) \otimes_{\mathbb{Z}} \mathbb{Q},$ we use [\cite{adler}, 1.4.1] to reduce the statement to the one above. The statement of the proposition now follows by noting that for any real number $r,$ we have
	
	$$ \mathfrak{g}_{x,r} = \bigcap_{s < r, s \in \mathbb{Q}} \mathfrak{g}_{x,s}.$$

\end{proof}
 
\begin{prop} \label{thstpa} Assume $p \neq 2.$ Let $F_i$ (i=1,2) be nonempty $(r,\theta)$-facets in some apartment $\mathcal{A} \subset \mathcal{B}(H),$ and suppose $F_i' \supset F_i$ is an $r$-facet in some apartment $\mathcal{A}' \subset \mathcal{B}(G)$ containing $\mathcal{A}.$ Define $\mathfrak{p}_{F_i}:=\mathfrak{p} \cap \mathfrak{g}_{F_i}$ and $\mathfrak{p}_{F_i}^+ := \mathfrak{p} \cap \mathfrak{g}_{F_i}^+.$ Then, for $i =1,2,$ we have: \end{prop}

\begin{enumerate} 

\item $ \mathfrak{g}_{F_i} = \mathfrak{h}_{F_i} \oplus \mathfrak{p}_{F_i} \textup{ and } \mathfrak{g}_{F_i}^+ = \mathfrak{h}_{F_i}^+ \oplus \mathfrak{p}_{F_i}^+ $

\item $F_1' = F_2' \textup{ if and only if } F_1 = F_2 $

\end{enumerate}
 
\begin{proof} Since $\theta$ is an automorphism of $\textbf{G}$ defined over $k,$ it induces an action on $\mathcal{B}(G)$ which is compatible with all structures on $\mathcal{B}(G).$ In particular, by [\cite{adler-debacker}, Proposition 2.2.1], since $F_i \subset \mathcal{B}(H),$ we have that $\mathfrak{g}_{F_i}$ and $\mathfrak{g}_{F_i}^+$ are $\theta$-stable. Let $X \in \mathfrak{g}_{F_i} = \mathfrak{g}_{x,r}.$ Since $\mathfrak{g}_{x,r}$ is $\theta$-stable, and $p \neq 2,$ we write

$$ X = \frac{X+\theta(X)}{2} + \frac{X-\theta(X)}{2} \in (\mathfrak{h} \cap \mathfrak{g}_{x,r}) \oplus (\mathfrak{p} \cap \mathfrak{g}_{x,r}). $$
Since $\mathfrak{h}, \mathfrak{p}$ and $\mathfrak{g}_{x,r}$ and $\mathfrak{g}_{x,r^+}$ are all $R$-modules, by Lemma \ref{mplattice}, the above computation shows 

$$ \mathfrak{g}_{x,r} = (\mathfrak{h} \cap \mathfrak{g}_{x,r}) \oplus (\mathfrak{p} \cap \mathfrak{g}_{x,r}) = \mathfrak{h}_{F_i} \oplus \mathfrak{p}_{F_i} $$
and
$$ \mathfrak{g}_{x,r^+} = (\mathfrak{h} \cap \mathfrak{g}_{x,r^+}) \oplus (\mathfrak{p} \cap \mathfrak{g}_{x,r^+}) = \mathfrak{h}_{F_i}^+ \oplus \mathfrak{p}_{F_i}^+.$$

	For the second claim,  the forward implication is trivial. For the other direction, note that if $F_1 = F_2,$ then

$$ \mathfrak{g}_{F_1'} = \mathfrak{g}_{F_1} = \mathfrak{g}_{F_2} = \mathfrak{g}_{F_2'}. $$
and
$$ \mathfrak{g}_{F_1'}^+ = \mathfrak{g}_{F_1}^+ = \mathfrak{g}_{F_2}^+ = \mathfrak{g}_{F_2'}^+$$
which is true if and only if $F_1' = F_2'$ by [\cite{d}, Lemma 3.1.4].

\end{proof}

\subsection{Generalized $(r,\theta)$-facets}

For the following definition, recall (from [\cite{d}, 3.2.1]) that for $x \in \mathcal{B}(G),$ the set $F^*(x) = \{ y \in \mathcal{B}(G) \mid \mathfrak{g}_{x,r} = \mathfrak{g}_{y,r} \textup{ and } \mathfrak{g}_{x,r^+} = \mathfrak{g}_{y,r^+} \}$ is called a generalized $r$-facet.

 \begin{defn} \label{genfacet} Let $x \in \mathcal{B}(H).$ Define \end{defn}
 
 $$ F_{\theta}^*(x) :=F^*(x)^{\theta}. $$

\vspace{.2in}

\begin{defn} $$ \mathcal{F}_{\theta}(r) := \{ F_{\theta}^*(x) \mid x \in \mathcal{B}(H) \}.$$ \end{defn}
We call an element of $\mathcal{F}_{\theta}(r)$ a generalized $(r,\theta)$-facet. 

\begin{remark} \label{patr} We briefly mention a fact which will be used many times throughout this section. By [\cite{bt2}, 4.6.28], if $\mathcal{A}, \tilde{\mathcal{A}}$ are two apartments in $\mathcal{B}(H)$ such that $\Omega = \{x,y\} \subset \mathcal{A} \cap \tilde{\mathcal{A}},$ then there exists an element $h \in H_{\Omega}$ such that $h\mathcal{A} = \tilde{\mathcal{A}}.$ More succinctly stated, $H_{\Omega}$ acts transitively on the apartments of $\mathcal{B}(H)$ containing $\Omega.$ \end{remark}

\begin{remark} We remark that if $F_{\theta}^*(x)$ is a generalized $(r,\theta)$-facet, and $\mathcal{A}$ is an apartment of $\mathcal{B}(H)$ such that $F_{\theta}^*(x) \cap \mathcal{A} \neq \emptyset,$ then $F_{\theta}^*(x) \cap \mathcal{A}$ is an $(r,\theta)$-facet of $\mathcal{A}.$ \end{remark}

\begin{lemma} \label{thfacetpa} Let $x \in \mathcal{B}(H)$ and $\mathcal{A}$ an apartment in $\mathcal{B}(H)$ such that $F:= F_{\theta}^*(x) \cap \mathcal{A}  \neq \emptyset.$ For all $y \in F,$ we have \end{lemma}

$$ F_{\theta}^*(x) = H_y \cdot F. $$

\begin{proof} Let $y \in F.$ 
\newline
$``\subset "$: Let $z \in F_{\theta}^*(x),$ and let $\tilde{A}$ be an apartment in $\mathcal{B}(H)$ containing $y$ and $z.$ By Remark \ref{patr}, there exists an element $h \in H_y$ such that $hz \in \mathcal{A}.$ We have 

$$ \mathfrak{g}_{hz,r} = {}^{h}\mathfrak{g}_{z,r} = {}^{h}\mathfrak{g}_{x,r} = {}^{h}\mathfrak{g}_{y,r} = \mathfrak{g}_{hy,r} = \mathfrak{g}_{y,r} = \mathfrak{g}_{x,r} $$
and similarly,

$$  \mathfrak{g}_{hz,r^+} = {}^{h}\mathfrak{g}_{z,r^+} = {}^{h}\mathfrak{g}_{x,r^+} = {}^{h}\mathfrak{g}_{y,r^+} = \mathfrak{g}_{hy,r^+} = \mathfrak{g}_{y,r^+} = \mathfrak{g}_{x,r^+}. $$
Thus, $hz \in \mathcal{A} \cap F_{\theta}^*(x) = F,$ so, in particular, $z \in H_y \cdot F.$ 
\newline
$``\supset "$: Let $z \in F$ and $h \in H_y.$ Then 

$$ \mathfrak{g}_{hz,r} = {}^{h}\mathfrak{g}_{z,r} = {}^{h}\mathfrak{g}_{y,r} = \mathfrak{g}_{hy,r} = \mathfrak{g}_{y,r} = \mathfrak{g}_{x,r}.$$
and
$$  \mathfrak{g}_{hz,r^+} = {}^{h}\mathfrak{g}_{z,r^+} = {}^{h}\mathfrak{g}_{y,r^+} = \mathfrak{g}_{hy,r^+} = \mathfrak{g}_{y,r^+} = \mathfrak{g}_{x,r^+}. $$
Thus,  $hz \in F_{\theta}^*(x).$

\end{proof}

\begin{cor} \label{imbo} If $F_{\theta}^* \in \mathcal{F}_{\theta}(r),$ then the image of $F_{\theta}^*$ in the reduced building of $H$ is bounded. \end{cor}

\begin{proof}

Let $x \in F_{\theta}^*,$ and let $F^*$ be the generalized $r$-facet in $\mathcal{B}(G)$ containing $x.$ The result follows directly from the result in [\cite{d}, Corollary 3.2.7] since the image of $F_{\theta}^*$ is contained in the image of $F^*$ in the reduced building.
\end{proof}

\begin{lemma} \label{normpa} Let $x \in B(H).$ We have \end{lemma}

$$ N_H(\mathfrak{g}_{x,r}) \cap N_H(\mathfrak{g}_{x,r^+}) = \textup{stab}_H(F_{\theta}^*(x)). $$

\begin{proof} $``\supset"$: If $h \in \textup{stab}_H(F_{\theta}^*(x)),$ then $hx \in F_{\theta}^*(x),$ so, in particular, 
$$ {}^{h}\mathfrak{g}_{x,r} = \mathfrak{g}_{hx,r} = \mathfrak{g}_{x,r}$$
and
$$ {}^{h}\mathfrak{g}_{x,r^+} = \mathfrak{g}_{hx, r^+} = \mathfrak{g}_{x,r^+}.$$
$``\subset"$: Let $n \in N_H(\mathfrak{g}_{x,r}) \cap N_H(\mathfrak{g}_{x,r^+}).$ Choose $z \in F_{\theta}^*(x).$ Since $n$ normalizes the lattices $\mathfrak{g}_{x,r}$ and $\mathfrak{g}_{x,r^+},$ we have

$$ \mathfrak{g}_{nz,r} = {}^{n}\mathfrak{g}_{z,r} = {}^{n}\mathfrak{g}_{x,r} = \mathfrak{g}_{x,r} $$
and
$$ \mathfrak{g}_{nz,r^+} = {}^{n}\mathfrak{g}_{z,r^+} = {}^{n}\mathfrak{g}_{x,r^+} = \mathfrak{g}_{x,r^+}. $$
This implies, by definition of $F^*(x),$ that $nz \in F^*(x).$ Since $z \in \mathcal{B}(H),$ and $n \in H,$ we have $nz \in F^*(x) \cap \mathcal{B}(H) = F^*(x)^{\theta} = F_{\theta}^*(x).$ Thus, since $z \in F_{\theta}^*(x)$ was arbitrary, $nF_{\theta}^*(x) \subset F_{\theta}^*(x)$ and thus $n \in \textup{stab}_H(F_{\theta}^*(x)).$

\end{proof}

\begin{lemma} \label{clfa} Let $F_{\theta}^* \in \mathcal{F}_{\theta}(r)$ and $\mathcal{A}$ an apartment in $B(H)$ such that $F:= F_{\theta}^* \cap \mathcal{A} \neq \emptyset.$ Then \end{lemma}

$$ \overline{F} = \overline{F_{\theta}^*} \cap \mathcal{A}. $$

\begin{proof} $``\subset"$: This inclusion is clear since $\overline{A \cap B} \subset \overline{A} \cap \overline{B}$ for any two subsets $A, B$ of $\mathcal{B}(H).$ \newline
$``\supset":$ Let $x \in \overline{F_{\theta}^*} \cap \mathcal{A}.$ We will produce a sequence converging to $x$ which lies in $F.$ Since $x \in \overline{F_{\theta}^*},$ there exists a sequence $\{x_n\}$ in $F_{\theta}^*$ converging to $x.$ Fix $y \in F.$ Without loss of generality, assume $\textup{\textup{dist}}(x_n,x) < \frac{1}{n}$ for each $n.$ Note that $ \bigcup_{x \in \overline{C}} \overline{C} $ contains a neighborhood of $x,$ where $C$ ranges over all alcoves in $\mathcal{B}(H).$ Thus, for large $n,$ there exist alcoves $C_n \subset \mathcal{B}(H)$ such that $x_n$ and $x$ lie in $\overline{C_n}.$ Let $\mathcal{A}_n$ be an apartment in $\mathcal{B}(H)$ which contains $C_n$ and $y.$ We now fix $n.$ Since $x$ and $y$ lie in $\mathcal{A}_n \cap \mathcal{A},$ by Remark \ref{patr}, there is an element $h_n \in H$ that maps $\mathcal{A}_n$ to $\mathcal{A}$ and fixes $x$ and $y.$ In particular, since $h_n x = x,$ we have

$$ \textup{dist}(h_nx_n, x) = \textup{dist}(x_n, x) < \frac{1}{n}.$$
Now, by Lemma \ref{normpa}, since $h_ny = y$ for each $n,$ we have $h_n \in N_H(\mathfrak{g}_{y,r}) \cap N_H(\mathfrak{g}_{y,r^+}) = \textup{stab}_H(F_{\theta}^*).$ Thus, $h_n x_n \in F_{\theta}^*$ and since $h_n \mathcal{A}_n = \mathcal{A},$ we also have $h_nx_n \in \mathcal{A},$ so $h_nx_n \in F_{\theta}^* \cap \mathcal{A} = F.$ Therefore, $\{h_nx_n\}$ is our desired sequence. 
\end{proof}

\begin{defn} For $F_{\theta}^* \in \mathcal{F}_{\theta}(r)$ and $\delta > 0,$ define \end{defn}

$$ F_{\theta}^*(\delta) := \{ x \in F_{\theta}^* \mid \textup{\textup{dist}}(x,z) \geq \delta \textup{ for all } z \in \overline{F_{\theta}^*} \backslash F_{\theta}^* \}. $$

\begin{lemma} \label{delta} Suppose $F_{\theta}^* \in \mathcal{F}_{\theta}(r)$ and $\delta > 0.$ Then $F_{\theta}^*(\delta)$ is a convex, closed, $\textup{stab}_H(F_{\theta}^*)$-invariant subset of $\mathcal{B}(H).$ Also, $F_{\theta}^*(\delta)$ is nonempty if and only if  there exists an apartment $\mathcal{A}$ in $\mathcal{B}(H)$ such that the following subset of $F:= F_{\theta}^* \cap \mathcal{A}$ \end{lemma}

$$ F_{\theta,\mathcal{A}}(\delta) := \{ x \in F \mid \textup{\textup{dist}}(x,z) \geq \delta \textup{ for all } z \in \overline{F} \backslash F \} $$
is nonempty.

\begin{proof} To see that $F_{\theta}^*(\delta)$ is closed, suppose $\{x_n\} \subset F_{\theta}^*(\delta)$ is a sequence converging to some $x \in \mathcal{B}(H).$ By the triangle inequality, we have

$$ \textup{\textup{dist}}(x,z) \geq \textup{\textup{dist}}(x_n,z) - \textup{\textup{dist}}(x_n,x) \geq \delta - \textup{\textup{dist}}(x_n,x), $$
for all $z \in \overline{F_{\theta}^*} \backslash F_{\theta}^*,$ so taking $n \rightarrow \infty,$ we see that $\textup{\textup{dist}}(x,z) \geq \delta.$ 
	
	For $\textup{stab}_H(F_{\theta}^*)$-invariance, we note that any element $h \in \textup{stab}_H(F_{\theta}^*)$ sends the boundary of $F_{\theta}^*$ to itself. In particular, for $z \in \overline{F}_{\theta}^* \backslash F_{\theta}^*,$ there exists an element $w \in \overline{F}_{\theta}^* \backslash F_{\theta}^*$ such that $hw = z.$ Thus, if $x \in F_{\theta}^*(\delta),$ then 

$$ (\dagger) \textup{ } \textup{\textup{dist}}(hx,z) = \textup{\textup{dist}}(hx,hw) = \textup{\textup{dist}}(x,w) \geq \delta. $$
We now prove the final statement of the lemma. \newline
$``\Rightarrow ":$ If $F_{\theta}^*(\delta)$ is nonempty, then there exists some apartment $\mathcal{A}$ for which $F_{\theta}^*(\delta) \cap \mathcal{A},$ and hence $F_{\theta,\mathcal{A}}(\delta),$ is nonempty. \newline
$``\Leftarrow ":$ We will prove a stronger claim here that will be used later to prove convexity of $F_{\theta}^*(\delta).$ In particular, we show that if there is an apartment $\mathcal{A} \subset \mathcal{B}(H)$ such that $F_{\theta,\mathcal{A}}(\delta) \neq \emptyset,$ then $H_yF_{\theta,\mathcal{A}}(\delta) = F_{\theta}^*(\delta)$ for all $y \in F_{\theta,\mathcal{A}}(\delta).$ This will show that $F_{\theta}^*(\delta) \neq \emptyset$ whenever $F_{\theta,\mathcal{A}}(\delta) \neq \emptyset.$ Suppose $\mathcal{A} \subset \mathcal{B}(H)$ is an apartment such that $F_{\theta,\mathcal{A}}(\delta) \neq \emptyset,$ and let $w \in F_{\theta,\mathcal{A}}(\delta).$ We first show $H_wF_{\theta,\mathcal{A}}(\delta) \subset F_{\theta}^*(\delta).$ \newline
$``\subset":$ Note that $H_w \subset \textup{stab}_H(F_{\theta}^*)$ by an application of Lemma \ref{normpa}. Thus by $(\dagger),$ we have that $F_{\theta}^*(\delta)$ is $H_w$-invariant. As a consequence, it suffices to show that $F_{\theta,\mathcal{A}}(\delta) \subset F_{\theta}^*(\delta).$ Let $x \in F_{\theta,\mathcal{A}}(\delta)$ and $z \in \overline{F_{\theta}^*} \backslash F_{\theta}^*.$ By [\cite{bt1}, 2.3.1], we may choose an apartment $\tilde{\mathcal{A}}$ containing $x$ and $z.$ By Remark \ref{patr}, there is an element $h \in H_x$ that maps $\tilde{\mathcal{A}}$ onto $\mathcal{A}.$ Since $hx=x,$ by an application of Lemma \ref{normpa}, $h \in \textup{stab}_H(F_{\theta}^*).$ Since $h\tilde{\mathcal{A}} =\mathcal{A},$ we have $hz \in \mathcal{A},$ so in particular, $hz \in (\overline{F_{\theta}^*} \backslash F_{\theta}^*) \cap \mathcal{A}.$ Thus, by Lemma \ref{clfa}, 

$$ hz \in (\overline{F_{\theta}^*} \cap \mathcal{A}) \backslash (F_{\theta}^* \cap \mathcal{A}) = \overline{F_{\theta,\mathcal{A}}} \backslash F_{\theta,\mathcal{A}}. $$ 
	Again, since $h \in H_x,$ we have $\textup{\textup{dist}}(x,z) = \textup{\textup{dist}}(x,hz) \geq \delta,$ so since $z$ was arbitrary, we must have $x \in F_{\theta}^*(\delta).$ \newline
$``\supset":$ Now, we show that $H_wF_{\theta,\mathcal{A}}(\delta) \supset F_{\theta}^*(\delta)$ for all $w \in F_{\theta,\mathcal{A}}(\delta).$ Let $x \in F_{\theta}^*(\delta).$ By Lemma \ref{thfacetpa}, there exist elements $h \in H_w$ and $z \in F$ such that $hz=x.$ Arguing as usual, since $w \in F_{\theta}^*,$ we have $H_w \subset \textup{stab}_H(F_{\theta}^*),$ so, in particular, 
	
$$ x \in hF \cap F_{\theta}^*(\delta) = h(F \cap F_{\theta}^*(\delta)) \subset hF_{\theta,\mathcal{A}}(\delta). $$
	
	Lastly, we must show that $F_{\theta}^*(\delta)$ is a convex subset of $\mathcal{B}(H).$ Assume $F_{\theta}^*(\delta)$ is nonempty. We first show that $F_{\theta,\mathcal{A}}(\delta) \subset \mathcal{A}$ is convex. Choose an origin $O$ in $\mathcal{A}.$ Note that the geodesics of $\mathcal{A}$ are nothing more than segments, so if $x,y \in F_{\theta,\mathcal{A}}(\delta)$ and $z \in [x,y],$ then considering $x,y,$ and $z$ as the vectors, $x-O, y-O,$ and $z-O,$ respectively, we have $z = tx + (1-t)y$ for some $t \in [0,1].$ Thus, for all $z' \in \overline{F}\backslash F,$
	
$$ \textup{\textup{dist}}(tx+(1-t)y,z') \geq \textup{\textup{dist}}(tx,tz') - \textup{\textup{dist}}((1-t)y,(1-t)z') = t \textup{\textup{dist}}(x,z') + (1-t) \textup{\textup{dist}}(y,z') \geq \delta.$$
so $z$ lies in $F_{\theta,\mathcal{A}}(\delta).$	
	
	Now, suppose $F_{\theta}^*(\delta)$ is nonempty. There is an apartment $\mathcal{A}$ such that $F_{\theta,\mathcal{A}}(\delta)$ is nonempty. Let $x,y \in F_{\theta}^*(\delta)$ and $z \in F_{\theta,\mathcal{A}}(\delta).$ Previously in this proof, we showed that for all $w \in F_{\theta,\mathcal{A}}(\delta),$ we have 
	
$$ (\ddagger) \textup{ } H_wF_{\theta,\mathcal{A}}(\delta) = F_{\theta}^*(\delta),$$
 so there is some $h \in H_z$ such that $hx \in F_{\theta,\mathcal{A}}(\delta).$ Since $h \in H_z \subset \textup{stab}_{H}(F_{\theta}^*),$ we have $hy \in F_{\theta}^*,$ so for all $z' \in \overline{F_{\theta}^*} \backslash F_{\theta}^*,$ we have
	
$$ \textup{dist}(hy,z') = \textup{dist}(hy,hw') = \textup{dist}(y,w') \geq \delta$$
where $w' \in \overline{F_{\theta}^*} \backslash F_{\theta}^*$ such that $hw'=z.$ Thus, $hy \in F_{\theta}^*(\delta).$ Applying $(\ddagger)$ again, there exists some $h' \in H_{hx}$ such that $h'hy \in F_{\theta,\mathcal{A}}(\delta).$ Thus, since $F_{\theta,\mathcal{A}}^*(\delta)$ is convex, we have

$$ [h'hx,h'hy] \subset F_{\theta,\mathcal{A}}^*(\delta) \subset F_{\theta}^*(\delta)$$
so, in particular, since $F_{\theta,\mathcal{A}}^*(\delta)$ is $\textup{stab}_H(F_{\theta}^*)$-invariant, $[x,y] \subset F_{\theta}^*(\delta).$
\end{proof}

\begin{defn} For $F_{\theta}^* \in \mathcal{F}_{\theta}(r),$ define \end{defn}

$$ C(F_{\theta}^*) := \left\{ y \in F_{\theta}^* \mid \textup {for all apartments } \mathcal{A} \textup{ of } \mathcal{B}(H) \textup{ for which } \mathcal{A} \cap F_{\theta}^* \neq \emptyset \textup{ we have } y \in \mathcal{A} \right\}. $$

\begin{remark} \label{boundedgroup} Suppose $\textbf{H}$ is semisimple. Following the discussion in [\cite{tits}, 2.2.1], there is a map 

$$ \pi: H \times \mathcal{B}(H) \rightarrow \mathcal{B}(H) \times \mathcal{B}(H) $$
given by $(h,x) \mapsto (hx,x),$ with the property that the inverse images of bounded sets are bounded. We note that if $\Omega$ is a bounded subset of $\mathcal{B}(H),$ this tells us that 

$$ \textup{stab}_H(\Omega ) \times \Omega = \pi^{-1}( \Omega \times \Omega) $$
is bounded. In particular, $\textup{stab}_H(\Omega)$ is bounded whenever $\Omega$ is bounded. \end{remark}

\begin{cor} \label{Cnonempty} If $F_{\theta}^* \in \mathcal{F}_{\theta}(r),$ then $C(F_{\theta}^*) \neq \emptyset.$ \end{cor}

\begin{proof} Let $\textup{pr}_{ss}$ denote the projection from the enlarged building of $H$ to the reduced building of $H.$ Since 

$$ C(\textup{pr}_{ss}(F_{\theta}^*)) \neq \emptyset \Rightarrow C(F_{\theta}^*) \neq \emptyset, $$
we may assume $\textbf{H}$ is semisimple. 

	By Corollary \ref{imbo}, $F_{\theta}^*$ is bounded in $\mathcal{B}(H),$ so, by Remark \ref{boundedgroup}, the stabilizer $N:=\textup{stab}_H(F_{\theta}^*)$ is a bounded subgroup of $H.$ If $F_{\theta}^*$ consists of a point $x,$ then $\mathcal{A} \cap F_{\theta}^* = \tilde{\mathcal{A}} \cap F_{\theta}^* = \{x\}$ for all apartments in $\mathcal{B}(H)$ that meet $\{x\},$ so clearly $x \in C(F_{\theta}^*).$ 

	Suppose $F_{\theta}^*$ is not a point, and let $\mathcal{A} \subset \mathcal{B}(H)$ be an apartment such that $F_{\theta}^* \cap \mathcal{A} \neq \emptyset.$ Then, by Lemma \ref{thfacetpa}, since dim $F_{\theta}^* > 0,$ we have dim $F_{\theta}^* \cap \mathcal{A} > 0.$ As a consequence, there exists some $\delta > 0$ for which $F_{\theta,\mathcal{A}}(\delta)$ is nonempty. By Lemma \ref{delta}, this implies $\emptyset \neq F_{\theta}^*(\delta)$ is convex and $N$-stable. Thus, by [\cite{bt2}, 3.2.4], since a bounded group of isometries acting on a nonempty, closed, convex set $F$ of $\mathcal{B}(H)$ has a fixed point, there exists some $y \in F_{\theta}^*(\delta)$ such that $ny = y$ for all $n \in N.$ Suppose now that $\tilde{\mathcal{A}}$ is an apartment of $\mathcal{B}(H)$ for which $F_{\tilde{\mathcal{A}}} := F_{\theta}^* \cap \tilde{\mathcal{A}} \neq \emptyset.$ Let $z \in F_{\tilde{\mathcal{A}}}.$ By Lemma \ref{thfacetpa}, we have $H_zF_{\tilde{\mathcal{A}}} = F_{\theta}^*,$ and by Lemma \ref{normpa}, we have $H_z \subset N.$ In particular, $hy =y$ for all $h \in H_z.$ Therefore, $y \in F_{\tilde{\mathcal{A}}} \subset \tilde{\mathcal{A}},$ so $y \in C(F_{\theta}^*).$
\end{proof}

\begin{cor} If $\mathcal{A}_1$ and $\mathcal{A}_2$ are two apartments in $\mathcal{B}(H),$ and $F_{\theta}^* \in \mathcal{F}_{\theta}(r)$ such that $F_{\theta}^* \cap \mathcal{A}_i \neq \emptyset,$ then dim $A(F_{\theta}^* \cap \mathcal{A}_1, \mathcal{A}_1) =$ dim $A(F_{\theta}^* \cap \mathcal{A}_2, \mathcal{A}_2).$ \end{cor}

\begin{defn} Let $F_{\theta}^* \in \mathcal{F}_{\theta}(r).$ Let $\mathcal{A}$ be an apartment in $\mathcal{B}(H)$ with $F_{\theta}^* \cap \mathcal{A} \neq \emptyset.$ Define \end{defn}

$$ \textup{dim } F_{\theta}^* := \textup{dim } A(F_{\theta}^* \cap \mathcal{A}, \mathcal{A}). $$

\begin{defn} Suppose $F_{\theta}^* \in \mathcal{F}_{\theta}(r).$ Fix $x \in F_{\theta}^*$ and define \end{defn}

$$ \mathfrak{g}_{F_{\theta}^*} := \mathfrak{g}_{x,r} $$

and

$$ \mathfrak{g}_{F_{\theta}^*}^+ := \mathfrak{g}_{x,r^+}. $$

	Let $F_{\theta}^*$ be a generalized $(r,\theta)$-facet. From this point on, we will refer to the lattice $\mathfrak{g}_{F_{\theta}^*}^+$ frequently. Recall that the quotient $V:=\mathfrak{g}_{F_{\theta}^*}/\mathfrak{g}_{F_{\theta}^*}^+$ is an $\mathfrak{f}$-vector space. Moreover, by Proposition \ref{thstpa}, both $\mathfrak{g}_{F_{\theta}^*}$ and $\mathfrak{g}_{F_{\theta}^*}^+$ are $\theta$-stable, so $V$ is equipped with an involution induced by $\theta.$ We will abuse notation and call this induced map $\theta.$ In order to avoid confusion in the next lemma, we will let $V^{\theta-1}$ denote the set of $\theta$-fixed points in $V$ and let $V^{\theta + 1} = \{ v \in V \mid \theta(v) = -v \}.$ We reserve the symbol $+$ for the lattice $\mathfrak{g}_{F_{\theta}^*}^+.$

\begin{lemma} \label{thstcoset} Assume $p \neq 2.$ Suppose $F_{\theta}^*$ is a generalized $(r,\theta)$-facet of $\mathcal{B}(H).$ Then \end{lemma}

$$ \mathfrak{g}_{F_{\theta}^*} / \mathfrak{g}_{F_{\theta}^*}^{+} = \mathfrak{h}_{F_{\theta}^*} / \mathfrak{h}_{F_{\theta}^*}^+ \oplus \mathfrak{p}_{F_{\theta}^*} / \mathfrak{p}_{F_{\theta}^*}^+, $$
where we use Lemma \ref{mplattice} to identify $\mathfrak{h}_{F_{\theta}^*}/ \mathfrak{h}_{F_{\theta}^*}^+$ and $\mathfrak{p}_{F_{\theta}^*} / \mathfrak{p}_{F_{\theta}^*}^+$ with their images inside $\mathfrak{g}_{F_{\theta}^*} / \mathfrak{g}_{F_{\theta}^*}^{+}.$

\begin{proof} We first show that $\mathfrak{h}_{F_{\theta}^*} / \mathfrak{h}_{F_{\theta}^*}^+$ is the $(+1)$-eigenspace of $\mathfrak{g}_{F_{\theta}^*} / \mathfrak{g}_{F_{\theta}^*}^+$ under $\theta.$ By Lemma \ref{mplattice}, we have $\mathfrak{h}_{x,r} = \mathfrak{g}_{x,r} \cap \mathfrak{h}.$ Thus, we can identify $\mathfrak{h}_{F_{\theta}^*}/\mathfrak{h}_{F_{\theta}^*}^+$ as a subset of $\mathfrak{g}_{F_{\theta}^*}/\mathfrak{g}_{F_{\theta}^*}^+.$ Moreover, it is clear that $\theta$ fixes every element of $\mathfrak{h}_{F_{\theta}^*}/\mathfrak{h}_{F_{\theta}^*}^+.$ Thus, $\mathfrak{h}_{F_{\theta}^*}/\mathfrak{h}_{F_{\theta}^*}^+ \subset ( \mathfrak{g}_{F_{\theta}^*} / \mathfrak{g}_{F_{\theta}^*}^+ )^{\theta-1}.$ 
	
	Now, let $\overline{X} \in (\mathfrak{g}_{F_{\theta}^*} / \mathfrak{g}_{F_{\theta}^*}^{+})^{\theta-1},$ and let $X$ be a lift of $\overline{X}$ in $\mathfrak{g}_{F_{\theta}^*}.$  By Proposition \ref{thstpa}, we have $\mathfrak{g}_{F_{\theta}^*} = \mathfrak{h}_{F_{\theta}^*} \oplus \mathfrak{p}_{F_{\theta}^*}$, so we may write $X = X_+ + X_-,$ where $\theta(X_+) = X_+$ and $\theta(X_-) = -X_-,$ with $X_+ \in \mathfrak{h}_{F_{\theta}^*}$ and $X_- \in \mathfrak{p}_{F_{\theta}^*}.$ Note that $\theta(X) - X \in \mathfrak{g}_{F_{\theta}^*}^+,$ so we have 
$$-2X_- = \theta(X_+ + X_-) - (X_+ + X_-) \in \mathfrak{g}_{F_{\theta}^*}^+. $$
Thus, since $p \neq 2,$ we may conclude that $X_- \in \mathfrak{g}_{F_{\theta}^*}^+.$ Thus $X_+ + \mathfrak{h}_{F_{\theta}^*}^+$ is mapped to $X_+ + \mathfrak{g}_{F_{\theta}^*}^+ = X + \mathfrak{g}_{F_{\theta}^*}^+.$ In other words, $\overline{X}$ has a representative in $\mathfrak{h}_{F_{\theta}^*},$ so $( \mathfrak{g}_{F_{\theta}^*} / \mathfrak{g}_{F_{\theta}^*}^+ )^{\theta -1} \subset \mathfrak{h}_{F_{\theta}^*} / \mathfrak{h}_{F_{\theta}^*}^+.$
	
	Now, let $\overline{X} \in (\mathfrak{g}_{F_{\theta}^*} / \mathfrak{g}_{F_{\theta}^*}^{+} )^{\theta+1}.$ We have $\theta(X) +X \in \mathfrak{g}_{F_{\theta}^*}^+,$ so
$$ 2X_+ = \theta(X_+ + X_-) + (X_+ + X_-) \in \mathfrak{g}_{F_{\theta}^*}^+. $$
Thus, since $p \neq 2,$ we have $X_+ \in \mathfrak{g}_{F_{\theta}^*}^+.$ Thus, the coset $X_- + \mathfrak{p}_{F_{\theta}^*}^+$ is mapped to $X + \mathfrak{g}_{F_{\theta}^*}^+.$ The inclusion $\mathfrak{p}_{F_{\theta}^*} / \mathfrak{p}_{F_{\theta}^*}^+ \subset (\mathfrak{g}_{F_{\theta}^*} / \mathfrak{g}_{F_{\theta}^*}^+ )^{\theta+1}$ is clear.

\end{proof}

\begin{defn} Suppose $F_{\theta}^* \in \mathcal{F}_{\theta}(r),$ and $\mathcal{A}$ is an apartment in $\mathcal{B}(H).$ Define \end{defn}

$$ A(\mathcal{A}, F_{\theta}^*) := A(F_{\theta}^* \cap \mathcal{A}, \mathcal{A}). $$

\subsection{Standard lifts and $r$-associativity}

\begin{remark} Let $x \in \mathcal{B}(H),$ and let $F_{\theta}^*(x) \in \mathcal{F}_{\theta}(r).$ We call a generalized $r$-facet $F^*$ in $\mathcal{B}(G)$ the standard lift of $F_{\theta}^*(x)$ if $F^*$ is the generalized $r$-facet in $\mathcal{B}(G)$ containing $x,$ as defined in [\cite{d}, 3.2.1] and in Section 4.2. \end{remark}

\begin{lemma} \label{intersecth} Let $y \in \mathcal{B}(G).$ The generalized $r$-facet $F^*(y)$ is $\theta$-stable if and only if $F^*(y) \cap \mathcal{B}(H) \neq \emptyset.$ In particular, if $F_{\theta}^*(x) \in \mathcal{F}_{\theta}(r),$ then the standard lift $F^*(x)$ of $F_{\theta}^*(x)$ is $\theta$-stable. \end{lemma}

\begin{proof} $``\Leftarrow"$: Let $z \in F^*(x)$ with $x \in \mathcal{B}(H).$ We must verify that $\theta(z) \in F^*(x).$ This occurs if and only if 

$$ (\dagger) \textup{ } \mathfrak{g}_{\theta(z),r} = \mathfrak{g}_{x,r} \textup{ and } \mathfrak{g}_{\theta(z),r^+} = \mathfrak{g}_{x,r^+}. $$
Since $x$ lies in $\mathcal{B}(H),$ we have $\theta(x) = x,$ so $z \in F^*(\theta(x)).$ Thus, we have
$$ \mathfrak{g}_{z,r} = \mathfrak{g}_{\theta(x),r} \textup{ and } \mathfrak{g}_{z,r^+} = \mathfrak{g}_{\theta(x),r^+}$$
which, since $\theta$ is an involution, is equivalent to $(\dagger).$

$``\Rightarrow"$: Let $F^* = F^*(y),$ for some $y \in \mathcal{B}(G).$ By [\cite{d}, Lemma 3.2.11], $F^*(\delta)$ is a convex, closed, stab$_{G}(F^*)$-invariant set of $\mathcal{B}(G).$ Also, note that $\theta$ preserves the boundary of $F^*,$ and $\theta$ acts on $\mathcal{B}(G)$ by an isometry. Thus, if $z \in \overline{F^*} \backslash F^*,$ and $x \in F^*(\delta),$ we have 

$$ \textup{dist}(\theta(x), z) = \textup{dist}(\theta(x), \theta(z')) = \textup{dist}(x,z') \geq \delta, $$
for all $z' \in \overline{F^*} \backslash F^*.$ In particular, $F^*(\delta)$ is $\theta$-stable. Now, since $\langle \theta \rangle$ is a finite group of isometries, we apply [\cite{tits}, 2.3.1] to conclude that $\langle \theta \rangle$ has a fixed point in $F^*(\delta),$ and hence in $F^*.$
\end{proof}

The following proposition gives us a way to translate the work done in Sections 4.1 and 4.2 into the framework of $\theta$-stable $r$-facets. 

\begin{prop} \label{reduction} Let $F_{1,\theta}^*, F_{2,\theta}^* \in \mathcal{F}_{\theta}(r),$ and let $\mathcal{A} \subset \mathcal{A}'$ be apartments in $\mathcal{B}(H)$ and $\mathcal{B}(G),$ respectively, such that $F_{i,\theta}^* \cap \mathcal{A} \neq \emptyset,$ for $i=1,2.$ If $F_1^*$ and $F_2^*$ are the standard lifts of $F_{1,\theta}^*$ and $F_{2,\theta}^*,$ respectively, then 

$$ A(\mathcal{A}, F_{1,\theta}^*) = A(\mathcal{A}, F_{2,\theta}^*) \Leftrightarrow A(\mathcal{A}', F_{1}^*) = A(\mathcal{A}', F_{2}^*). $$
\end{prop}

\begin{proof} ``$\Leftarrow$ " : We first claim that $A(\mathcal{A}', F_1^*)^{\theta} = A(\mathcal{A}, F_{1,\theta}^*).$ If this is not true, then, since $A(\mathcal{A}', F_1^*)^{\theta}$ is convex, there is an $(r,\theta)$-facet $C$ in $A(\mathcal{A}', F_1^*)^{\theta}$ properly containing $F_{1,\theta}^* \cap \mathcal{A}$ in its closure. Let $C'$ be the $r$-facet in $\mathcal{A}'$ containing $C.$ Since $F_{1,\theta}^* \cap \mathcal{A} \subset \overline{C},$ we have $F_1^* \cap \mathcal{A}' \subset \overline{C'}.$ Note that $\overline{C'}$ is the union of $C'$ and $r$-facets of strictly smaller dimension. Thus, if dim $F_1'$ = dim $C',$ then we must have $F_1' = C'.$ By our choice of $C',$ dim $F_1' \neq$ dim $C'.$ In particular, $C'$ is a $\theta$-stable $r$-facet (contained in $A(\mathcal{A}', F_1^*)$ by Remark \ref{facetremark}) of strictly larger dimension than $F_1^* \cap \mathcal{A}',$ a contradiction.

``$\Rightarrow$": Since $A(\mathcal{A}, F_{1,\theta}^*) = A(\mathcal{A}, F_{2,\theta}^*),$ in particular, we know $F_1^* \cap \mathcal{A}'$ intersects $A(\mathcal{A}', F_2^*)$ nontrivially. Thus, by Remark \ref{facetremark}, $F_1^* \cap \mathcal{A}'$ is entirely contained in $A(\mathcal{A}', F_2^*).$ Similarly, $F_2^* \cap \mathcal{A}'$ is entirely contained inside $A(\mathcal{A}', F_1^*).$

\end{proof}

\begin{defn} \label{defstrong} Let $F_{1, \theta}^*, F_{2, \theta}^* \in \mathcal{F_{\theta}}(r).$ We say that $F_{1, \theta}^*$ and $F_{2, \theta}^*$ are strongly $r$-associated if for all apartments $\mathcal{A}$ in $\mathcal{B}(H)$ such that $F_{1, \theta}^* \cap \mathcal{A}$ and $F_{2, \theta}^* \cap \mathcal{A}$ are nonempty, we have \end{defn}

$$ A(\mathcal{A}, F_{1,\theta}^*) = A(\mathcal{A}, F_{2,\theta}^*). $$

\begin{remark} \label{Gassoc} By [\cite{d}, Lemma 3.3.3] and Proposition \ref{reduction}, this is the same as demanding that the standard lifts $F_1^*, F_2^*$ be strongly $r$-associated in the sense of [\cite{d}, Definition 3.3.2]. \end{remark}

\begin{lemma} \label{assocalt} Two generalized $(r,\theta)$-facets $F_{1,\theta}^*, F_{2, \theta}^* \in \mathcal{F}_{\theta}(r)$ are strongly $r$-associated if and only if there exists an apartment $\mathcal{A} \subset \mathcal{B}(H)$ for which $F_{1,\theta}^* \cap \mathcal{A}, F_{2,\theta}^* \cap \mathcal{A}$ are nonempty, and

$$ A(\mathcal{A}, F_{1,\theta}^*) = A( \mathcal{A}, F_{2,\theta}^*). $$ 
 \end{lemma}

\begin{proof} $``\Rightarrow":$ The definition of strong $r$-associativity proves the forward implication. \newline
$``\Leftarrow":$ Let $\tilde{\mathcal{A}}$ be an apartment for which $A(\tilde{\mathcal{A}}, F_{1,\theta}^*) = A(\tilde{\mathcal{A}}, F_{2,\theta}^*) \neq \emptyset.$ By [\cite{d}, Lemma 3.3.3] and Proposition \ref{reduction}, we have that the standard lifts $F_1^*$ and $F_2^*$ (of $F_{1,\theta}^*$ and $F_{2,\theta}^*$ respectively) are strongly $r$-associated. By Proposition \ref{reduction}, the result follows by taking $\theta$-fixed points.

\end{proof}

\begin{defn} Two generalized $(r,\theta)$-facets $F_{1,\theta}^*, F_{2,\theta}^* \in \mathcal{F}_{\theta}(r)$ are said to be $r$-associated if there exists an $h \in H$ such that $F_{1,\theta}^*$ and $hF_{2,\theta}^*$ are strongly $r$-associated. \end{defn}

\begin{lemma} \label{assocequiv} $r$-associativity is an equivalence relation on $\mathcal{F}_{\theta}(r).$ Whenever two generalized $(r,\theta)$-facets $F_{1,\theta}^*, F_{2,\theta}^*$ are $r$-associated, we write $F_{1,\theta}^* \sim F_{2,\theta}^*.$ \end{lemma}

\begin{proof} For reflexivity, let $F_{\theta}^* \in \mathcal{F}_{\theta}(r).$ Suppose $\mathcal{A} \subset \mathcal{B}(H)$ such that $\mathcal{A} \cap F_{\theta}^* \neq \emptyset.$ We have $A(\mathcal{A}, F_{\theta}^*) = A(\mathcal{A}, F_{\theta}^*),$ so the relation is reflexive. 

	Now, suppose $F_{1,\theta}^*  \sim F_{2,\theta}^*.$ Then, there exists an apartment $\mathcal{A} \subset \mathcal{B}(H)$ and an element $h \in H$ such that

$$ A(\mathcal{A}, F_{1,\theta}^*) = A(\mathcal{A}, hF_{2,\theta}^*) \neq \emptyset. $$ 
Recall $hA(\mathcal{A}, F_{\theta}^*) = A(h\mathcal{A}, hF_{\theta}^*)$ for any $h \in H.$ Thus, multiplying the equation above by $h^{-1},$ we obtain

$$A(h^{-1}\mathcal{A}, h^{-1}F_{1,\theta}^*) = A(h^{-1}\mathcal{A}, F_{2,\theta}^*) \neq \emptyset.$$
In particular, $F_{2,\theta}^* \sim F_{1,\theta}^*.$

	Now, suppose $F_{1,\theta}^*, F_{2,\theta}^*,$ and $F_{3,\theta}^*$ are generalized $(r,\theta)$-facets such that $F_{1,\theta}^* \sim F_{2,\theta}^*$ and $F_{2,\theta}^* \sim F_{3,\theta}^*.$ Then, by definition, there exist $h_2, h_3 \in H$ and apartments $\mathcal{A}_{12}, \mathcal{A}_{23} \subset \mathcal{B}(H)$ such that 

$$  A(\mathcal{A}_{12}, F_{1,\theta}^*) = A(\mathcal{A}_{12}, h_2F_{2,\theta}^*) \neq \emptyset $$
and
$$ A(\mathcal{A}_{23}, F_{2,\theta}^*) = A(\mathcal{A}_{23}, h_3F_{3,\theta}^*) \neq \emptyset. $$
Let $z \in C(F_{2,\theta}^*).$ Then, since $h_2^{-1}\mathcal{A}_{12} \cap F_{2,\theta}^* \neq \emptyset$ and $\mathcal{A}_{23} \cap F_{2,\theta}^* \neq \emptyset,$ $z$ lies in $h_2^{-1}\mathcal{A}_{12} \cap \mathcal{A}_{23}.$ By Remark \ref{patr}, there exists some $h \in H_z \subset \textup{stab}_H(F_{2,\theta}^*)$ (since $h$ fixes $z$) such that $hh_2^{-1}\mathcal{A}_{12} = \mathcal{A}_{23}.$ 

	Using these facts, we have

\begin{eqnarray*} 
\emptyset \neq A(\mathcal{A}_{12}, F_{1,\theta}^*) &=& A(\mathcal{A}_{12}, h_2F_{2,\theta}^*) = h_2A(h_2^{-1}\mathcal{A}_{12}, F_{2,\theta}^*) \\
&=& h_2h^{-1}A(\mathcal{A}_{23}, hF_{2,\theta}^*) = h_2h^{-1}A(\mathcal{A}_{23}, F_{2,\theta}^*) \\
&=& A(\mathcal{A}_{12}, h_2h^{-1}h_3F_{3,\theta}^*). 
\end{eqnarray*}
\end{proof}

\subsection{Identification of some $\mathfrak{f}$-vector spaces }

\begin{defn} As in [\cite{d}, 3.4.1], for $x \in \mathcal{B}(G),$ let $V_{x,r}$ denote the $\mathfrak{f}$-vector space $\mathfrak{g}_{x,r} / \mathfrak{g}_{x,r^+}.$

\end{defn}

\begin{defn} If $F_{\theta}^* \in \mathcal{F}_{\theta}(r)$ and $x \in F_{\theta}^*,$ define $V_{F_{\theta}^*}:=V_{x,r}.$

\end{defn}

Using Lemma \ref{thstcoset}, we identify $V_{x,r}^+$ with the image of $\mathfrak{h}_{x,r}/\mathfrak{h}_{x,r^+}.$ In particular, we interpret the quotient map given by $\mathfrak{h}_{x,r} \rightarrow V_{x,r}^+$ below using this identification.

\begin{lemma} \label{assocident} Suppose $F_{1, \theta}^*, F_{2, \theta}^* \in \mathcal{F}_{\theta}(r)$ are strongly $r$-associated. Then the natural maps \end{lemma}

$$ \mathfrak{h}_{F_{1, \theta}^*} \cap \mathfrak{h}_{F_{2, \theta}^*} \rightarrow V_{F_{i, \theta}^*}^+ $$

$$ \mathfrak{p}_{F_{1, \theta}^*} \cap \mathfrak{p}_{F_{2, \theta}^*} \rightarrow V_{F_{i, \theta}^*}^- $$ 
are surjective with kernels $\mathfrak{h}_{F_{1,\theta}^*}^+  \cap \mathfrak{h}_{F_{2,\theta}^*} = \mathfrak{h}_{F_{1,\theta}^*} \cap \mathfrak{h}_{F_{2,\theta}^*}^+ = \mathfrak{h}_{F_{1, \theta}^*}^+ \cap \mathfrak{h}_{F_{2, \theta}^*}^+ $ and $\mathfrak{p}_{F_{1,\theta}^*}^+ \cap \mathfrak{p}_{F_{2,\theta}^*} = \mathfrak{p}_{F_{1,\theta}^*} \cap \mathfrak{p}_{F_{2,\theta}^*}^+ = \mathfrak{p}_{F_{1, \theta}^*}^+ \cap \mathfrak{p}_{F_{2, \theta}^*}^+,$ respectively.

\begin{proof} By Remark \ref{Gassoc}, we know that the standard lifts $F_1^*$ and $F_2^*$ are strongly $r$-associated. Let $e \in V_{F_{i,\theta}^*}^+.$ By [\cite{d}, 3.5.1], there is a lift $X \in \mathfrak{g}_{F_{1,\theta}^*} \cap \mathfrak{g}_{F_{2,\theta}^*}.$ Let $X_+$ denote the projection of $X$ to $\mathfrak{h}_{F_{i, \theta}^*}.$ Then, by the proof of Lemma \ref{thstcoset}, $X_+$ is mapped to $e.$ By Lemma \ref{mplattice}, $X_+$ lies in $\mathfrak{h} \cap \mathfrak{g}_{F_{1,\theta}^*} \cap \mathfrak{g}_{F_{2,\theta}^*} = \mathfrak{h}_{F_{1,\theta}^*} \cap \mathfrak{h}_{F_{2,\theta}^*}.$ Thus, the map is surjective. If $X$ lies in the kernel of the first map, then, by [\cite{d}, 3.5.1], $X$ is contained in $\mathfrak{g}_{F_{1,\theta}^*}^+ \cap \mathfrak{g}_{F_{2,\theta}^*} = \mathfrak{g}_{F_{1,\theta}^*} \cap \mathfrak{g}_{F_{2,\theta}^*}^+ = \mathfrak{g}_{F_{1,\theta}^*}^+ \cap \mathfrak{g}_{F_{2,\theta}^*}^+.$ Thus, the kernel is $\mathfrak{h} \cap \mathfrak{g}_{F_{1,\theta}^*}^+ \cap \mathfrak{g}_{F_{2,\theta}^*} = \mathfrak{h}_{F_{1,\theta}^*} \cap \mathfrak{h}_{F_{2,\theta}^*}^+ = \mathfrak{h}_{F_{1,\theta}^*}^+ \cap \mathfrak{h}_{F_{2,\theta}^*}^+.$ 
	
	Similarly, let $X \in \mathfrak{p}_{F_{1,\theta}^*} \cap \mathfrak{p}_{F_{2,\theta}^*},$ and suppose $X$ is mapped to the trivial coset in $V_{F_{i,\theta}^*}^-.$ Then, again by [\cite{d}, 3.5.1], we have $X \in \mathfrak{p}_{F_{1,\theta}^*}^+ \cap \mathfrak{p}_{F_{2,\theta}^*} = \mathfrak{p}_{F_{1,\theta}^*} \cap \mathfrak{p}_{F_{2,\theta}^*}^+ = \mathfrak{p}_{F_{1,\theta}^*}^+ \cap \mathfrak{p}_{F_{2,\theta}^*}^+.$ Let $e$ be a coset in $V_{F_{i,\theta}^*}^-.$ By the same result, there exists a lift $X \in \mathfrak{g}_{F_{1,\theta}^*} \cap \mathfrak{g}_{F_{2,\theta}^*}.$ By Proposition \ref{thstpa}, we may project $X$ to $X_- \in \mathfrak{p}_{F_{1,\theta}^*}.$ This is the desired lift which lies in $\mathfrak{p}_{F_{1,\theta}^*} \cap \mathfrak{p}_{F_{2,\theta}^*}.$ 
\end{proof}

\begin{remark} Due to the previous result, whenever $F_{1, \theta}^*$ and $F_{2,\theta}^*$ are strongly $r$-associated, we are able to identify $V_{F_{1, \theta}^*}^+$ with $V_{F_{2,\theta}^*}^+$ and $V_{F_{1, \theta}^*}^-$ with $V_{F_{2,\theta}^*}^-.$ We let $i^+$ and $i^-$ denote the respective bijective identifications.\end{remark}

\begin{defn} If $F_{\theta}^* \in \mathcal{F}_{\theta}(r)$ and $x \in F_{\theta}^*,$ then the image of $H_x$ in $\textup{Aut}_{\mathfrak{f}}(V_{F_{\theta}^*}^-)$ is denoted by $N_x^-(F_{\theta}^*).$ \end{defn}

\begin{lemma} \label{normassoc} Suppose $F_{i,\theta}^* \in \mathcal{F}_{\theta}(r)$ and $x_i \in F_{i,\theta}^*$ for $i=1,2.$ If $F_{1,\theta}^*$ and $F_{2,\theta}^*$ are strongly $r$-associated, then $N_{x_i}^-(F_{i,\theta}^*)$ is the image of $H_{x_1} \cap H_{x_2}$ in $\textup{Aut}_{\mathfrak{f}}(V_{F_{i, \theta}^*}^-)$ for $i=1,2.$ Moreover, \end{lemma}

$$ N_{x_1}^-(F_{1,\theta}^*) = N_{x_2}^-(F_{2,\theta}^*) $$
under the identification induced by $i^-.$

\begin{proof} Let $\mathcal{A}$ be an apartment in $\mathcal{B}(H)$ containing $x_1$ and $x_2.$ Choose $\psi \in \Psi(\mathcal{A})$ such that the image of $U_{\psi}$ in Aut$_{\mathfrak{f}}(V_{F_{1,\theta}^*}^-)$ is nontrivial and $\psi(x_1) = 0.$ We will show that $\psi(x_2) = 0.$
	
	Suppose $\psi(x_2) > 0.$ Since the image of $U_{\psi}$ in Aut$_{\mathfrak{f}}(V_{F_{1,\theta}^*}^-)$ is nontrivial, using the identification from Lemma \ref{assocident}, there exists some $h \in U_{\psi}$ and $X \in \mathfrak{p}_{x_1,r} \cap \mathfrak{p}_{x_2,r}$ such that 
	
	$$ {}^{h}X \neq X \textup{ mod } \mathfrak{p}_{x_1,r^+} \cap \mathfrak{p}_{x_2,r^+}. $$
On the other hand, we have ${}^{h}X - X \in \mathfrak{p}_{x_2,r^+},$ so by Lemma \ref{assocident}, we have 

$$ {}^{h}X - X \in \mathfrak{p}_{x_1,r} \cap \mathfrak{p}_{x_2,r^+} = \mathfrak{p}_{x_1,r^+} \cap \mathfrak{p}_{x_2,r^+}. $$
This is a contradiction, so we must have $\psi(x_2) \leq 0.$ 

	Suppose $\psi(x_2) < 0.$ Since $x_1$ and $x_2$ lie in an affine space, we regard $v=x_2-x_1$ as a vector. Consider the function 
	
	$$ f_v : \mathbb{R} \rightarrow \mathbb{R} $$
defined by $\epsilon \mapsto \psi(x_1 + \epsilon v),$ where $x_1+\epsilon v$ is interpreted as the point $z \in \mathcal{A}$ for which $z-x_1 = \epsilon v.$  For all $\epsilon \in \mathbb{R},$ we have $x_1+\epsilon v \in \mathcal{A},$ so $f_v$ is well-defined. Since $\psi(x_2) < 0,$ we have 

$$ f_v(1) = \psi(x_2) < 0,$$
so, since $\psi$ is continuous, we must have $f_v(\epsilon) < 0$ whenever $\epsilon > 0,$ and similarly $f_v(\epsilon) > 0$ whenever $\epsilon < 0.$ Recall that $A(\mathcal{A}, F_{1,\theta}^*) = A(\mathcal{A}, F_{2,\theta}^*)$ by definition of strong $r$-associativity. Since $x_1,x_2 \in A(\mathcal{A},F_{2,\theta}^*) = A(\mathcal{A}, F_{1,\theta}^*),$ we must have $x_1 + \mathbb{R}v \subset A(\mathcal{A}, F_{1,\theta}^*).$ Thus, since $F_{1,\theta}^* \cap \mathcal{A}$ is open in $A(\mathcal{A}, F_{1,\theta}^*) = A(\mathcal{A}, F_{2,\theta}^*),$ there is some $\epsilon < 0$ for which $x_1+\epsilon v \in F_{1,\theta}^* \cap \mathcal{A}.$ In particular, by Lemma \ref{facetpa}, we have $\mathfrak{p}_{x_1+ \epsilon v, r} = \mathfrak{p}_{x_1, r}$ and $\mathfrak{p}_{x_1+ \epsilon v, r^+} = \mathfrak{p}_{x_1, r^+}.$ Moreover, $\psi(x_1+\epsilon v) = f_v(\epsilon) > 0,$ so we have $U_{\psi} \subset H_{x_1+\epsilon v,0^+}$ Thus, $h \in U_{\psi}$ acts trivially on $\mathfrak{p}_{x_1+\epsilon v,r}/\mathfrak{p}_{x_1+\epsilon v, r^+}  = V_{F_{1,\theta}^*}^-,$ a contradiction. 
	
	We have thus shown that $\psi(x_2) = 0.$ If $\mathcal{A}$ corresponds to a maximal $k$-split torus $\textbf{S}$ of $\textbf{H},$ recall that $\textbf{S}$ lies inside a maximal $k$-torus $\textbf{Z}$ as described in Section 2.3. Since the image of $H_{x_1}$ is determined by a filtration subgroup of $\textbf{Z}$ (which also lies in $H_{x_2}$) and the $U_{\psi}$'s, the proof shows that if $h \in H_{x_1}$ has nontrivial image in Aut$_{\mathfrak{f}}(V_{F_{1,\theta}^*}^-),$ then there exists some $h' \in H_{x_1} \cap H_{x_2}$ for which the images of $h$ and $h'$ in $\textup{Aut}_{\mathfrak{f}}(V_{F_{i, \theta}^*}^-)$ coincide. 
\end{proof}

\begin{defn} Let $F_{\theta}^* \in \mathcal{F}_{\theta}(r)$ and $ x \in F_{\theta}^*.$ Define $N^-(F_{\theta}^*) \subset \textup{Aut}_{\mathfrak{f}}(V_{F_{\theta}^*}^-)$ by \end{defn}

$$ N^-(F_{\theta}^*) = N_x^-(F_{\theta}^*). $$

\subsection{An equivalence relation}

\begin{defn}

$$ I_r := \{ (F_{\theta}^*, v) \mid F_{\theta}^* \in \mathcal{F}_{\theta}(r) \textup{ and } v \in V_{F_{\theta}^*}^- \} $$
\end{defn}

Let $x \in \mathcal{B}(H).$ For $v \in V_{x,r},$ we interpret ${}^{h}v$ as the image of ${}^{h}X,$ where $X$ is a lift of $v$ in $\mathfrak{g}_{x,r}.$

\begin{defn} \label{anequiv} For $(F_{1,\theta}^*, v_1)$ and $(F_{2,\theta}^*,v_2)$ in $I_r,$ we write $(F_{1,\theta}^*,v_1) \sim (F_{2,\theta}^*,v_2)$ provided that there exists some $h \in H$ and an apartment $\mathcal{A} \subset \mathcal{B}(H),$ for which $F_{1,\theta}^* \cap \mathcal{A}, F_{2,\theta}^* \cap \mathcal{A}$ are nonempty, and \end{defn}

\begin{enumerate}

\item $A(\mathcal{A}, F_{1,\theta}^*) = A(\mathcal{A}, hF_{2,\theta}^*)  \textup{ and } $
\item $v_1= {}^{h}v_2 \textup{ in } V_{F_{1,\theta}^*}^- = V_{hF_{2, \theta}^*}^- ,$

\end{enumerate}
where we use the usual identification from Lemma \ref{assocident} for the second condition.

\begin{lemma} \label{equivI} The relation defined above is an equivalence relation on $I_r.$\end{lemma}

\begin{proof} For reflexivity, let $h=1.$ 
	
	Now, suppose $(F_{1,\theta}^*,v_1) \sim (F_{2,\theta}^*,v_2).$ By definition, there exists an apartment $\mathcal{A} \subset \mathcal{B}(H)$ and an element $h \in H$ such that

 \begin{enumerate}

\item $A(\mathcal{A}, F_{1,\theta}^*) = A(\mathcal{A}, hF_{2,\theta}^*) \neq \emptyset \textup{ and } $
\item $v_1= {}^{h}v_2 \textup{ in } V_{F_{1,\theta}^*}^- = V_{hF_{2, \theta}^*}^- .$

\end{enumerate}
Since $h^{-1}A(\mathcal{A}, F_{1,\theta}^*) = A(h^{-1}\mathcal{A}, h^{-1}F_{2,\theta}^*),$ we have 

 \begin{enumerate}

\item $A(h^{-1}\mathcal{A}, h^{-1}F_{1,\theta}^*) = A(h^{-1}\mathcal{A}, F_{2,\theta}^*) \neq \emptyset \textup{ and } $
\item ${}^{h^{-1}}v_1= v_2 \textup{ in } V_{h^{-1}F_{1,\theta}^*}^- = V_{F_{2, \theta}^*}^- .$

\end{enumerate}
In particular, $(F_{2,\theta}^*,v_2) \sim (F_{1,\theta}^*,v_1).$ 

	For transitivity, suppose $(F_{1,\theta}, v_1), (F_{2,\theta}^*, v_2), (F_{3,\theta}^*,v_3) \in I_r$ such that $(F_{1,\theta}^*,v_1) \sim (F_{2,\theta}^*,v_2)$ and $(F_{2,\theta}^*,v_2) \sim (F_{3,\theta}^*,v_3).$ By definition, there exist $h_2,h_3 \in H$ and apartments $\mathcal{A}_{12}, \mathcal{A}_{23} \subset \mathcal{B}(H)$ such that 
	
\begin{eqnarray*}
 A(\mathcal{A}_{12}, F_{1,\theta}^*) &=& A(\mathcal{A}_{12}, h_2F_{2,\theta}^*) \neq \emptyset \\
 A(\mathcal{A}_{23}, F_{2,\theta}^*) &=& A(\mathcal{A}_{23}, h_3F_{3,\theta}^*) \neq \emptyset
 \end{eqnarray*} 
 and 
\begin{eqnarray*}
v_1 &=& {}^{h_2}v_2 \textup{ in } V_{F_{1,\theta}^*}^- = V_{h_2F_{2,\theta}^*}^- \\
v_2 &=& {}^{h_3}v_3 \textup{ in } V_{F_{2,\theta}^*}^- = V_{h_3F_{3,\theta}^*}^-. 
\end{eqnarray*}

Fix $x_i \in C(F_{i,\theta}^*).$ The first and second lines show that $A(\mathcal{A}_{12}, h_2F_{2,\theta}^*), A(h_2\mathcal{A}_{23}, h_2F_{2,\theta}^*) \neq \emptyset.$ In particular, $\mathcal{A}_{12} \cap h_2F_{2,\theta}^*, h_2\mathcal{A}_{23} \cap h_2F_{2,\theta}^* \neq \emptyset,$ so there exists an element $h \in H_{h_2x_2} \subset \textup{stab}_H(h_2F_{2,\theta}^*)$ such that $h\mathcal{A}_{12} = h_2\mathcal{A}_{23}.$ We have 

\begin{eqnarray*} 
\emptyset \neq A(\mathcal{A}_{12},F_{1,\theta}^*) &=& A(\mathcal{A}_{12},h_2F_{2,\theta}^*) = h^{-1}A(h\mathcal{A}_{12},hh_2F_{2,\theta}^*) \\
&=& h^{-1}A(h_2\mathcal{A}_{23},h_2F_{2,\theta}^*) = h^{-1}h_2A(\mathcal{A}_{23},h_3F_{3,\theta}^*) \\
&=& A(\mathcal{A}_{12}, h^{-1}h_2h_3F_{3,\theta}^*). 
\end{eqnarray*}
Now, by the proof of [\cite{d}, 3.5.1] and Lemma \ref{assocident}, we have a surjection
$$ \mathfrak{p}_{F_{1,\theta}^*} \cap \mathfrak{p}_{h_2F_{2,\theta}^*} \cap \mathfrak{p}_{h^{-1}h_2h_3F_{3,\theta}^*} \rightarrow V_{F_{1,\theta}^*}^-. $$
As a result, there is some $X \in  \mathfrak{p}_{F_{1,\theta}^*} \cap \mathfrak{p}_{h_2F_{2,\theta}^*} \cap \mathfrak{p}_{h^{-1}h_2h_3F_{3,\theta}^*}$ such that the image of $X$ in $V_{F_{1,\theta}^*}^-$ is $v_1.$ Since ${}^{h_2}v_2 = v_1$ under the standard identification, we have that the image of $X$ in $V_{h_2F_{2,\theta}^*}^-$ is ${}^{h_2}v_2,$ so the image of ${}^{h_2^{-1}}X$ in $V_{F_{2,\theta}^*}^-$ is $v_2.$ Recall that $h \in H_{h_2x_2} = \textup{Int}(h_2)H_{x_2},$ so $h_2^{-1}hh_2 \in H_{x_2}.$ Thus, the image of ${}^{h_2^{-1}h}X = {}^{(h_2^{-1}hh_2)h_2^{-1}}X$ in $V_{F_{2,\theta}^*}^-$ is ${}^{h_2^{-1}hh_2}v_2.$  By the previous computation and the fact that $X$ was chosen in $\mathfrak{p}_{h^{-1}h_2h_3F_{3,\theta}^*},$ we have ${}^{h_2^{-1}h}X \in \mathfrak{p}_{F_{2,\theta}^*} \cap \mathfrak{p}_{h_3F_{3,\theta}^*}.$ Since $F_{2,\theta}^*$ and $h_3F_{3,\theta}^*$ are strongly $r$-associated, by Lemma \ref{normassoc}, $N^-(F_{2,\theta}^*) = N^-(h_3F_{3,\theta}^*).$ Thus, again by Lemma \ref{normassoc}, there is an element $h' \in H_{h_3x_3} \cap H_{x_2}$ such that 

$$ {}^{h_2^{-1}hh_2}v_2 = {}^{h'}v_2 = {}^{h'h_3}v_3 \textup{ in } V_{F_{2,\theta}^*}^- = V_{h_3F_{3,\theta}^*}^-. $$
As a consequence, the image of $X$ in $V_{h^{-1}h_2h_3F_{3,\theta}^*}^-$ is
 
 $$ {}^{h^{-1}h_2h'h_3}v_3 = {}^{h^{-1}(h_2h'h_2^{-1})h_2h_3}v_3 = {}^{h''h^{-1}h_2h_3}v_3 $$
 with $h'' \in \textup{Int}(h^{-1}h_2)(H_{h_3x_3} \cap H_{x_2}) \subset H_{h^{-1}h_2h_3x_3}.$ Since $h'' \in H_{h^{-1}h_2h_3x_3} \subset  \textup{stab}_H(h^{-1}h_2h_3F_{3,\theta}^*) ,$ we have
 
$$ A(\mathcal{A}_{12}, F_{1,\theta}^*) = A(\mathcal{A}_{12}, h^{-1}h_2h_3F_{3,\theta}^*) = A(\mathcal{A}_{12}, h''h^{-1}h_2h_3F_{3,\theta}^*) \neq \emptyset. $$
Moreover, 
 
 $$ v_1 = {}^{h_2}v_2 = {}^{h''h^{-1}h_2h_3}v_3 \textup{ in } V_{F_{1,\theta}^*}^- = V_{h''h^{-1}h_2h_3F_{3,\theta}^*}^-. $$
 This shows that $(F_{1,\theta}^*, v_1) \sim (F_{3,\theta}^*, v_3).$
\end{proof}

\section{Jacobson-Morosov triples over $\mathfrak{f}$ and $k$}

Fix $r \in \mathbb{R}.$ Before attaching a nilpotent $H$-orbit to the types of pairs discussed at the end of Section 4, we will need a way to pass from $\mathfrak{sl}_2(\mathfrak{f})$-triples to $\mathfrak{sl}_2(k)$-triples and vice versa. In this section, we describe this procedure in detail. Recall that $\mathcal{N}$ denotes the set of nilpotent elements in $\mathfrak{g}$ as defined in the preliminaries. As in Lemma \ref{thstcoset}, we will identify $\mathfrak{p}_{x,r}/\mathfrak{p}_{x,r^+}$ with $V_{x,r}^-.$

\begin{defn} Let $F_{\theta}^* \in \mathcal{F}_{\theta}(r).$ An element $e \in V_{F_{\theta}^*}^-$ is called degenerate provided that there exists a lift $E \in \mathfrak{p}_{F_{\theta}^*} \cap \mathcal{N}.$ \end{defn}
The following lemma gives us an alternate characterization of degenerate elements.
\begin{lemma} \label{deglemma} Fix $F_{\theta}^* \in \mathcal{F}_{\theta}(r).$ An element $e \in V_{F_{\theta}^*}^-$ is degenerate if and only if zero lies in the Zariski closure of ${}^{\textsf{H}_x}e$ for all $x \in F_{\theta}^*.$ \end{lemma}

\begin{proof} $``\Rightarrow "$: We refer to [\cite{moy-prasad}, Proposition 4.3]. Fix $x \in F_{\theta}^*$ and a lift $E \in \mathfrak{p}_{x,r} \cap \mathcal{N}.$ In the notation of [\cite{moy-prasad}, Proposition 4.3], we take $V = \mathfrak{p}_{x,r}, W=\mathfrak{p}_{x,r^+},$ and let $\rho: H_x \rightarrow GL(V)$ be the rational representation given by the adjoint action of $H_x$ on $\mathfrak{p}$ restricted to the lattice $\mathfrak{p}_{x,r}.$ We note that $\varpi \mathfrak{p}_{x,r} = \mathfrak{p}_{x,r+1} \subset \mathfrak{p}_{x,r^+},$ and $E$ is nilpotent lift of $e,$ so all hypotheses are satisfied. From [\cite{moy-prasad}, Proposition 4.3], we conclude that zero lies in the Zariski closure of ${}^{\textsf{H}_x}e,$ with respect to the induced representation of $\rho$ from $\textsf{H}_x$ to $GL(V/W).$
	
	$``\Leftarrow"$: Fix $x \in F_{\theta}^*.$ Let $\textbf{S}$ be a maximal $k$-split torus in $\textbf{H}$ with $x \in \mathcal{A}(\textbf{S},k).$ We consider $V_{F_{\theta}^*}^-$ as the vector space of $\mathfrak{f}$-rational points of the affine $\textsf{H}_x$-scheme $\textup{Lie}(\textsf{G}_x)^-.$ Then, by [\cite{kempf}, Theorem 1.4], there exists a one-parameter subgroup $\overline{\lambda} \in \textbf{X}_{*}^{\mathfrak{f}}(\textsf{H}_x)$ such that 
	
$$ \lim_{t \rightarrow 0} {}^{\overline{\lambda}(t)}e = 0.$$
Let $\textsf{S}$ be the maximal $\mathfrak{f}$-split torus in $\textsf{H}_x$ corresponding to $\textbf{S}.$ Then, since $\textsf{H}_x(\mathfrak{f})$ acts transitively on the set of maximal $\mathfrak{f}$-split tori in $\textsf{H}_x$, there exists an element $\overline{h} \in \textsf{H}_x(\mathfrak{f})$ and a one-parameter subgroup $\overline{\mu} \in \textbf{X}_{*}^{\mathfrak{f}}(\textsf{S})$ such that 

$$ \lim_{t \rightarrow 0} {}^{\overline{\mu}(t)\overline{h}}e = 0. $$
Let $\mu \in \textbf{X}_{*}(\textbf{S})$ be a lift of $\overline{\mu}$ and let $h \in H_x$ be a lift of $\overline{h}.$ Also, let $E'$ be a lift of $e$ in $\mathfrak{p}_{x,r}.$ Without loss of generality, we may assume that 

$$ {}^{h}E' = \sum_{\psi} X_{\psi}$$
where $X_{\psi} \in \mathfrak{g}_{\psi} \cap \mathfrak{p}$ and $\psi(x) = r.$ We claim that $\psi(x+\epsilon \cdot \mu) > r$ for all $\psi $ appearing in the sum. This will happen precisely when $\langle \mu, \dot{\psi} \rangle > 0$ since 

$$ \psi(x+\epsilon\mu) = \psi(x) + \epsilon \cdot \langle \mu, \dot{\psi} \rangle = r + \epsilon \cdot \langle \mu, \dot{\psi} \rangle.$$
 Note, however, that 

$$ 0 = \lim_{t \rightarrow 0} {}^{\overline{\mu}(t)\overline{h}}\overline{E'} = \sum \lim_{t \rightarrow 0} {}^{\overline{\mu}(t)}\overline{X_{\psi}} = \sum \lim_{t \rightarrow 0} t^{\langle \mu, \dot{\psi} \rangle} \overline{X_{\psi}} $$
so, in particular, the limit is 0 if and only if $\langle \mu, \dot{\psi} \rangle > 0.$ This shows that ${}^{h}E' \in \mathfrak{p}_{x+\epsilon\mu,r^+}.$  For $\epsilon$ sufficiently small, $x+\epsilon\mu$ lies in a generalized $(r,\theta)$-facet $C_{\theta}^*$ containing $F_{\theta}^*$ in its closure. By [\cite{d}, Corollary 3.2.19], we have $\mathfrak{p}_{x,r^+} \subset \mathfrak{p}_{x+\epsilon \mu,r^+}.$ Thus,

$$ {}^{h}(E' + \mathfrak{p}_{x,r^+}) = {}^{h}(E' + \mathfrak{p}_{F_{\theta}^*}^+) \subset \mathfrak{p}_{x+ \epsilon \cdot \mu, r^+} $$
for $\epsilon$ taken to be sufficiently small. We have shown that the coset $E' + \mathfrak{p}_{x,r}$ lies in $\mathfrak{g}_{r^+}$ as defined in [\cite{adler-debacker}, 3.2.5]. In particular, by [\cite{adler-debacker}, Corollary 3.2.6], $e$ is a degenerate coset. 

\end{proof}

In order to discuss $\mathfrak{sl}_2(\mathfrak{f})$-triples, we next introduce an $\mathfrak{f}$-Lie algebra $\overline{\mathfrak{g}}_x$ which is associated to a point $x \in \mathcal{B}(H).$ In the preliminaries, we chose a uniformizer $\varpi$ for $k,$ which allows us to identify $V_{x,s}$ with $V_{x,s+j\cdot \ell}$ where $L$ is the splitting field of $\textbf{G}$ containing $K,$ and $\ell = [L:K].$ Using this identification, we define

$$ \overline{\mathfrak{g}}_{x} := \bigoplus_{s \in \R /\ell \cdot \Z} V_{x,s}. $$
If $\overline{X}_s \in V_{x,s}$ and $\overline{X}_t \in V_{x,t},$ then define $[\overline{X_s}, \overline{X_t}]$ to be the image of $[X_s,X_t] \in \mathfrak{g}_{x,(s+t)}$ in $V_{x,(s+t)}$ where $X_s \in \mathfrak{g}_{x,s}$ and $X_t \in \mathfrak{g}_{x,t}$ are any lifts of $\overline{X_s}$ and $\overline{X_t}$ respectively. We can then linearly extend to obtain a well-defined bracket on all of $\overline{\mathfrak{g}}_x.$ With this product, $\overline{\mathfrak{g}}_x$ is an $\mathfrak{f}$-Lie algebra. 

\subsection{Some hypotheses}

We now list some hypotheses (which occur also in [\cite{d}]) needed in order to utilize the theory of $\mathfrak{sl}_2$-triples and pass from the Lie algebra setting to the group setting when necessary. These hypotheses hold under mild restrictions on $\textbf{G}, \textbf{H}$ and $k,$ and we give some references for more details on when each hypothesis is valid. It should be noted that in characteristic 0, all hypotheses hold. 

\begin{hyp} \label{hyp1} Suppose $x \in \mathcal{B}(H).$ If $X \in \mathcal{N} \cap (\mathfrak{p}_{x,r} \backslash \mathfrak{p}_{x,r^+}),$ then there exist $H \in \mathfrak{h}_{x,0}$ and $Y \in \mathfrak{p}_{x,-r}$ such that \end{hyp}

\begin{eqnarray*}
{} [H,X]  &=& 2X \textup{ } mod \textup{ } \mathfrak{p}_{x,r^+} \\
{} [H,Y]  &=& -2Y \textup{ } mod \textup{ } \mathfrak{p}_{x,(-r)^+} \\
{} [X,Y]  &=& H \textup{ } mod \textup{ } \mathfrak{h}_{x,0^+}. 
\end{eqnarray*}

If $(f,h,e)$ denotes the image of $(Y,H,X)$ in $V_{x,-r} \times V_{x,0} \times V_{x,r} \subset \overline{\mathfrak{g}}_x,$ then $\{ f,h,e \}$ is an $\mathfrak{sl}_2(\mathfrak{f})$-triple, and $\overline{\mathfrak{g}}_x$ decomposes into a direct sum of irreducible $\langle f,h,e \rangle$-modules of highest weight at most $p-3.$ Moreover, there exists some $\overline{\lambda} \in \textbf{X}_{*}^{\mathfrak{f}}(\textsf{H}_x),$ uniquely determined up to an element of $\textbf{X}_{*}(\textsf{Z}_x)$ whose differential is zero, such that the following hold:

\begin{enumerate}

\item The image of $d\overline{\lambda}$ in Lie($\textsf{H}_{x})$ coincides with the subspace spanned by $h.$

\item Suppose $i \in \Z.$ For $v \in \overline{\mathfrak{g}}_x$

\end{enumerate}

$$ \textup{ if } {}^{\overline{\lambda}(t)}v = t^iv, \textup{ then } |i| \leq p-3 \textup{ and } \textup{ad}(h)v = iv. $$
More details on Hypothesis \ref{hyp1} can be found in [\cite{d}, Appendix A].

	Following [\cite{kostant-rallis}, I.2], we define a \emph{normal} $\mathfrak{sl}_2$-triple below.

\begin{defn} Let $\{ Y, H, X \}$ \textup{(}resp. $\{f,h,e\}$\textup{)} be an $\mathfrak{sl}_2(k)$-triple in $\mathfrak{g}$ \textup{(}resp. $\mathfrak{sl}_2(\mathfrak{f})$-triple in $\overline{\mathfrak{g}}_x$\textup{)}. We call $\{ Y, H, X \}$ \textup{(}resp.$\{f,h,e\}$\textup{)} a normal $\mathfrak{sl}_2(k)$-triple \textup{(}resp. $\mathfrak{sl}_2(\mathfrak{f})$-triple\textup{)} provided that $X, Y \in \mathfrak{p}$ \textup{(}resp. $e,f \in (\overline{\mathfrak{g}}_x)^-$\textup{)} and $H \in \mathfrak{h}$ \textup{(}resp. $h \in (\overline{\mathfrak{g}}_x)^+\textup{)}.$ \end{defn}

\begin{remark} \label{normal} We note that if $\{f,h,e\}$ is any $\mathfrak{sl}_2(\mathfrak{f})$-triple in $\overline{\mathfrak{g}}_x$ with $e \in V_{x,r}^-,$ then it is normal. By projecting $h$ to $V_{x,0},$ we may assume $h \in V_{x,0}.$ By Lemma \ref{thstcoset}, we may write $h = h^+ + h^-$ where $h^+ \in V_{x,0}^+$ and $h^- \in V_{x,0}^-.$ We have $[h,e] = 2e \in V_{x,r}^-,$ so $[h,e] = [h^+,e] + [h^-,e] = 2e.$ By Lemma \ref{thstcoset}, $V_{x,r} = V_{x,r}^+ \oplus V_{x,r}^-$ is direct, so, we have $[h^-,e] = 0$ since $[V_{x,0}^-, V_{x,r}^-]  \subset V_{x,r}^+.$ By a similar argument, we have $f \in V_{x,-r}^-,$ so $\{f,h,e\}$ is a normal triple. This also shows that $h$ is $\theta$-fixed. In particular, the one-parameter subgroup $\overline{\lambda} \in \textbf{X}_{*}^{\mathfrak{f}}(\textsf{G}_x)$ has image inside $\textsf{H}_x.$ Moreover, it is clear that conditions 1) and 2) hold in this context as a consequence of Appendix A in [\cite{d}].  \end{remark}

\begin{defn} Keeping the above notation, we say that $\overline{\lambda} \in \textbf{X}_{*}^{\mathfrak{f}}(\textsf{H}_x)$ is adapted to the $\mathfrak{sl}_2(\mathfrak{f})$-triple obtained from the image of $(Y,H,X)$ in $V_{x,-r} \times V_{x,0} \times V_{x,r}.$ \end{defn}

\begin{hyp} \label{hyp2} If $X \in \mathcal{N}^-,$ then there exists some $m \in \mathbb{N}$ with $m \leq p-2$ such that ad$(X)^m = 0.$ \end{hyp}

\begin{hyp} \label{hyp3} Choose $m \in \N$ such that Hypothesis \ref{hyp2} holds. Suppose the characteristic of $k$ is zero or greater than $m.$ Then there exists a $G$-equivariant map \textup{exp}$: \mathcal{N} \rightarrow \mathcal{U}$ such that for all $X \in \mathcal{N},$ the adjoint action of \textup{exp}$(X)$ on $\mathfrak{g}$ is given by: \end{hyp}

$$ \textup{Ad}(\textup{exp}(X)) = \sum_{i=0}^{m}\frac{(\textup{ad}(X))^i}{i!}. $$

In the next hypothesis, we use the letter $H$ in two different contexts. In the first occurrence, it appears as an element of $\mathfrak{g}$ which is part of an $\mathfrak{sl}_2(k)$-triple. In the last line of the hypothesis, it occurs as the group of $k$-rational points of $\textbf{H}.$ This notation is unfortunate, but in most cases, the meaning of this symbol will be clear from context.

\begin{hyp} \label{hyp4} Suppose Hypothesis \ref{hyp3} holds. Let $X \in \mathcal{N}^-.$ There exists a normal $\mathfrak{sl}_2(k)$-triple completing $X.$ Moreover,  if $\{ Y, H, X \}$ is a normal $\mathfrak{sl}_2(k)$-triple completing  $X,$ then there is an algebraic group homomorphism $ \varphi: \textbf{SL}_2 \rightarrow \textbf{G} $ defined over $k$ such that $d \varphi  \begin{pmatrix} 0 & 1 \\ 0 & 0 \end{pmatrix} = X, d\varphi  \begin{pmatrix} 0 & 0 \\ 1 & 0 \end{pmatrix}  = Y, d\varphi \begin{pmatrix} 1 & 0 \\ 0 & -1 \end{pmatrix}  = H,$ and for all $t \in k,$

\begin{enumerate}
\item $\varphi \begin{pmatrix}
1 & t \\
0 & 1 \end{pmatrix} = \textup{exp}(tX) \textup{ }$ and 

\item $\varphi \begin{pmatrix}
1 & 0 \\
t & 1 \end{pmatrix}  = \textup{exp}(tY).$

\end{enumerate}
Lastly, (see below), any two normal $\mathfrak{sl}_2(k)$-triples which complete $X$ are conjugate by an element of $C_{H}(X).$ \end{hyp}

\begin{prop} Assume Hypotheses \ref{hyp1} and \ref{hyp3} hold. If $\{Y',H', X\}$ and $\{Y,H,X\}$ are two normal $\mathfrak{sl}_2(k)$-triples completing $X,$ then there exists an element $h \in C_H(X)$ for which $Y' = {}^{h}Y$ and $H'={}^{h}H.$
\end{prop}

\begin{proof} (\emph{a generalization of an argument of Kostant}) In order to verify the claim, we slightly modify the notation and argument given in [\cite{kostant}, Theorem 3.6] and [\cite{col-mc}, Lemma 3.4.7].  

	Define $\mathfrak{h}_X:=[\mathfrak{p}, X] \cap C_{\mathfrak{h}}(X),$ and let $U, V \in \mathfrak{h}_X.$ Since $V = [W, X],$ for some $W \in \mathfrak{p},$ and since $U$ centralizes $X,$ we have 
	
	$$ [U,V] = [U, [W, X]] = [X, [W,U]].$$
Since $W \in \mathfrak{p}$ and $U \in \mathfrak{h},$ it follows that $[W,U] \in \mathfrak{p},$ so $[U,V] \in [\mathfrak{p}, X].$ This shows $\mathfrak{h}_X$ is a Lie subalgebra of $\mathfrak{g}.$

	Kostant also shows that every element of $\mathfrak{g}_X:=[\mathfrak{g}, X] \cap C_{\mathfrak{g}}(X)$ is nilpotent. It follows that every element of $\mathfrak{h}_X$ is nilpotent. (Note that $\mathfrak{h}_X$ is also invariant under ad$(H).$) By Hypothesis \ref{hyp3}, the adjoint action of $\textup{exp}(W)$ for $W \in \mathfrak{h}_X$ on an element of $\mathfrak{g}$ is given by $\textup{Ad}(\textup{exp}(W)) = \sum_{i}\frac{(\textup{ad}(W))^i}{i!}.$ In particular, we have

$$ \textup{Ad}(\textup{exp}(W))(H) = \sum_{i}\frac{(\textup{ad}(W))^i(H)}{i!} \in H + \mathfrak{h}_X. $$

	We will show that for every $V \in \mathfrak{h}_X,$ there exists some $W \in \mathfrak{h}_X$ such that Ad(exp$(W))(H) = H + V.$ Define $C_{\mathfrak{h}}(X)(i):=\{ Z \in C_{\mathfrak{h}}(X) \mid [H,Z] = iZ \}.$ Kostant shows that $\mathfrak{g}_X \subset \oplus_{i=1}^{m} C_{\mathfrak{g}}(X)(i)$ for some natural number $m.$ In particular, $\mathfrak{h}_X \subset  \oplus_{i=1}^{m} C_{\mathfrak{h}}(X)(i).$ We now construct the element $W$ inductively. 
			
	Set $W_1 = -V_1,$ where $V_1$ is the component of $V$ lying in $C_{\mathfrak{h}}(X)(1).$ Then, $W_1$ lies in $\mathfrak{h}_X,$ and we have ad$(W_1)(H) = -[H,W_1] = -W_1 = V_1.$ Again, using Hypothesis \ref{hyp3}, we have 
	
\begin{eqnarray*}
 \textup{Ad(exp}(W_1))(H) - (H+V) &=& 	\sum_{i=0}^{m}\frac{(\textup{ad}(W_1))^i(H)}{i!} - (H+V) \\
 &=& (V_1 - V) + \sum_{i=2}^{m}\frac{(\textup{ad}(W_1))^i(H)}{i!} \in \bigoplus_{i \geq 2} C_{\mathfrak{h}}(X)(i).
\end{eqnarray*}	
The last line results from the fact that the restriction of ad$(H)$ to $\mathfrak{h}_X$ takes strictly positive integral values as eigenvalues.
	
	Thus, we have verified the base case. We now assume that we have constructed elements $W_j$ such that 
\begin{enumerate}
\item $W_j \in \bigoplus_{1 \leq i \leq  j} C_{\mathfrak{h}}(X)(i)$
\item Ad(exp($W_j))(H) - (H+V) \in \bigoplus_{j+1 \leq i \leq m} C_{\mathfrak{h}}(X)(i).$
\end{enumerate}
 Now, let $W_{j+1}'$ be the component of $\textup{Ad(exp}(W_j))(H) - (H+V)$ which lies in $C_{\mathfrak{h}}(X)(j+1).$
	
	Letting $W_{j+1} = W_j + \frac{1}{j+1}W_{j+1}',$ it is clear that $W_{j+1} \in \oplus_{1 \leq i \leq j+1} C_{\mathfrak{h}}(X)(i).$ Moreover, we have
\begin{eqnarray*}
\textup{Ad(exp}(W_{j+1}))(H) - (H+V) &=& \sum_{i=0}^m \frac{(\textup{ad}(W_{j+1}))^i(H)}{i!} - (H+V) \\
&=& H + [W_{j+1}, H] + \cdots - (H+V) \\
&=& H + [W_j, H] + \frac{1}{j+1}[W_{j+1}', H] + \cdots - (H+V) \\
&=& H + [W_j, H] - W_{j+1}' + \cdots - (H+V)
\end{eqnarray*}
Only the terms with indices up to $i=1$ have been expanded in the last line written. If we expand higher terms, we obtain a sum of the form
$$ H + [W_j, H] - W_{j+1}' + \frac{[W_j, [W_j, H]]}{2!} + \frac{[\frac{1}{j+1}W_{j+1}', [W_j, H]]}{2!} - \frac{[W_{j+1}, W_{j+1}']}{2!} \cdots - (H+V) $$
so it becomes clear that expanding will further will give us the sum including $ \textup{Ad(exp}(W_j))(H) - (H+V),$ $-W_{j+1}',$ and terms which lie in weight spaces of $C_{\mathfrak{h}}(X)$ with weights greater than or equal to $(j+2).$ Thus, by definition of $W_{j+1}',$ we have
$$ \textup{Ad(exp}(W_{j+1}))(H) - (H+V) \in \bigoplus_{j+2 \leq i \leq m} C_{\mathfrak{h}}(X)(i).$$
Finally, letting $W = W_m,$ we have Ad(exp$(W))(H) = H + V.$ 
	
	Now, since $[H',X] = 2X = [H,X],$ we have $H'-H \in C_{\mathfrak{h}}(X).$ On the other hand, since $[X, Y' - Y] = H'- H,$ we have $H'-H \in [\mathfrak{p}, X].$ In particular, we have $H'-H \in \mathfrak{h}_X.$ By the argument above, there is some $W \in \mathfrak{h}_X$ such that

$$ \textup{Ad(exp}(W))(H) = H + (H'-H) = H'.$$
By the construction of the element $W$ in the proof, it is clear that $W$ lies in $\mathfrak{h},$ so since exp takes $\mathfrak{h} \cap \mathcal{N}$ into $H,$ we set $h = \textup{exp}(W).$

\end{proof}

\begin{hyp} \label{hyp5} Let $x \in \mathcal{B}(H).$ For all $s \in \R_{>0}$ and for all $t \in \R,$ there exists a map $\phi_{x}: \mathfrak{h}_{x,s} \rightarrow H_{x,s}$ such that for $V \in \mathfrak{h}_{x,s}$ and $W \in \mathfrak{p}_{x,t}$ we have  \end{hyp}

$$ {}^{\phi_x(V)}W = W + [V,W] \textup{ mod } \mathfrak{p}_{x, (s+t)^+}. $$
	
	Hypothesis \ref{hyp5} as stated above is weaker than its counterpart in the group case. More precisely, as in [\cite{adler}, 1.3-1.7], suppose $x \in \mathcal{B}(G).$ Then for all $s \in \mathbb{R}_{> 0}$ and for all $t \in \mathbb{R}$ there exists a map $\phi_x: \mathfrak{g}_{x,s} \rightarrow G_{x,s}$ such that for $V \in \mathfrak{g}_{x,s}$ and $W \in \mathfrak{g}_{x,t}$ we have 

$$ {}^{\phi_x(V)}W = W + [V,W] \textup{ mod  } \mathfrak{g}_{x,(s+t)^+}. $$
From the above equation, we can derive Hypothesis \ref{hyp5} provided that the restriction of $\phi_x$ to $\mathfrak{h}_{x,s}$ maps into $H_{x,s}.$ For more details on this assumption, see [\cite{debacker-reeder}, Appendix B].

\subsection{Obtaining $\mathfrak{sl}_2(k)$-triples from $\mathfrak{sl}_2(\mathfrak{f})$-triples}
	Our next step will be to show how to obtain a normal $\mathfrak{sl}_2(k)$-triple from a normal $\mathfrak{sl}_2(\mathfrak{f})$-triple. We first recall the setup. 
	
	Let $x \in \mathcal{B}(H)$ and suppose $(f,h,e) \subset V_{x,-r}^- \times V_{x,0}^+ \times V_{x,r}^- \subset \overline{\mathfrak{g}}_x$ is a nontrivial normal $\mathfrak{sl}_2(\mathfrak{f})$-triple. Suppose $\overline{\mu} \in \textbf{X}_{*}^{\mathfrak{f}}(\textsf{H}_x)$ is adapted to $\{f,h,e\}.$ Let $\textbf{S}$ be a maximal $k$-split torus of $\textbf{H}$ such that $x \in \mathcal{A}(\textbf{S}, k).$ Let $\textsf{S}$ be the maximal $\mathfrak{f}$-split torus in $\textsf{G}_x$ corresponding to $\textbf{S}.$ Since $\textsf{H}_x$ is a reductive group over $\mathfrak{f},$ all maximal $\mathfrak{f}$-split tori are $\textsf{H}_x(\mathfrak{f})$-conjugate, so, in particular, there is a one-parameter subgroup $\overline{\lambda} \in \textbf{X}_{*}(\textsf{S})$ and an element $\overline{h} \in \textsf{H}_x(\mathfrak{f})$ with $\overline{\lambda} = {}^{\overline{h}}\overline{\mu}.$ Now, let $\lambda \in \textbf{X}_{*}(\textbf{S})$ be a lift of $\overline{\lambda}$ and substitute $ \{ {}^{\overline{h}}f, {}^{\overline{h}}h, {}^{\overline{h}}e \}$ for $ \{ f,h,e \}.$ Under the action of $\lambda$ we have the following grading on the Lie algebra $\mathfrak{g}:$

$$ \mathfrak{g}(i) := \{ X \in \mathfrak{g} \mid {}^{\lambda(t)}X = t^i\cdot X \} \textup{ and } \overline{\mathfrak{g}}_x(i) := \{ v \in \overline{\mathfrak{g}}_x \mid {}^{\overline{\lambda}(t)}v = t^i\cdot v \}.$$ 
For $s \in \mathbb{R},$ we have analogous gradings on $\mathfrak{g}_{x,s}$ and $V_{x,s}$ defined by

$$ \mathfrak{g}_{x,s}(i):= \{ Z \in \mathfrak{g}_{x,s} \mid {}^{\lambda(t)}Z = t^i\cdot  Z \} \textup{ and } V_{x,s}(i):= \{v \in V_{x,s} \mid {}^{\overline{\lambda}(t)}v = t^i \cdot v \}. $$
Define $\mathfrak{p}(i):= \mathfrak{p} \cap \mathfrak{g}(i),$ $\mathfrak{p}_{x,r}(i):= \mathfrak{p}_{x,r} \cap \mathfrak{g}(i),$ and $V_{x,r}^-(i):= V_{x,r}^- \cap V_{x,r}(i).$

\begin{remark} We recall that the Lie bracket on $\mathfrak{g}$ does not preserve $\mathfrak{p}.$ In fact, we have $[V,W] \in \mathfrak{h}$ for all $V,W \in \mathfrak{p}.$ In particular, if $X \in \mathfrak{p},$ and $Y \in \mathfrak{p},$ the element $\textup{ad}(X)^2(Y)$ lies in $\mathfrak{p}.$ This shows why the map in the following lemma is well-defined. \end{remark}

\begin{lemma} \label{sllemma} Suppose Hypothesis \ref{hyp1} holds. If $X \in \mathfrak{p}_{x,r}(2)$ is a lift of $e,$ then, for all $s \in \mathbb{R},$ the map 

$$ \textup{ad}(X)^2: \mathfrak{p}_{x,s-r}(-2) \rightarrow \mathfrak{p}_{x,s+r}(2) $$
is an isomorphism of $R$-modules. \end{lemma}

\begin{proof} By [\cite{d}, Lemma 4.3.1], we know that the map $\textup{ad}(X)^2: \mathfrak{g}_{x,s-r}(-2) \rightarrow \mathfrak{g}_{x,s+r}(2)$ is an isomorphism of $R$-modules. Thus, 

$$ \textup{ad}(X)^2: \mathfrak{p}_{x,s-r}(-2) \rightarrow \mathfrak{p}_{x,s+r}(2)$$
is injective. 

	Let $Z \in \mathfrak{p}_{x,s+r}(2) \subset \mathfrak{g}_{x,s+r}(2).$ By [\cite{d}, 4.3.1], there is an element $Z' \in \mathfrak{g}_{x,s-r}(-2)$ such that $(\textup{ad}(X)^2)(Z') = Z.$ By Lemma \ref{thstpa}, we may write $Z'=Z'_+ + Z'_-,$ where $Z'_+ \in \mathfrak{h}_{x,s-r}$ and $Z'_- \in \mathfrak{p}_{x,s-r}.$ By the last line in [\cite{d}, Section 4.3], the projection $\mathfrak{g} \rightarrow \mathfrak{g}(i)$ preserves depth, so we let $W$ denote the projection of $Z'_-$ to the $(-2)$ weight space. Then, since $Z \in \mathfrak{p}_{x,s+r}(2),$ we have $(\textup{ad}(X)^2)(W) = Z.$ Thus, the map is surjective.
\end{proof}

\begin{cor} \label{sllift} Suppose Hypotheses \ref{hyp1} and \ref{hyp2} hold. If $X \in \mathfrak{p}_{x,r}(2)$ is a lift of $e,$ then there are lifts $Y \in \mathfrak{p}_{x,-r}$ of $f$ and $H \in \mathfrak{h}_{x,0}$ of $h$ such that $\{Y,H,X\}$ is a normal $\mathfrak{sl}_2(k)$-triple in $\mathfrak{g}.$ \end{cor}

\begin{proof} Let $X \in \mathfrak{p}_{x,r}(2)$ be a lift of $e.$ By the previous lemma, $\textup{ad}(X)^2: \mathfrak{p}_{x,-r}(-2) \rightarrow \mathfrak{p}_{x,r}(2)$ is surjective, so there exists an element $Y \in \mathfrak{p}_{x,-r}(-2)$ with $\textup{ad}(X)^2(Y) = -2X.$ By the proof of [\cite{d}, Lemma 4.3.1], $\textup{ad}(e)^2 : \overline{\mathfrak{g}}_x(-2) \rightarrow \overline{\mathfrak{g}}_x(2)$ is injective, so since ad$(e)^2(f) = -2e$ and ad$(e)^2(\overline{Y} - f) = 0,$ we have that $Y$ is a lift of $f.$ Set $H=[X,Y].$ By a computation, $[H,X] = 2X,$ so in order to show that $\{Y,H,X\}$ is our desired $\mathfrak{sl}_2(k)$-triple, we must verify that $[H,Y] = -2Y.$ 
	
		By [\cite{carter}, Theorem 5.3.2], there exists some $Y' \in \mathfrak{g}$ which completes $\{H,X\}$ to an $\mathfrak{sl}_2(k)$-triple. By projecting $Y'$ to $\mathfrak{p}(-2),$ we can assume it lies in this weight space. Now ad$(X)^2(Y') = \textup{ad}(X)^2(Y),$ so by Lemma \ref{sllemma}, since $\textup{ad}(X)^2$ is injective, we have $Y=Y'.$ 
\end{proof}

\subsection{One-parameter subgroups}

We now fix a one-parameter subgroup $\lambda \in \textbf{X}_{*}^{k}(\textbf{H}).$ The following material is obtained from results in [\cite{d}, Section 4.4].

	Fix an element $X \in \mathcal{N} \cap \mathfrak{p}.$ Suppose Hypothesis \ref{hyp4} holds. Then, there exists a normal $\mathfrak{sl}_2(k)$-triple $\{Y,H,X\}$ completing $X$ and a homomorphism $\varphi: \textbf{SL}_2 \rightarrow \textbf{G}$ so that $H =d\varphi( \begin{pmatrix} 1 & 0 \\ 0 & -1 \end{pmatrix})$ and $Y = d\varphi(\begin{pmatrix} 0 & 0 \\ 1 & 0 \end{pmatrix}).$ Note that such a map is Gal$(K/k$)-equivariant.
	
	We will now exhibit a point $y \in \mathcal{B}(H)$ such that $Y \in \mathfrak{p}_{y,-r}, H \in \mathfrak{h}_{y,0}, $ and $X \in \mathfrak{p}_{y,r}.$ The argument given in the lemma below (excluding the last paragraph) is due to Gopal Prasad.

\begin{lemma} \label{Gopal} (Barbasch and Moy). Suppose Hypothesis \ref{hyp4} holds. There exists some $x \in \mathcal{B}(H)$ such that $Y, H, X \in \mathfrak{g}_{x,0}.$\end{lemma}

\begin{proof} (\emph{Gopal Prasad}) Let $J = \varphi(\textbf{SL}_2(R_K)) \subset \textbf{G}^{\circ}(K).$ Then, $B:= (J \rtimes Gal(K/k)) $ is a subgroup of the group of polysimplicial automorphisms of $\mathcal{B}(\textbf{G},K).$ Note that Gal$(K/k)$ is a profinite group; in particular, it is compact and bounded. Thus, $B$ is also bounded, so by [\cite{tits}, 2.3.1], there exists a fixed point $x' \in \mathcal{B}(\textbf{G}, K)$ under the action of $B.$ Let $\mathcal{G}$ denote the smooth affine $R$-group scheme whose $R_K$-points form the group stab$_{\textbf{G}^{\circ}(K)}(x')$ and whose generic fiber is $\textbf{G}^{\circ}.$ Let $L(\mathcal{G})$ denote the Lie algebra of $\mathcal{G},$ and let $\mathcal{J}$ denote the $R$-group scheme associated to the parahoric subgroup $\textbf{SL}_2(R_K).$ By [\cite{bt2}, 1.7.6], $\varphi$ induces a map of $R_K$-schemes from $\mathcal{J}$ to $\mathcal{G}.$ Thus $d\varphi$\mbox{\boldmath($\mathfrak{sl}_2$}$(R_K)) \subset$ \mbox{\boldmath$\mathfrak{g}$}$(K)_{x'}.$ Now, since $x'$ is fixed by Gal$(K/k),$ we have $Y,H,X \in \mathfrak{g}_{x'}.$ 

	We have shown that the set of $B$-fixed points $\Omega:=\mathcal{B}(\textbf{G}, K)^{B}$ is nonempty. Since $\{Y,H,X\}$ is a normal triple, we have $d\varphi(\mathfrak{sl}_2(R)) \subset \mathfrak{g}_{\theta(x')}.$ In particular, by [\cite{d}, Corollary 4.5.5], we have $\theta(x') \in \Omega,$ so $\Omega$ is $\theta$-stable. In particular, since $\Omega$ is convex and closed, and $\langle \theta \rangle$ is a bounded group of isometries, there exists a $\theta$-fixed point $x \in \Omega$ for which $Y, H, X \in \mathfrak{g}_x.$

\end{proof}

Under Hypothesis \ref{hyp4}, there is a homomorphism $\varphi: \textbf{SL}_2 \rightarrow \textbf{G}$ with some nice properties with respect to $\{Y,H,X\}.$ Let $\lambda \in \textbf{X}_{*}^{k}(\textbf{G})$ be defined by $\lambda(t) = \varphi \left( \begin{pmatrix} t & 0 \\ 0 & t^{-1} \end{pmatrix} \right) .$ 

\begin{defn} The one-parameter subgroup $\lambda$ described above is said to be adapted to the $\mathfrak{sl}_2(k)$-triple $\{Y,H,X\}.$ \end{defn}

\begin{remark} \label{kempf} In the preliminaries of Section 2.1, we declared that an element $X \in \mathfrak{p}$ is nilpotent provided that there exists some one-parameter subgroup $\mu \in \textbf{X}_*^k(\textbf{G})$ such that 

$$ \lim_{t \rightarrow 0} {}^{\mu(t)}X = 0.$$
However, assuming Hypothesis \ref{hyp4} is valid, we can give an alternate characterization of nilpotence which coincides with this notion. Namely, suppose $X$ lies in $\mathcal{N} \cap \mathfrak{p},$ and suppose $\{Y,H,X\}$ is a normal $\mathfrak{sl}_2(k)$-triple completing $X.$ By Jacobson-Morosov, there exists some one-parameter subgroup $\lambda$ for which ${}^{\lambda(t)}X = t^2X,$ for $t \in k^{\times}.$ Since $\{Y,H,X\}$ is normal, $H$ is $\theta$-fixed; in particular, we may assume $\lambda$ is fixed by $\theta.$ Thus, under Hypothesis \ref{hyp4}, $X$ lies in $\mathcal{N} \cap \mathfrak{p}$ if and only if there exists some one-parameter subgroup in $\textbf{X}_*^k(\textbf{H})$ which annihilates $X$ in the limit described above. \end{remark}

As noted in the remark above, if $\{Y,H,X\}$ is normal, we may assume $\lambda \in \textbf{X}_*^k(\textbf{H}).$ Define $M = C_{\textbf{G}^{\circ}(k)}(\lambda).$

\begin{cor} \label{bsetnonempty} Suppose Hypotheses \ref{hyp2} and \ref{hyp4} hold. There exists some $y \in \mathcal{B}(H)$ such that $Y \in \mathfrak{p}_{y,-r}, X \in \mathfrak{p}_{y,r},$ and $H \in \mathfrak{h}_{y,0}.$ \end{cor}

\begin{proof} Together, Hypotheses \ref{hyp2} and \ref{hyp4} imply that the residue field $\mathfrak{f}$ has cardinality greater than 3. By Lemma \ref{Gopal}, there is an element $x \in \mathcal{B}(H)$ such that $Y,H,X \in \mathfrak{g}_x.$ Since $\lambda(R^{\times}) \subset J$ as in the proof of Lemma \ref{Gopal}, we know that the point $x$ is fixed by $ \lambda (R^{\times}).$ In particular, by [\cite{d}, Corollary 4.4.2], $x$ lies in $\mathcal{B}(M).$ Choose an apartment $\mathcal{A} \subset \mathcal{B}(H)$ which contains $x.$ Since $\lambda$ lies in the center of $M,$ $\lambda$ acts on every apartment in $\mathcal{B}(M)$ by translation. Using this fact, define $y = x + \frac{r}{2} \cdot \lambda \in \mathcal{A}.$ By Lemma \ref{Gopal}, $X \in \mathfrak{p}_{x,0},$ so we write $X = \sum_{ \psi} X_{\psi},$ where $X_{\psi} \in \mathfrak{g}_{\psi},$ for $\psi(x) \geq 0.$ For all such $\psi$ such that $X_{\psi} \neq 0,$ we have $\langle \lambda, \dot{\psi} \rangle = 2$ since $\lambda$ acts by squares on $X$ by Hypothesis \ref{hyp1}. For any such $\psi,$ we have

$$ \psi(y) = \psi(x) + \frac{r}{2} \langle \lambda, \dot{\psi} \rangle \geq r. $$
Therefore, $X$ lies in $\mathfrak{p}_{y,r}.$ By a similar argument, $H \in \mathfrak{h}_{y,0}$ and $Y \in \mathfrak{p}_{y,-r}.$ 

\end{proof}

\section{The parametrization}

		Fix $r \in \mathbb{R}.$ We now discuss the notion of the building set associated to an $\mathfrak{sl}_2(k)$-triple, so we assume that Hypotheses \ref{hyp2} and \ref{hyp4} hold. We follow the discussion in [\cite{d}, Section 5]. Fix $Z \in \mathcal{N}^-$ and $s \in \mathbb{R}.$
		
\subsection{The building set}

\begin{defn} 

$$ \mathcal{B}(Z,s) := \{z \in \mathcal{B}(G) \mid Z \in \mathfrak{g}_{z,s} \}.$$
\end{defn}

From [\cite{d}], we know that $\mathcal{B}(Z,s)$ is nonempty, convex and closed. 

\begin{defn} \label{thetabset} 

$$ \mathcal{B}_{\theta}(Z,s) := \mathcal{B}(Z,s) \cap \mathcal{B}(H). $$
\end{defn}

Corollary \ref{bsetnonempty} tells us that $\mathcal{B}_{\theta}(Z,s)$ is nonempty. From Definition \ref{thetabset}, we see that if $x \in \mathcal{B}_{\theta}(Z,s),$ then $F_{\theta}^*(x) \subset \mathcal{B}_{\theta}(Z,s).$ In particular, $\mathcal{B}_{\theta}(Z,s)$ is the union of generalized $(s,\theta)$-facets of $\mathcal{B}(H).$ Since $\mathcal{B}(Z,s)$ and $\mathcal{B}(H)$ are convex, $\mathcal{B}_{\theta}(Z,s)$ is also convex. 

\begin{lemma} \label{bclosed} $\mathcal{B}_{\theta}(Z,s)$ is closed. \end{lemma}

\begin{proof} This follows from the fact that $\mathcal{B}(H)$ and $\mathcal{B}(Z,s)$ are closed.
\end{proof}

Fix a (possibly trivial) normal $\mathfrak{sl}_2(k)$-triple $\{Y,H,X\}$ in $\mathfrak{g}.$ 

\begin{defn} Define
$$ \mathcal{B}(Y,H,X):= \mathcal{B}(X,r) \cap \mathcal{B}(Y,-r). $$
\end{defn}

\begin{defn} Define
$$ \mathcal{B}_{\theta}(Y,H,X):= \mathcal{B}(Y,H,X)^{\theta}.$$
\end{defn}

By [\cite{d}, Remark 5.1.5], $\mathcal{B}(Y,H,X)$ is convex. In particular, $\mathcal{B}_{\theta}(Y,H,X)$ is a closed, convex set which is the union of generalized $(r,\theta)$-facets. 

\begin{lemma} \label{maxassoc} Suppose $F_{1,\theta}^*, F_{2,\theta}^*$ are maximal generalized $(r,\theta)$-facets in $\mathcal{B}_{\theta}(Y,H,X).$ Then, $F_{1,\theta}^*$ and $F_{2,\theta}^*$ are strongly $r$-associated. 
\end{lemma}

\begin{proof} Let $x_i \in F_{i,\theta}^*,$ for $i=1,2,$ and let $\mathcal{A}$ be an apartment of $\mathcal{B}(H)$ containing $x_1$ and $x_2.$ If $x_1 \notin A(\mathcal{A}, F_{2,\theta}^*),$ then, since $\mathcal{B}_{\theta}(Y,H,X)$ is convex, there is some generalized $(r,\theta)$-facet in $\mathcal{B}_{\theta}(Y,H,X)$ of strictly larger dimension than $F_{2,\theta}^* \cap \mathcal{A},$ a contradiction. Thus, $A(\mathcal{A}, F_{1,\theta}^*) \subset A(\mathcal{A}, F_{2,\theta}^*),$ and similarly, $A(\mathcal{A}, F_{2,\theta}^*) \subset A(\mathcal{A}, F_{1,\theta}^*).$ 

\end{proof}

We now suppose that Hypotheses \ref{hyp2}, \ref{hyp4}, and \ref{hyp5} hold. Fix $X \in \mathcal{N}^- \backslash \{ 0 \}$ and $r \in \mathbb{R}.$ Suppose that $\{ Y,H,X \}$ is a normal $\mathfrak{sl}_2(k)$-triple completing $X$ and that $\lambda \in \textbf{X}_*^k(\textbf{H})$ is adapted to $\{Y,H,X\}.$ Fix $x \in \mathcal{B}_{\theta}(Y,H,X).$ We would like for ${}^{H}X$ to be the unique nilpotent $H$-orbit in $\mathfrak{p}$ of minimal dimension which intersects the coset $X + \mathfrak{p}_{x,r^+}$ nontrivially. 
	The next lemma gives us a decomposition of the coset $X + \mathfrak{p}_{x,r^+}$ up to conjugacy by $H_x^+.$ Recall that the one-parameter subgroup $\lambda$ induces a grading on the Lie algebra of $\mathfrak{g}$ as noted in the beginning of Section 5.2. For the following lemma, we imitate the argument in [\cite{d}, 5.2.1].

\begin{lemma} \label{walds1} Assume Hypotheses \ref{hyp2}, \ref{hyp4}, and \ref{hyp5} hold. Then\end{lemma}

$$ {}^{H_x^+}(X + C_{\mathfrak{p}_{x,r^+}}(Y)) = X + \mathfrak{p}_{x,r^+}. $$

\begin{proof} $``\subset "$: By [\cite{adler}, Prop 1.4.3], $H_x^+$ induces the trivial action on $V_{x,r}.$ \newline
$`` \supset "$: From Hypothesis \ref{hyp1}, we know that as a representation of $\langle Y,H,X \rangle,$ $\mathfrak{g}$ decomposes into a direct sum of irreducible $\langle Y,H,X \rangle$-modules with highest weight at most $p-3.$ Thus, we can write 

$$ \mathfrak{g} = \bigoplus_{\rho \in \mathbb{Z}_{\geq 0}} \mathfrak{g}_{\rho} $$
where $\mathfrak{g}_{\rho}$ is the isotypic component of $\langle Y,H,X \rangle$-modules in $\mathfrak{g}$ with highest weight $\rho.$ In other words, $\mathfrak{g}_{\rho}$ is the direct sum of irreducible $\langle Y,H,X \rangle$-submodules of $\mathfrak{g}$ of dimension ${\rho}+1.$ Let $\mathfrak{g}(\rho,i) := \mathfrak{g}_{\rho} \cap \mathfrak{g}(i),$ and let $\mathfrak{p}(\rho,i):=\mathfrak{g}(\rho, i) \cap \mathfrak{p}.$ Then, $\mathfrak{g}(i) = \oplus_{\rho}\mathfrak{g}(\rho,i)$ and thus $\mathfrak{g} = \oplus_{\rho,i}\mathfrak{g}(\rho,i).$ Also, note that we have that $C_{\mathfrak{g}}(X) = \oplus_{i \geq 0} \mathfrak{g}(i,i)$ since ad$(X)(\mathfrak{g}(i)) = \mathfrak{g}(i+2),$ and $\mathfrak{g}(i,i)$ is the sum of $i$-weight spaces of all irreducible $\langle Y,H,X \rangle$-submodules with highest weight $i.$ Similarly, we have $C_{\mathfrak{g}}(Y) = \oplus_{i \geq 0}\mathfrak{g}(i,-i).$  It follows that $C_{\mathfrak{p}}(Y) = \oplus_{i \geq 0} \mathfrak{p}(i,-i)$ and $C_{\mathfrak{p}}(X) = \oplus_{i \geq 0} \mathfrak{p}(i,i).$ Define $\mathfrak{g}_{x,s}(\rho,i):= \mathfrak{g}_{x,s} \cap \mathfrak{g}(\rho,i).$ By the proof of [\cite{d}, Lemma 5.2.1], since the projection $\mathfrak{g} \rightarrow \mathfrak{g}(i)$ preserves depth, we have the following decompositions:
$$ \mathfrak{g}_{x,s}(i) = \bigoplus_{\rho \in \mathbb{Z}_{\geq 0}} \mathfrak{g}_{x,s}(\rho, i)$$
and
$$ \mathfrak{g}_{x,s} = \bigoplus_{\rho, i}\mathfrak{g}_{x,s}(\rho,i). $$
From the first decomposition, using Proposition \ref{thstpa} it follows that 

$$ \mathfrak{h}_{x,s}(i) = \bigoplus_{\rho \in \mathbb{Z}_{\geq 0}} \mathfrak{h}_{x,s}(\rho, i), $$
and
$$ \mathfrak{p}_{x,s}(i) = \bigoplus_{\rho \in \mathbb{Z}_{\geq 0}} \mathfrak{p}_{x,s}(\rho, i). $$
We claim that the following decompositions hold:
$$ (\dagger) \textup{ For } i > 0,  \textup{ } \mathfrak{p}_{x,s}(i) = \textup{ad}(X)(\mathfrak{h}_{x,s-r}(i-2)) $$
and 
$$ (\dagger \dagger) \textup{ For } i \leq 0, \mathfrak{p}_{x,s}(i) = \mathfrak{p}_{x,s}(-i,i) + \textup{ad}(X)(\mathfrak{h}_{x,s-r}(i-2)).$$
The first decomposition $(\dagger)$ is obtained using decompositions  $\mathfrak{h}_{x,s}(i) = \bigoplus_{\rho \in \mathbb{Z}_{\geq 0}} \mathfrak{h}_{x,s}(\rho, i)$ and $\mathfrak{p}_{x,s}(i) = \bigoplus_{\rho \in \mathbb{Z}_{\geq 0}} \mathfrak{p}_{x,s}(\rho, i)$ and the fact that $\mathfrak{p}(\rho, i) = \{ Z \in \mathfrak{p}(i) \mid (\textup{ad}(X) \circ \textup{ad}(Y))(Z) = j(\rho,i) \cdot Z \},$ which can be found in [\cite{d}, 5.2.1].

	The second decomposition $(\dagger \dagger)$ results from the fact that for $Z \in \mathfrak{p}_{x,s}(i), i \leq 0,$ $Z$ may lie in $C_{\mathfrak{p}}(Y).$ We also use the fact that $\mathfrak{p}_{x,s}(i) = \bigoplus_{\rho \in \mathbb{Z}_{\geq 0}} \mathfrak{p}_{x,s}(\rho, i).$
	
	Now, summing over all $i,$ we obtain 
	
$$ ( \star) \textup{ } \mathfrak{p}_{x,s} = C_{\mathfrak{p}_{x,s}}(Y) + \textup{ad}(X)(\mathfrak{h}_{x,s-r}).$$ 
	
	Let $Z \in \mathfrak{p}_{x,r^+}.$ We will show the existence of elements $h \in H_x^+$ and $C \in C_{\mathfrak{p}_{x,r^+}}(Y)$ such that ${}^{h}(X+C) = X+Z.$ First, let $h_0 = 1$ and $C_0 = 0.$ Now, choose $s_1 \in \R$ with $\mathfrak{p}_{x,r^+} = \mathfrak{p}_{x,s_1} \neq \mathfrak{p}_{x,s_1^+}.$ Using $(\star),$ we can write $Z = C_1' + \textup{ad}(X)(P_1)$ where $C_1' \in C_{\mathfrak{p}_{x,s_1}}(Y)$ and $P_1 \in \mathfrak{h}_{x,(s_1-r)}.$ Applying Hypothesis \ref{hyp5} with $s=s_1-r$ and $t=s_1,$ there exists a map $\phi_x : \mathfrak{h}_{x,s_1-r} \rightarrow H_{x,s_1-r}$ such that 

$$ (1) \textup{ }  {}^{\phi_x(-P_1)}(X+C_0+C_1') = X + C_1' + \textup{ad}(X)(P_1) \textup{ mod } \mathfrak{p}_{x,s_1^+}. $$
Set $h_1' = \phi_x(-P_1).$ Rewriting the above equation, we have ${}^{h_1'h_0}(X+C_0+C_1') = X + Z - Z_1$ for some $Z_1 \in \mathfrak{p}_{x,s_1^+}.$ Let $h_1=h_1'h_0$ and $C_1=C_0+C_1'.$ Now, fix an element $s_2 > s_1$ such that $\mathfrak{p}_{x,s_1^+} = \mathfrak{p}_{x,s_2} \neq \mathfrak{p}_{x,s_2^+}.$ Continuing as in the previous case, from $(\star),$ we write $Z_1 = C_2' + \textup{ad}(X)(P_2)$ where $C_2' \in C_{\mathfrak{p}_{x,s_2}}(Y)$ and $P_2 \in \mathfrak{h}_{x,s_2-r}.$ Applying Hypothesis \ref{hyp5} and (1), there exists a map $\phi_x$ such that 
\begin{eqnarray*}
{}^{\phi_x(-P_2)}(X+C_1+C_2') &=& {}^{h_2'}(X+Z-Z_1+C_2') \textup{ mod } \mathfrak{p}_{x,s_2^+} \\
&=& X+Z-Z_1+C_2' + \textup{ad}(X)P_2 \textup{ mod } \mathfrak{p}_{x,s_2^+} \\
&=& X+Z-Z_2.
 \end{eqnarray*}
where $Z_2 \in \mathfrak{p}_{x,s_2^+}$ and $h_2' = \phi_x(-P_2).$ Set $h_2 = h_2'h_1$ and $C_2=C_1+C_2'.$ Proceeding as above, we obtain a strictly increasing sequence $\{s_i\}$ with $s_1 > r$ such that $h_n \in H_x^+$ and $h_n' \in H_{x,(s_n-r)}.$ Moreover, we have elements $C_n = C_{n-1} + C_n' \in C_{\mathfrak{p}_{x,r^+}}(Y)$ such that $C_n' \in C_{\mathfrak{p}_{x,s_n}}(Y)$ and 

$$ {}^{h_n}(X+C_n) = X+Z \textup{ mod } \mathfrak{p}_{x,s_n^+}.$$
Now set $h=\lim_{n \rightarrow \infty} h_n$ and $C = \lim_{n \rightarrow \infty}C_n.$ Clearly, these elements lie in $H_x^+$ and $C_{\mathfrak{p}_{x,r^+}}(Y),$ respectively. By construction, we have ${}^{h}(X+C) = X+Z.$
\end{proof}

\begin{lemma} \label{walds2} Suppose Hypothesis \ref{hyp4} holds. Then \end{lemma}

$$ (X + C_{\mathfrak{p}}(Y)) \cap {}^{H}X = \{ X \}.$$

\begin{proof} The result [ \cite{walds}, V.7 (9)] tells us that $(X+C_{\mathfrak{g}}(Y))) \cap {}^{G}X = \{ X\},$ so, consequently, $(X+C_{\mathfrak{p}}(Y)) \cap {}^{H}X \subset \{X\}.$ Since $\{X\}$ is clearly contained in the intersection, the result follows. \end{proof}

\begin{cor} \label{intersection} Suppose Hypotheses \ref{hyp2}, \ref{hyp4}, and \ref{hyp5} hold. Then \end{cor}

$$ (X+\mathfrak{p}_{x,r^+} ) \cap {}^{H}X = {}^{H_x^+}X. $$

\begin{proof} $``\supset"$: This inclusion follows from the first part of the proof of Lemma \ref{walds1}. \newline
$``\subset "$: Using Lemma \ref{walds1}, we have $X+\mathfrak{p}_{x,r^+} = {}^{H_x^+}\left( X + C_{\mathfrak{p}_{x,r^+}}(Y) \right).$ Thus, if $Z$ lies in $(X+\mathfrak{p}_{x,r^+}) \cap {}^{H}X,$ then there exist elements $h_1 \in H_x^+, h_2 \in H$ and $Z_1 \in C_{\mathfrak{p}_{x,r^+}}(Y)$ such that $Z = {}^{h_1}(X+Z_1) = {}^{h_2}X.$ Thus ${}^{h_1^{-1}}Z = X + Z_1 = {}^{h_1^{-1}h_2}X,$ so by Corollary \ref{walds2}, we have

$$ X+Z_1 \in (X+C_{\mathfrak{p}}(Y)) \cap {}^{H}X = \{ X \} $$
Thus, $Z_1 = 0.$ 

\end{proof}

\begin{defn} We denote by $\mathcal{O}_{\theta}(0)$ the set of all nilpotent $H$-orbits in $\mathfrak{p}.$ \end{defn}

In the statement of the following corollary, if $\mathcal{O}_{\theta}$ is an element of $\mathcal{O}_{\theta}(0),$ we let $\overline{\mathcal{O}_{\theta}}$ denote the $p$-adic closure of $\mathcal{O}_{\theta}.$
\begin{cor} \label{orbit} Suppose Hypotheses \ref{hyp2}, \ref{hyp4}, and \ref{hyp5} hold. If $\mathcal{O}_{\theta} \in \mathcal{O}_{\theta}(0)$ such that \end{cor}

$$ (X+\mathfrak{p}_{x,r^+}) \cap \mathcal{O}_{\theta} \neq \emptyset, $$
then ${}^{H}X \subset \overline{\mathcal{O}_{\theta}}.$ 

\begin{proof} (\emph{J.L. Waldspurger}) Let $Z \in (X + \mathfrak{p}_{x,r^+}) \cap \mathcal{O}_{\theta}.$ Then by Lemma \ref{walds1}, there exist elements $h \in H_x^+$ and $C \in C_{\mathfrak{p}_{x,r^+}}(Y)$ such that ${}^{h}(X+C) =Z \in \mathcal{O}_{\theta}.$ Inverting $h,$ we have $X+C \in \mathcal{O}_{\theta},$ which is therefore nilpotent. By Jacobson-Morosov, there exists a one-parameter subgroup $\mu \in \textbf{X}^{k}_{*}(\textbf{H})$ such that ${}^{\mu(t)}(X+C) = t^2\cdot (X+C)$ for all $t \in k^{\times}.$ Recall that we let $\lambda$ denote the one-parameter subgroup adapted to $\mathfrak{sl}_2(k)$-triple $\{Y,H,X\}.$ In particular, ${}^{\lambda(t)}X = t^2\cdot X,$ for $t \in k^{\times}.$ Moreover, from the proof of Lemma \ref{walds1}, we know that $C_{\mathfrak{p}}(Y) = \oplus_{i \geq 0} \mathfrak{p}(i,-i),$ so, in particular, $C \in \oplus_{i \leq 0}\mathfrak{p}(i).$ Thus, 

$$ \lim_{t \rightarrow 0}{}^{\lambda(t)^{-1}\mu(t)}(X+C) = \lim_{t \rightarrow 0}{}^{\lambda(t)^{-1}}t^2(X+C) = X+\lim_{t \rightarrow 0}{}^{\lambda(t)^{-1}}C = X.$$ 
\end{proof}

\begin{cor} \label{orbclosure} Suppose Hypotheses \ref{hyp2}, \ref{hyp4}, and \ref{hyp5} hold. Choose $F_{\theta}^* \in \mathcal{F}_{\theta}(r)$ such that $F_{\theta}^* \subset \mathcal{B}_{\theta}(Y,H,X).$ If $\mathcal{O}_{\theta} \in \mathcal{O}_{\theta}(0)$ such that \end{cor}

$$ (X+\mathfrak{p}_{F_{\theta}^*}^+) \cap \mathcal{O}_{\theta} \neq \emptyset,$$ 
then ${}^{H}X \subset \overline{\mathcal{O}_{\theta}}.$ 

\begin{proof} If $x \in F_{\theta}^*,$ then $X \in \mathfrak{p}_{x,r} = \mathfrak{p}_{F_{\theta}^*}.$ Thus, by Corollary \ref{orbit}, ${}^{H}X \subset \overline{\mathcal{O}_{\theta}}.$
\end{proof}

\begin{defn} Define $I_r^n:=\{ (F_{\theta}^*,v) \in I_r \mid$ $v$ is $degenerate$ in $V_{F_{\theta}^*}^- \}.$ \end{defn}

	Suppose $(F_{\theta}^*,e) \in I_r^n.$ Let $x \in F_{\theta}^*.$ If $e$ is trivial, then we call the normal $\mathfrak{sl}_2(\mathfrak{f})$-triple completing $e$ (in $V_{x,-r}^- \times V_{x,0}^+ \times V_{x,r}^-)$ the trivial $\mathfrak{sl}_2(\mathfrak{f})$-triple. Similarly, given a trivial $\mathfrak{sl}_2(\mathfrak{f})$-triple, we declare that the $\mathfrak{sl}_2(k)$-triple lifting our $\mathfrak{sl}_2(\mathfrak{f})$-triple is the trivial $\mathfrak{sl}_2(k)$-triple. 

\begin{lemma} \label{facettoorbit} Suppose all hypotheses from Section 5 hold, and let $(F_{\theta}^*,e) \in I_r^n.$ \newline
\begin{enumerate}
\item Fix $x \in F_{\theta}^*.$ There exists a normal $\mathfrak{sl}_2(\mathfrak{f})$-triple $(f,h,e) \in V_{x,-r}^- \times V_{x,0}^+ \times V_{x,r}^-$ completing $e$ and a normal $\mathfrak{sl}_2(k)$-triple $\{Y,H,X\}$ which lifts $\{f,h,e\}.$ \newline
\item For any $x \in F_{\theta}^*,$ for any normal $\mathfrak{sl}_2(\mathfrak{f})$-triple $(f,h,e) \in V_{x,-r}^- \times V_{x,0}^+ \times V_{x,r}^-$ completing $e,$ and for any normal $\mathfrak{sl}_2(k)$-triple $\{Y,H,X\}$ which lifts $\{f,h,e\},$ we have $F_{\theta}^* \subset \mathcal{B}_{\theta}(Y,H,X),$ and ${}^{H}X$ is the unique nilpotent $H$-orbit in $\mathfrak{p}$ of minimal dimension which intersects the coset $e$ nontrivially. 
\end{enumerate}
\end{lemma}
\begin{proof} If $e$ is trivial, then all the conclusions are obvious. Assume $e$ is nontrivial. By Hypothesis \ref{hyp1}, there exist elements $H \in \mathfrak{h}_{x,0}$ and $Y \in \mathfrak{p}_{x,-r}$ such that the image $\{f,h,e\}$ of the triple $\{Y,H,X\}$ forms a normal $\mathfrak{sl}_2(\mathfrak{f})$-triple in $\overline{\mathfrak{g}}_x.$ By Corollary \ref{sllift}, there is a normal $\mathfrak{sl}_2(k)$-triple lifting $\{f,h,e\}.$ 

	For (2), suppose $x \in F_{\theta}^*,$ the triple $\{f,h,e\} \textup{ }(\subset V_{x,-r}^- \times V_{x,0}^+ \times V_{x,r}^-)$ is a normal $\mathfrak{sl}_2(\mathfrak{f})$-triple in $\overline{\mathfrak{g}}_x$ completing $e,$ and $\{Y,H,X\}$ is a normal $\mathfrak{sl}_2(k)$-triple which lifts $\{f,h,e\}.$ Note that since $x \in F_{\theta}^*,$ we have $F_{\theta}^* \subset \mathcal{B}_{\theta}(Y,H,X)$ since $X \in \mathfrak{p}_{x,r}$ and $Y \in \mathfrak{p}_{x,-r}.$ Now, if ${}^{H}{Z}$ is a nilpotent $H$-orbit such that $(X + \mathfrak{p}_{x,r^+}) \cap {}^{H}Z \neq \emptyset,$ then by Corollary \ref{orbclosure}, we have ${}^{H}X \subset \overline{{}^{H}Z}.$ Thus, since the closure of any $H$-orbit is the union of the the orbit itself and orbits of strictly smaller dimension, ${}^{H}X$ is the unique nilpotent $H$-orbit in $\mathfrak{p}$ of minimal dimension which intersects the coset $e$ nontrivially.  
\end{proof}

\begin{defn} \label{orbdef} Suppose all the hypotheses of Section 5 hold. For $(F_{\theta}^*, e) \in I_r^n$, let $\mathcal{O}_{\theta}(F_{\theta}^*,e)$ denote the unique nilpotent $H$-orbit of minimal dimension which intersects the coset $e$ nontrivially. \end{defn}

\begin{remark} \label{hremark} If $h \in H$ and $(F_{\theta}^*,e) \in I_r^n,$ then it is clear from Definition \ref{orbdef} that $\mathcal{O}_{\theta}(hF_{\theta}^*,{}^{h}e) = \mathcal{O}_{\theta}(F_{\theta}^*,e).$ \end{remark}

\begin{lemma} \label{welldef} Suppose all the hypotheses of Section 5 hold. The map $\varphi: I_r^n \rightarrow \mathcal{O}_{\theta}(0)$ defined by $(F_{\theta}^*, e) \mapsto \mathcal{O}_{\theta}(F_{\theta}^*,e)$ induces a well-defined map from $I_r^n/\sim$ to $\mathcal{O}_{\theta}(0).$ \end{lemma}

\begin{proof} Suppose $(F_{i,\theta}^*,e_i) \in I_r^n$ and $(F_{1,\theta}^*,e_1) \sim (F_{2,\theta}^*,e_2).$ Suppose $e_i \in V_{F_{i,\theta}^*}^-$ is nontrivial. Choose $x_i \in C(F_{i,\theta}^*).$ Since $(F_{1,\theta}^*,e_1) \sim (F_{2,\theta}^*,e_2),$ there exists an element $h \in H$ and an apartment $\mathcal{A} \subset \mathcal{B}(H)$ such that

\begin{enumerate}
\item $A(\mathcal{A},F_{1,\theta}^*) = A(\mathcal{A},hF_{2,\theta}^*) \neq \emptyset, $
\item $e_1 = {}^{h}e_2 \textup{ in } V_{F_{1,\theta}^*}^- = V_{hF_{2,\theta}^*}^-. $
\end{enumerate}
As a result of Remark \ref{hremark}, we assume now that $h=1.$ Let $\textbf{S}$ be the maximal $k$-split torus of $\textbf{H}$ corresponding to the apartment $\mathcal{A}.$ Let $\textsf{S}$ denote the maximal $\mathfrak{f}$-split torus inside $\textsf{H}_{x_1}$ corresponding to $\textbf{S}.$ By the previous lemma, we can complete $e_1$ to a normal $\mathfrak{sl}_2(\mathfrak{f})$-triple $(f,h,e_1) \in V_{x_1,-r}^- \times V_{x_1,0}^+ \times V_{x_1,r}^-.$ Suppose that $\overline{\lambda} \in \textbf{X}_{*}^{\mathfrak{f}}(\textsf{H}_{x_1})$ is adapted to this triple. Since any one-parameter subgroup is contained in some maximal $\mathfrak{f}$-split torus, there exists some $h' \in H_{x_1}$ such that ${}^{\overline{h'}}\overline{\lambda} \in \textbf{X}_{*}^{\mathfrak{f}}(\textsf{S}).$ By Lemma \ref{normassoc}, there is an element $h'' \in H_{x_1} \cap H_{x_2}$ such that its image in Aut$_{\mathfrak{f}}(V_{F_{i,\theta}^*}^-)$ coincides with the image of $h'.$ In other words, we have Ad$(h'')\vert_{V_{F_{i,\theta}^*}^-} = \textup{Ad}(h')\vert_{V_{F_{i,\theta}^*}^-}.$ In summary, we have

$$ {}^{h'}e_1 = {}^{h''}e_1 = {}^{h''}e_2 \textup{ in } V_{F_{1,\theta}^*}^- = V_{F_{1,\theta}^*}^- = V_{F_{2,\theta}^*}^-.$$

Now let $\lambda \in \textbf{X}_{*}^{k}(\textbf{S})$ be a lift of ${}^{\overline{h'}}\overline{\lambda}.$ As usual, we have a grading of $\mathfrak{g}$ under the action of $\lambda.$ As in the proof of Lemma \ref{walds1}, we also have the decomposition

$$ \mathfrak{g}_{F_{i,\theta}^*} = \bigoplus_{j}\mathfrak{g}_{F_{i,\theta}^*}(j).$$
By Lemma \ref{assocident}, there is a lift $X \in \mathfrak{p}_{F_{1,\theta}^*}(2) \cap \mathfrak{p}_{F_{2,\theta}^*}(2)$ of ${}^{h''}e_i.$ Note that $X$ lifts ${}^{h'}e_i$ and ${}^{h''^{-1}}X$ lifts $e_i.$ We apply Corollary \ref{sllift} and Lemma \ref{facettoorbit} to conclude

$$ \mathcal{O}_{\theta}(F_{i,\theta}^*, e_i) = \mathcal{O}_{\theta}(F_{i,\theta}^*, {}^{h''}e_i) = {}^{H}X. $$
\end{proof}

Now, in order to obtain a bijection between depth $r$ cosets and nilpotent $H$-orbits, we have to shrink $I_r^n/ \sim.$ We do this by restricting to \emph{noticed} orbits in $\mathfrak{p}_{F_{\theta}^*}/\mathfrak{p}_{F_{\theta}^*}^+.$ 

\begin{defn} Define $I_r^d \subset I_r^n$ to be those pairs $(F_{\theta}^*,e) \in I_r^n$ such that for any $x \in F_{\theta}^*,$ for any normal $\mathfrak{sl}_2(\mathfrak{f})$-triple $(f,h,e) \in V_{x,-r} \times V_{x,0} \times V_{x,r}$ completing $e,$ and for any normal $\mathfrak{sl}_2(k)$-triple $\{Y,H,X\}$ in $\mathfrak{g}$ which lifts $\{f,h,e\},$ we have that $F_{\theta}^* $ is a maximal generalized $(r,\theta)$-facet in $\mathcal{B}_{\theta}(Y,H,X).$ \end{defn}

In order to simplify notation below, we will refer to a generalized $(0,\theta)$-facet as a generalized $\theta$-facet.

\begin{remark} In the case when $r=0,$ we can give an alternate characterization of noticed nilpotent orbits inside a Lie algebra which is analogous to the definition in [\cite{noel}, Definition 2.1]. In particular, if a pair $(F_{\theta}^*,e)$ lies in $I_0^d,$ then $C_{V_{F_{\theta}^*}^+}(e)$ does not contain certain non-central semisimple elements. We give a precise formulation below. \end{remark}

Suppose $(F_{\theta}^*, e) \in I_0^n,$ and let $x \in F_{\theta}^*.$ In this paragraph, we let $\textsf{L}_x$ denote the Lie algebra of $\textsf{G}_x.$ Under some restrictions on the characteristic of $\mathfrak{f},$ it is shown in [\cite{carter}, Proposition 5.7.4] that if $\{ f, h, e\}$ is an $\mathfrak{sl}_2(\mathfrak{f})$-triple in $\textsf{L}_x$ completing $e,$ then $C_{\textsf{L}_x}(e)$ is a subalgebra of $\textsf{L}_x$ of the form $\mathfrak{c} \oplus \mathfrak{u},$ where $\mathfrak{c}$ is a reductive subalgebra which centralizes the triple $\{f,h,e\}$ and $\mathfrak{u}$ is a nilpotent ideal in $C_{\textsf{L}_x}(e).$ In particular, we may consider the $\mathfrak{f}$-rank of $\mathfrak{c}$ in the sense of [\cite{borel}, 21.1].  By [\cite{carter}, Proposition 5.9.3], if $\{f',h',e\}$ is another $\mathfrak{sl}_2(\mathfrak{f})$-triple in $\textsf{L}_x$ completing $e,$ and if $\mathfrak{c}'$ is defined relative to $\{f', h', e\},$ then $\mathfrak{c}$ and $\mathfrak{c}'$ are conjugate by an element of $C_{\textsf{G}_x(\mathfrak{f})}(e).$ In particular, the $\mathfrak{f}$-rank of $\mathfrak{c}$ and $\mathfrak{c}'$ are the same. 	
	
	In the following proposition, we take the $\mathfrak{f}$-rank of $C_{\textsf{L}_x}(e)$ to mean the $\mathfrak{f}$-rank of the centralizer $C_{\textsf{L}_x}( \textup{im } \phi),$ where $\phi : \mathfrak{sl}_2(\mathfrak{f}) \rightarrow V_{F_{\theta}^*}$ is an $\mathfrak{f}$-map whose image contains $e.$
	
\begin{prop} Let $(F_{\theta}^*,e) \in I_0^n.$ Suppose Hypothesis \ref{hyp1} holds, $\mathfrak{f}$ is finite, and $p > 3(j-1),$ where $j$ is the Coxeter number as defined in [\cite{carter}, Section 1.9]. The following are equivalent: 
\begin{enumerate} 
\item $C_{V_{F_{\theta}^*}^+}(e) \cap [V_{F_{\theta}^*}, V_{F_{\theta}^*}]$ has $\mathfrak{f}$-rank equal to zero.
\item $(F_{\theta}^*, e)$ lies in $I_0^d.$
\end{enumerate}
\end{prop}

\begin{proof} (1) $\Rightarrow$ (2) : Suppose there is some $x \in F_{\theta}^*$ for which there exists a lift $\{Y,H,X\} \subset V_{x,0}$ of an $\mathfrak{sl}_2(\mathfrak{f})$-triple in $\overline{\mathfrak{g}}_x$ completing $e$ such that $F_{\theta}^*$ is not maximal in $\mathcal{B}_{\theta}(Y,H,X).$ Then, there exists some generalized $\theta$-facet $C_{\theta}^*$ in $\mathcal{B}_{\theta}(Y,H,X)$ containing $F_{\theta}^*$ in its closure. In particular, by [\cite{d}, Cor. 3.2.19], we may identify $V_{C_{\theta}^*}$ with a Levi subalgebra of the parabolic subalgebra $\mathfrak{g}_{C_{\theta}^*} / \mathfrak{g}_{F_{\theta}^*}^+,$ which are both $\theta$-stable by the proof of Proposition \ref{thstpa}. Let $\mathcal{A}$ be an apartment in $\mathcal{B}(H)$ containing $x,$ and let $\textbf{S}$ be the maximal $k$-split torus in $\textbf{H}$ corresponding to $\mathcal{A}.$ By [\cite{kim}, Proposition 5.4] and [\cite{prasad-yu}, Theorem 1.9], we can choose some $\theta$-stable maximal $k$-split torus $\textbf{S}'$ containing $\textbf{S}$ such that $\mathcal{A}(\textbf{S}, k) \subset \mathcal{A}(\textbf{S}', k).$ Let $\textsf{S}$ denote the maximal $\mathfrak{f}$-split torus in $\textsf{H}_{x}$ corresponding to $\textbf{S}.$ In the notation of [\cite{helminck-wang}, Lemma 3.3] (letting $k=\mathfrak{f}$ and $\lambda = \mu$), there exists some $\mu \in \textbf{X}_*^{\mathfrak{f}}(\textsf{S})$ such that $V_{C_{\theta}^*} = C_{V_{F_{\theta}^*}}(d\mu).$ In particular, since $C_{\theta}^* \subset \mathcal{B}_{\theta}(Y,H,X),$ we have $e \in V_{C_{\theta}^*}.$ Thus, there is a semisimple element in $V_{F_{\theta}^*}^+$ which centralizes $e$ and lies in the Lie algebra of some $\mathfrak{f}$-split torus in $\textsf{H}_x.$ Moreover, this element does not lie in the center of $V_{F_{\theta}^*}$ since $V_{C_{\theta}^*}$ is properly contained in $V_{F_{\theta}^*}.$ \newline

(2) $\Rightarrow$ (1): We prove the contrapositive. Let $x \in F_{\theta}^*.$ Suppose there is some element $s \in C_{V_{F_{\theta}^*}^+}(e) \cap [V_{F_{\theta}^*}, V_{F_{\theta}^*}],$ which lies in the Lie algebra of some $\mathfrak{f}$-split torus in $\textsf{H}_x.$ By [\cite{borel}, 8.15(d)], we may assume there is some one-parameter subgroup $\overline{\mu} \in \textbf{X}_*^{\mathfrak{f}}(\textsf{H}_x)$ for which $s \in \textup{im}(d\overline{\mu}).$ By [\cite{kim}, Proposition 3.4.1] or [\cite{helminck-wang}, Proposition 2.3], there is  some $\theta$-stable maximal $\mathfrak{f}$-split torus $\textsf{T}$ in $\textsf{G}_x$ such that $\overline{\mu} \in \textbf{X}_*^{\mathfrak{f}}(\textsf{T}).$ Since $\overline{\mu}$ is $\theta$-fixed, $\overline{\mu}$ determines a $\theta$-stable $\mathfrak{f}$-parabolic subgroup $\textsf{P}$ of $\textsf{G}_x$ which contains a $\theta$-stable $\mathfrak{f}$-Levi subgroup $\textsf{M}:=C_{\textsf{G}_x}(\overline{\mu}).$ Let $\mathfrak{m}$ denote the Lie algebra of $\textsf{M},$ which lies inside the Lie algebra of $\textsf{P}.$ By [\cite{kim}, Proposition 5.4], we can choose a $\theta$-stable maximal $k$-split torus $\textbf{T}$ such that $\mathcal{A}(\textbf{T}, k)$ contains $x$ and $\mathcal{A}(\textbf{T}, k)^{\theta}$ is an affine space whose dimension is equal to the dimension of a maximal $k$-split torus of $\textbf{H}.$ Using the embedding from [\cite{prasad-yu}, Theorem 1.9], we must have that $\mathcal{A}(\textbf{T}, k)^{\theta}$ is an apartment of $\mathcal{B}(H)$ containing $x.$ The image of $\textbf{T}$ in $\textsf{G}_x$ is maximal $\mathfrak{f}$-split torus $\textsf{S}$ for which $\textsf{S}^{\theta}$ is a maximal $\mathfrak{f}$-split torus in $\textsf{H}_x.$ Thus, by conjugacy of maximal $\mathfrak{f}$-split tori in $\textsf{H}_x,$ we may assume that the image of $\textbf{T}$ in $\textsf{G}_x$ is $\textsf{T}.$ Using the identification of $\textbf{X}_*^k(\textbf{T})$ with $\textbf{X}_*^{\mathfrak{f}}(\textsf{T}),$ let $\mu$ be a lift of $\overline{\mu}$ in $\textbf{X}_*^k(\textbf{T}) \otimes \mathbb{R},$ let $C$ be the first $\theta$-facet of $\mathcal{A}(\textbf{T}, k)$ encountered when moving from $x$ to $x + \mu,$ and let $C_{\theta}^*$ be the generalized $\theta$-facet containing $C.$  Then, the vector space of $\mathfrak{f}$-rational points of the Lie algebra of $\textsf{P}$ is of the form $\mathfrak{g}_{C_{\theta}^*} / \mathfrak{g}_{F_{\theta}^*}^+.$ By our choice of $\overline{\mu},$ we have $e \in \mathfrak{m} \subset \mathfrak{g}_{C_{\theta}^*} / \mathfrak{g}_{F_{\theta}^*}.$ 

	Let $e = X' + \mathfrak{g}_{F_{\theta}^*}^+,$ for some representative $X'$ which lies in $\mathfrak{g}_{C_{\theta}^*}.$ If $e$ is trivial, then we complete $e$ to the trivial $\mathfrak{sl}_2(\mathfrak{f})$-triple inside $V_{C_{\theta}^*}$ and let $\{Y,H,X\}$ be the trivial $\mathfrak{sl}_2(k)$-triple in $\mathfrak{g}_{C_{\theta}^*}$ lifting $\{f,h,e\}.$ This shows that $F_{\theta}^*$ cannot be maximal in $\mathcal{B}_{\theta}(Y,H,X).$ 
	
	Suppose $e$ is nontrivial. We want to show that there exists some lift $X$ of $e$ in $\mathfrak{g}_{C_{\theta}^*} \cap \mathcal{N}.$ Since $e$ is degenerate in $V_{F_{\theta}^*},$ using the same argument as that given in the last paragraph of Lemma \ref{deglemma}, there exists some $X \in \mathfrak{g}_{C_{\theta}^*} \cap \mathcal{N}$ whose image in $\mathfrak{g}_{C_{\theta}^*} / \mathfrak{g}_{F_{\theta}^*}^+$ is $e.$ Since $e$ is a nonzero element of $\mathfrak{m},$ it does not lie in the kernel $\mathfrak{g}_{C_{\theta}^*}^+ / \mathfrak{g}_{F_{\theta}^*}^+$ of the projection from $\mathfrak{g}_{C_{\theta}^*} / \mathfrak{g}_{F_{\theta}^*}^+$ onto $\mathfrak{m}.$ In particular, we have that $X \notin \mathfrak{g}_{C_{\theta}^*}^+.$ Thus, by Hypothesis \ref{hyp1}, there exists some $\mathfrak{sl}_2(\mathfrak{f})$-triple $\{ f, h, e\}$ in $V_{C_{\theta}^*}$ completing $e,$ and by Corollary \ref{sllift}, there exists a lift $\{ Y, H, X\}$ of $\{ f, h, e\}$ which lies in $\mathfrak{g}_{C_{\theta}^*}.$ This shows that $F_{\theta}^*$ cannot be maximal in $\mathcal{B}_{\theta}(Y,H,X).$ 

\end{proof}

\begin{lemma} \label{distinguished} Suppose all hypotheses in Section 5 hold. If $(F_{\theta}^*,e) \in I_r^n$ and $e$ is nontrivial, then $(F_{\theta}^*,e) \in I_r^d$ if and only if there exists some $x \in F_{\theta}^*,$ a normal $\mathfrak{sl}_2(\mathfrak{f})$-triple $(f,h,e) \in V_{x,-r}^- \times V_{x,0}^+ \times V_{x,r}^-$ completing $e,$ and a normal $\mathfrak{sl}_2(k)$-triple $\{Y,H,X\}$ in $\mathfrak{g}$ lifting $\{f,h,e\}$ such that $F_{\theta}^*$ is a maximal generalized $(r,\theta)$-facet of $\mathcal{B}_{\theta}(Y,H,X).$

\end{lemma}

\begin{proof} ``$\Rightarrow:$" This is a consequence of the definition. \newline
``$\Leftarrow:$" Let $x \in F_{\theta}^*, \{f,h,e\},$ and $\{Y,H,X\}$ be data satisfying the hypotheses of the lemma such that $F_{\theta}^*$ is a maximal generalized $(r,\theta)$-facet of $\mathcal{B}_{\theta}(Y,H,X).$
	Suppose there is some $x' \in F_{\theta}^*,$ a normal $\mathfrak{sl}_2(\mathfrak{f})$-triple $(f',h',e) \in V_{x',-r}^- \times V_{x',0}^+ \times V_{x',r}^-$ completing $e,$ and a normal $\mathfrak{sl}_2(k)$-triple $\{Y',H',X'\}$ in $\mathfrak{g}$ lifting $\{f',h',e\}.$ We argue by contradiction, and suppose $F_{\theta}^*$ is not maximal in $\mathcal{B}_{\theta}(Y',H',X').$ Since $x,x' \in F_{\theta}^*,$ we have $\mathfrak{g}_{x',-r} = \mathfrak{g}_{x,-r}$ and $\mathfrak{g}_{x,r} = \mathfrak{g}_{x',r}.$ Since $[\mathfrak{g}_{x',-r}, \mathfrak{g}_{x',r}] \subset \mathfrak{g}_{x,0},$ the lift $(Y',H',X')$ lies in $\mathfrak{p}_{x,-r} \times \mathfrak{h}_{x,0} \times \mathfrak{p}_{x,r},$ so we may assume $x'=x.$ By Lemma \ref{facettoorbit}, we have that ${}^{H}X = \mathcal{O}_{\theta}(F_{\theta}^*,e) = {}^{H}X'.$ Combining this with Corollary \ref{intersection}, $X$ is $H_x^+$-conjugate to $X'.$ Thus, by replacing $\{Y'H',X'\}$ with some $H_x^+$-conjugate, we may assume $X=X'.$ By the last line of Hypothesis \ref{hyp4}, there exists some element $h \in C_{H}(X)$ with ${}^{h}Y=Y'$ and ${}^{h}H=H'.$ In particular, $Y \in {}^{h^{-1}}\mathfrak{g}_{x,-r} = \mathfrak{g}_{h^{-1}x,-r}$ and $X \in {}^{h^{-1}}\mathfrak{g}_{x,r} = \mathfrak{g}_{h^{-1}x,r},$ so $h^{-1}\mathcal{B}_{\theta}(Y',H',X) = \mathcal{B}_{\theta}(Y,H,X).$ This shows that $\mathcal{B}_{\theta}(Y',H',X)$ and $\mathcal{B}_{\theta}(Y,H,X)$ have the same dimension. However, we assumed that $F_{\theta}^*$ was not maximal in $\mathcal{B}(Y',H',X),$ so $h^{-1}F_{\theta}^*$ is not maximal in $\mathcal{B}_{\theta}(Y,H,X),$ which is a contradiction since the action of $H$ preserves dimension.

\end{proof}

\begin{remark} Suppose $(F_{1,\theta}^* , e_1) \sim (F_{2,\theta}^*, e_2).$ As a consequence of the proof of Lemma \ref{distinguished}, we have $(F_{1,\theta}^*, e_1) \in I_r^d$ if and only if $(F_{2,\theta}^*,e_2) \in I_r^d.$
\end{remark}

\begin{thm} Suppose all hypotheses of Section 5 hold. There is a bijective correspondence between $I_r^d/\sim$ and $\mathcal{O}_{\theta}(0)$ given by the map that sends $(F_{\theta}^*,e)$ to $\mathcal{O}_{\theta}(F_{\theta}^*,e).$\end{thm}

\begin{proof} We have already shown that the map which sends $(F_{\theta}^*,e)$ to $\mathcal{O}_{\theta}(F_{\theta}^*,e)$ is well-defined in Lemma \ref{welldef}. We will first show that the map restricted to $I_r^d$ is injective. Suppose $\mathcal{O}_{\theta}(F_{1,\theta}^*,e_1) = \mathcal{O}_{\theta}(F_{2,\theta}^*,e_2).$ Recall from Corollary \ref{Cnonempty} that for $F_{\theta}^* \in \mathcal{F}(r),$ the set $C(F_{\theta}^*)$ is nonempty. Choose $x_i \in C(F_{i,\theta}^*)$ and complete $e_i$ to an $\mathfrak{sl}_2(\mathfrak{f})$-triple $(f_i,h_i,e_i)$ in $V_{F_{i,\theta}^*}^- \times V_{F_{i,\theta}^*}^+ \times V_{F_{i,\theta}^*}^-.$ By Corollary \ref{sllift}, we can lift these $\mathfrak{sl}_2(\mathfrak{f})$-triples to $\mathfrak{sl}_2(k)$-triples $\{Y_i,H_i,X_i\}$ in $\mathfrak{g}.$ By Lemma \ref{facettoorbit}, $\mathcal{O}_{\theta}(F_{i,\theta}^*,e_i) = {}^{H}X_i.$ Since $\mathcal{O}_{\theta}(F_{1,\theta}^*,e_1) = \mathcal{O}_{\theta}(F_{2,\theta}^*,e_2),$ we thus have ${}^{H}X_1 = {}^{H}X_2,$ so by Hypothesis \ref{hyp4}, there exists an $h \in H$ such that $\{Y_1,H_1,X_1\} = \{{}^{h}Y_2,{}^{h}H_2,{}^{h}X_2\}.$ By Lemma \ref{facettoorbit}, we have $F_{1,\theta}^* \subset \mathcal{B}_{\theta}(Y_1,H_1,X_1),$ and $hF_{2,\theta}^* \subset \mathcal{B}_{\theta}(Y_1,H_1,X_1),$ so since $(F_{i,\theta}^*,e_i) \in I_r^d,$ $F_{1,\theta}^*$ and $hF_{2,\theta}^*$ are maximal generalized $(r,\theta)$-facets of $\mathcal{B}_{\theta}(Y_1,H_1,X_1).$ Thus, by Lemma \ref{maxassoc}, $F_{1,\theta}^*$ and $hF_{2,\theta}^*$ are strongly $r$-associated. In particular, there exists an apartment $\mathcal{A} \subset \mathcal{B}(H)$ such that \begin{eqnarray*}
 A(\mathcal{A}, F_{1,\theta}^*) = A(\mathcal{A}, hF_{2,\theta}^*) \neq \emptyset. \\
 \end{eqnarray*}

As $X_1$ has image $e_1$ in $V_{F_{1,\theta}^*}$ and $X_1$ has image ${}^{h}e_2$ in $V_{hF_{2,\theta}^*},$ the element $X_1$ lies in $\mathfrak{p}_{F_{1,\theta}^*} \cap \mathfrak{p}_{hF_{2,\theta}^*},$ and

$$e_1 = {}^{h}e_2 \textup{ in } V_{F_{1,\theta}^*} = V_{hF_{2,\theta}^*}$$ 
Thus, in particular, $(F_{1,\theta}^*,e_1) \sim (F_{2,\theta}^*,e_2),$ i.e the map is injective. 
	
	For surjectivity, first let $\{ 0 \}$ be the trivial orbit. Let $F_{\theta}^*$ be an open generalized $(r,\theta)$-facet, and let $e$ be the trivial element of $V_{F_{\theta}^*}^-.$ Then, $(F_{\theta}^*,e)$ maps to $\{0\}.$ Now, suppose $\mathcal{O}$ is nontrivial and let $X \in \mathcal{O}.$ Complete $X$ to an $\mathfrak{sl}_2(k)$-triple $\{Y,H,X\}$ and choose a maximal generalized $(r,\theta)$-facet $F_{\theta}^* \subset \mathcal{B}_{\theta}(Y,H,X).$ Let $e$ denote the image of $X$ in $V_{F_{\theta}^*}.$ Then by Lemma \ref{facettoorbit} (2), we have $\mathcal{O}_{\theta}(F_{\theta}^*,e) = {}^{H}X.$

\end{proof}

\section{Appendix A}

\subsection{Calculation of nilpotent orbits associated to the pair $(\textbf{SL}_3, \textbf{PGL}_2).$}

Recall that we have an involution $\theta: \textbf{SL}_3 \rightarrow \textbf{SL}_3$ defined by $A \mapsto J(A^t)^{-1}J,$ where $ J= \left(\begin{array}{ccc}
0 & 0 & 1 \\
0 & 1 & 0   \\
1 & 0 & 0 
\end{array}\right). $ The Lie algebra of $\textbf{SL}_3$, denoted $\mathfrak{sl}_3,$ has a decomposition on the level of $k$-points which is given by

$$  \mathfrak{sl}_3(k) =  \left(\begin{array}{ccc}
a & b & 0 \\
c & 0 & -b   \\
0 & -c & -a 
\end{array}\right) \oplus  \left(\begin{array}{ccc}
x & y & s \\
z & -2x & y   \\
u & z & x 
\end{array}\right), $$
with $a,b,c,s,u,x,y,z \in k.$
Let $X =  \left(\begin{array}{ccc}
x & y & s \\
z & -2x & y   \\
u & z & x 
\end{array}\right)$ be nilpotent. We would like to find a nicer representative for $X$ up to $H = \textbf{PGL}_2(k)$-conjugacy. 

	By Remark \ref{normal}, since $X$ is nilpotent, there is a one-parameter subgroup $\lambda: \textbf{GL}_1 \rightarrow \textbf{H}$ such that ${}^{\lambda(t)}X = t^2X.$ On the other hand, since the image of $\lambda$ lies in a maximal $k$-split torus of $\textbf{H},$ there is some $h \in \textbf{H}$ and $n \in \mathbb{Z}$ such that $({}^{h}\lambda)(t) =  \left(\begin{array}{ccc}
t^n & 0 & 0 \\
0 & 1 & 0   \\
0 & 0 & t^{-n} 
\end{array}\right).$ Thus, we have $ {}^{{}^{({}^{h}\lambda)(t)}}({}^{h}X) = t^2({}^{h}X). $
Letting $Z = {}^{h}X =  \left(\begin{array}{ccc}
u & v & w \\
r & -2u & v   \\
s & r & u 
\end{array}\right) ,$ we have 
$$  \left(\begin{array}{ccc}
u & t^nv & t^{2n}w \\
t^{-n}r & -2u & t^nv   \\
t^{-2n}s & t^{-n}r & u 
\end{array}\right) =  \left(\begin{array}{ccc}
t^2u & t^2v & t^2w \\
t^2r & -2ut^2 & t^2v   \\
t^2s & t^2r & t^2u 
\end{array}\right).$$ Assume $X$ is nontrivial. If $v \neq 0,$ we must have $n=2,$ but this forces all other entries to be zero. If $w \neq 0,$ then, $n$ must be equal to 1, but this forces all other entries to be zero. Thus, the nilpotent $H$-conjugacy classes in the $(-1)$-eigenspace of $\mathfrak{sl}_3(k)$ under $d\theta$ lie in one of the three subsets of $\mathfrak{sl}_3(k):$

$$ \{ \textup{triv}\}, \left\{  \left(\begin{array}{ccc}
0 & v & 0 \\
0 & 0 & v   \\
0 & 0 & 0 
\end{array}\right)  \mid v \in \mathbb{Q}_p^{\times} \right\}, \textup{and} \left\{ \left(\begin{array}{ccc}
0 & 0 & w \\
0 & 0 & 0   \\
0 & 0 & 0 
\end{array}\right) \mid w \in \mathbb{Q}_p^{\times} \right\},$$
where $\textup{triv}$ denotes the trivial orbit.

We note that all matrices of the form  $\left(\begin{array}{ccc}
0 & v & 0 \\
0 & 0 & v   \\
0 & 0 & 0 
\end{array}\right)$ are $\textbf{H}(k)$-conjugate by the diagonal maximal $\mathbb{Q}_p$-split torus. Matrices of the form $\left(\begin{array}{ccc}
0 & 0 & w \\
0 & 0 & 0   \\
0 & 0 & 0 
\end{array}\right)$ split up into four conjugacy classes which are parametrized by $(\mathbb{Q}_p)^{\times}/(\mathbb{Q}_p^{\times})^2.$

\end{document}